%% file: FZZ.tex
\pdfoutput=1
\documentclass[11pt]{article}
\usepackage{xcolor}
\usepackage{graphicx}
\usepackage{amsmath,amsfonts,amssymb,graphics,amsthm}
\usepackage{hyperref}
\usepackage{comment}
\usepackage{tabularx}
\usepackage[protrusion=true,expansion=true]{microtype}
\usepackage{enumerate}
\usepackage{bbm}
\usepackage{mathrsfs}
\usepackage[margin=1in]{geometry}

\hypersetup{
    colorlinks=false,
    linktocpage,
    }

\numberwithin{equation}{section}

\newtheorem{theorem}{Theorem}[section]
\newtheorem{corollary}[theorem]{Corollary}
\newtheorem{lemma}[theorem]{Lemma}
\newtheorem{proposition}[theorem]{Proposition}
\newtheorem{prop}[theorem]{Proposition} 
\newtheorem{conjecture}[theorem]{Conjecture}

\newtheorem{remark}[theorem]{Remark}
\newtheorem{definition}[theorem]{Definition}

\theoremstyle{remark}

\def\@rst #1 #2other{#1}
\newcommand\MR[1]{\relax\ifhmode\unskip\spacefactor3000 \space\fi
  \MRhref{\expandafter\@rst #1 other}{#1}}
\newcommand{\MRhref}[2]{\href{http://www.ams.org/mathscinet-getitem?mr=#1}{MR#2}}

\def\MR#1{\href{http://www.ams.org/mathscinet-getitem?mr=#1}{MR#1}}

\newcommand{\C}{\mathbbm{C}}
\newcommand{\D}{\mathbbm{D}}
\newcommand{\E}{\mathbbm{E}}

\newcommand{\Z}{\mathbbm{Z}}
\newcommand{\R}{\mathbbm{R}}
\renewcommand{\P}{\mathbbm{P}}
\newcommand{\bbH}{\mathbbm{H}}

\newcommand{\eps}{\varepsilon}
\newcommand{\1}{\mathbf{1}}

\newcommand{\disk}{\mathrm{disk}}

\newcommand{\sm}{\mathsf{m}}

\newcommand{\cMtwo}{\mathcal{M}_{0,2}^\mathrm{disk}}

\newcommand{\LF}{\mathrm{LF}}
\newcommand{\QD}{\mathrm{QD}}

\let\Re\undefined
\DeclareMathOperator{\Re}{Re}
\let\Im\undefined
\DeclareMathOperator{\Im}{Im}

\DeclareMathOperator{\diam}{diam}

\DeclareMathOperator{\Cov}{Cov}
\DeclareMathOperator{\Var}{Var}
\DeclareMathOperator{\SLE}{SLE}

\def\cS{\mathcal{S}}

\def\cP{\mathcal{P}}

\def\cM{\mathcal{M}}
\def\cL{\mathcal{L}}
\def\cK{\mathcal{K}}

\def\cH{\mathcal{H}}

\def\cF{\mathcal{F}}

\def\cD{\mathcal{D}}
\def\cC{\mathcal{C}}

\def\cA{\mathcal{A}}

\def\alb#1\ale{\begin{align*}#1\end{align*}}
\def\allb#1\alle{\begin{align}#1\end{align}}

\newcommand{\aryb}{\begin{eqnarray*}}
\newcommand{\arye}{\end{eqnarray*}}
\def\alb#1\ale{\begin{align*}#1\end{align*}}
\newcommand{\eqb}{\begin{equation}}
\newcommand{\eqe}{\end{equation}}
\newcommand{\eqbn}{\begin{equation*}}
\newcommand{\eqen}{\end{equation*}}

\newcommand{\BB}{\mathbbm}
\newcommand{\ol}{\overline}
\newcommand{\ul}{\underline}
\newcommand{\op}{\operatorname}

\newcommand{\rta}{\rightarrow}

\newcommand{\wt}{\widetilde}
\newcommand{\wh}{\widehat}

\newcommand{\bdy}{\partial}

\let\originalleft\left
\let\originalright\right
\renewcommand{\left}{\mathopen{}\mathclose\bgroup\originalleft}
\renewcommand{\right}{\aftergroup\egroup\originalright}

\DeclareMathAlphabet{\mathpzc}{OT1}{pzc}{m}{it}

\begin{document}

\title{FZZ formula of boundary Liouville CFT via conformal welding}
\author{
\begin{tabular}{c}Morris Ang\\[-5pt]\small MIT\end{tabular}\; 
\begin{tabular}{c}Guillaume Remy\\[-5pt]\small Columbia University\end{tabular}\;
\begin{tabular}{c}Xin Sun\\[-5pt]\small University of Pennsylvania\end{tabular}
} 
\date{  }

\maketitle

\input{intro.tex}

\input{pre.tex}
\input{sec3.tex}
\input{sec4.tex}
\input{sec5.tex}

\bibliographystyle{hmralphaabbrv}
%\addcontentsline{toc}{section}{References}
\bibliography{cibib}

\end{document}

%% file: intro.tex
\begin{abstract}
Liouville Conformal Field Theory (LCFT) on the disk describes  the conformal factor of the quantum disk, 
which is the natural  random surface in Liouville quantum gravity with disk topology. 
Fateev, Zamolodchikov and Zamolodchikov (2000) 
proposed an explicit expression, the so-called FZZ formula,  for the one-point bulk structure constant for LCFT on the disk.
In this paper we give a proof of the FZZ formula in the probabilistic framework of LCFT,
which represents the first step towards rigorously solving boundary LCFT using conformal bootstrap.
In contrast to previous works,
our proof is based on conformal welding of quantum disks and the mating-of-trees theory for Liouville quantum gravity.
As a byproduct of our proof, we also obtain the exact value of the variance for the  Brownian motion in the mating-of-trees theory. 
Our paper is an essential part of an ongoing program proving  integrability results for Schramm-Loewner evolutions, LCFT, and in the mating-of-trees theory.
\end{abstract}

\section{Introduction}
Liouville quantum gravity (LQG) first appeared in theoretical physics in A. Polyakov's seminal work \cite{Polyakov1981} where he proposed a theory of summation over the space of Riemannian metrics on a given two dimensional surface. The fundamental building block of his framework is the Liouville conformal field theory (LCFT), 
which  describes the law of the conformal factor of the metric tensor of a surface of fixed complex structure.
LCFT  was first made rigorous in probability theory  in the case of the Riemann sphere in \cite{dkrv-lqg-sphere}, and then in the case of a simply connected domain with boundary in \cite{hrv-disk}; see also \cite{LQG_tori, Remy_annulus, LQG_higher_genus} for the case of other topologies.

On surfaces without boundary, solving Liouville theory amounts to computing the three-point function on the sphere - which is given by the DOZZ formula proposed in \cite{DOZZ1,DOZZ2} - and arguing that correlation functions of higher order or in higher genus can be obtained from it using the conformal bootstrap method of \cite{BPZ1984}. 
Recently,  two major breakthroughs have been achieved, namely the rigorous proof of the DOZZ formula \cite{DOZZ_proof} and of the conformal bootstrap  on the sphere \cite{GKRV_bootstrap}. 
A similar program can be pursued for surfaces with boundary, where the most basic correlation function is the bulk one-point function on the disk with expression given by the Fateev-Zamolodchikov-Zamolodchikov (FZZ) formula proposed in \cite{FZZ}. 
In this  paper we will prove the FZZ formula, which represents the first step towards rigorously solving boundary LCFT.

Our approach is completely different from the one used in~\cite{DOZZ_proof, remy-fb-formula, RZ_boundary} 
which is based on the BPZ equations and on the operator product expansion of~\cite{BPZ1984}.
As explained in Section~\ref{subsec:LCFT}, that approach has essential obstructions to proving  the FZZ formula.
Instead, we rely on the rich interplay between LCFT and the  random geometry corresponding to LQG. 
In particular, we use the idea of the \emph{quantum zipper}, which says that the conformal welding of two LQG type random surfaces 
gives a LQG type surface decorated with a Schramm Loewner evolution (SLE). Building on the original work of~\cite{shef-zipper,wedges} 
and the recent work of the first and the third authors with Holden~\cite{ahs-disk-welding,AHS-SLE-integrability}, we prove a new quantum zipper result and use it to obtain  the FZZ formula. 
As an intermediate step in our proof, we also obtain the exact value of the variance of the Brownian motion in the mating-of-trees theory by Duplantier, Miller and Sheffield~\cite{wedges}.

Besides its intrinsic interest and its relevance to  conformal bootstrap, 
the FZZ formula  yields integrability results on Gaussian multiple chaos on the unit disk or upper half plane; see Section~\ref{subsub:GMC-disk}.  
Moreover, it is a crucial input to the paper~\cite{AS-CLE} of the first and the third authors on the integrability of conformal loop ensemble on the sphere.
We will discuss these aspects and related ongoing projects and open questions in Section~\ref{subsec:outlook}, after stating our main result in Section~\ref{subsec:LCFT} and summarizing  the proof strategy in Section~\ref{subsec:sketch}.

\subsection{Boundary Liouville  conformal field theory and the FZZ formula}\label{subsec:LCFT}
 In the physics literature LCFT is defined by a formal path integral. We work on a simply connected domain with boundary, which by conformal invariance can equivalently be the upper-half plane $\mathbb{H}$ or the unit disk $\mathbb{D}$. For almost all of this paper we will work with $\mathbb{H} = \{ z \in \mathbb{C} \vert \mathrm{Im}(z) >0 \}$ with boundary given by the real line $\mathbb{R}$. The most basic observable of Liouville theory is the correlation function of $N$ marked points in the bulk $z_i \in \mathbb{H}$ with associated weights $\alpha_i \in \mathbb{R}$ and $M$ marked points on the boundary $s_j \in \mathbb{R}$ with associated weights $\beta_j \in \mathbb{R}$. The physics path integral definition of this correlation function is then
 \begin{align}\label{eq:path_integral}
 \left \langle \prod_{i=1}^N e^{\alpha_i \phi(z_i)} \prod_{j=1}^M e^{\frac{\beta_j}{2} \phi(s_j)} \right \rangle = \int_{X: \mathbb{H} \rightarrow \mathbb{R}} DX e^{-S_{L}(X)} \left(  \prod_{i=1}^N e^{\alpha_i X(z_i)} \prod_{j=1}^M e^{\frac{\beta_j}{2} X(s_j)}  \right),
 \end{align}
 where $DX$ is a formal uniform measure over the space of all maps $X$ from the domain $\mathbb{H}$ to $\mathbb{R}$ and $S_L(X)$ is the 
 so-called  Liouville action which has expression given by:
 \begin{align}\label{liouville_action}
 S_L(X) &= \frac{1}{4 \pi} \int_{\mathbb{H}} \left( \vert \partial^g X \vert^2 + Q R_g X + 4 \pi \mu e^{\gamma X} \right) g(x) d^2 x + \frac{1}{2 \pi} \int_{ \mathbb{R}}  \left( Q K_g X + \mu_B  e^{\frac{\gamma}{2} X} \right) g(x)^{1/2} dx.
 \end{align}
Here the background metric is $g = g(x)dx^2 $, $R_g$ and $K_g$ are respectively the Ricci and geodesic curvatures on $\mathbb{H}$ and $\mathbb{R}$,
$\gamma\in (0,2)$ is the coupling parameter for LCFT, $Q = \frac{\gamma}{2} + \frac{2}{\gamma}$ is called the  \emph{background charge},
and $\mu , \mu_B > 0 $ are called \emph{cosmological constant}s. They tune respectively the interaction strength of the Liouville potentials $e^{\gamma X}$ and $e^{\frac{\gamma}{2} X}$.
Although the definition of $ S_L(X)$ depends on a choice of background metric  $g$, the correlation functions depend trivially on this choice thanks to the Weyl anomaly proved in \cite{hrv-disk}.

As a conformal field theory, it is well known that the (bulk) one-point correlation function $ \left \langle  e^{\alpha \phi(z)}  \right \rangle $ of LCFT 
must have the following form
 \begin{equation}
 \left \langle  e^{\alpha \phi(z)}  \right \rangle = \frac{U(\alpha)}{\left|\Im z\right|^{2 \Delta_{\alpha}}} \quad \textrm{ for }z\in \bbH,
 \end{equation}
 where $U(\alpha)$ is called the \emph{structure constant} and $\Delta_{\alpha} = \frac{\alpha}{2}(Q - \frac{\alpha}{2})$ is called the \emph{scaling dimension}. 
In~\cite{FZZ}, the following exact formula for $U(\alpha)$  was proposed
 \begin{equation}
 U_{\mathrm{FZZ}}(\alpha) :=  \frac{4}{\gamma}     2^{-\frac{\alpha^2}2} \left( \frac{ \pi \mu}{ 2^{\gamma \alpha }}  \frac{\Gamma(\frac{\gamma^2}{4})}{\Gamma(1 -\frac{\gamma^2}{4})} \right)^{\frac{Q - \alpha}{\gamma} } \Gamma( \frac{\gamma \alpha}{2} - \frac{\gamma^2}{4} )  \Gamma( \frac{2\alpha}{\gamma} - \frac{4}{\gamma^2} -1)  \cos ((\alpha - Q) \pi s),
 \end{equation}
 where the parameter $s$ is related to the ratio of cosmological constants $\frac{\mu_B}{\sqrt{\mu}} $ through the relation:
 \begin{equation}\label{eq:def_s}
 \cos \frac{\pi \gamma s}{2} = \frac{\mu_B}{\sqrt{\mu}} \sqrt{ \sin \frac{\pi \gamma^2}{4}}, \quad \text{with} \quad \begin{cases} s \in[0, \frac{1}{\gamma}), \quad \text{when} \quad  \frac{\mu^2_B}{\mu} \sin \frac{\pi \gamma^2}{4} \leq 1,\\  s \in i [0, +\infty), \quad \text{when} \quad \frac{\mu^2_B}{\mu} \sin \frac{\pi \gamma^2}{4} \geq 1. \end{cases}
 \end{equation}
 Notice $U_{\mathrm{FZZ}}(\alpha)$ depends non-trivially on $\mu, \mu_B$ only through the ratio  $\frac{\mu_B}{\sqrt{\mu}}$, 
 the dependence being encoded in the intricate relation \eqref{eq:def_s} defining the parameter $s$. 
 
The main result of our paper is a proof that $U(\alpha)=U_{\mathrm{FZZ}}(\alpha)$ where $U(\alpha)$ is defined in the rigorous probabilistic framework.
 Let us now outline the procedure of \cite{dkrv-lqg-sphere} adapted to the case of $\mathbb{H}$ in \cite{hrv-disk} 
 that allows to give a rigorous meaning to \eqref{eq:path_integral} and thus  for $U(\alpha)$.
 All definitions will be precisely restated in Section \ref{subsec:pre-LCFT}. The first step is to interpret the $DX$ of \eqref{eq:path_integral} combined with the gradient squared term $ \vert \partial^g X \vert^2 $ of $S_L(X)$ as giving the law of the Gaussian Free Field (GFF). Concretely let $P_\bbH$ be the probability measure corresponding to the free-boundary GFF on $\bbH$ normalized to have average zero on the upper half unit circle $\partial \D \cap \bbH$. Define now the infinite measure $\LF_\bbH(d\phi)$ obtained by sampling $(h, \mathbf c)$ 
 according to $P_\bbH \times [e^{-Qc} \, dc]$ and setting $\phi(z) = h(z) - 2Q \log |z|_+ + \mathbf c$, where $|z|_+ := \max(|z|,1)$.
  The $ \mathbf c$ is known as the zero mode in physics, which comes from the fact that the gradient $ \vert \partial^g X \vert^2 $ only defines the field up to a global constant, and one must integrate over this degree of freedom. 
 This construction corresponds to choosing  $g(x) = \max(1, |x|)^{-4} $  as the background metric in the Liouville action $S_L(X)$ in~\eqref{eq:path_integral}.
 The term $2Q \log |z|_+$ comes form the curvature terms of $S_L(X)$. As explained in  \cite{hrv-disk} and \cite{RZ_boundary}, the choice of the background metric only affects the law of field by an explicit multiplicative constant given by the Weyl anomaly.
 
 To make sense of the effect of $e^{\alpha \phi(z)}$, we let  $\LF_\bbH^{(\alpha, z)} = \lim_{\eps \to 0} \eps^{\alpha^2/2} e^{\alpha \phi_\eps(z)} \LF_\bbH(d\phi)$, 
 where $\phi_\eps$ is a suitable regularization at scale $\epsilon$ of $\phi$.  By virtue of the Girsanov theorem, $\LF_\bbH^{(\alpha, z)}$ can be realized as a sample from $\LF_\bbH$ plus an $\alpha$-log singularity at $z$. Lastly to handle the Liouville potentials $e^{\gamma X}$ and $e^{\frac{\gamma}{2} X}$ present in $S_L(X)$, 
 define the bulk and boundary Gaussian multiplicative chaos (GMC) measures  of $\phi$ as the limits (see e.g.~\cite{Ber_GMC, rhodes-vargas-review}):
 \begin{equation}\label{def:GMC_epsilon}
 \mu_\phi(\bbH) =  \lim_{\eps \to 0} \epsilon^{\frac{\gamma^2}{2}} \int_{\bbH} e^{\gamma \phi_{\epsilon}(z)} d^2z, \quad \text{and} \quad \nu_\phi(\R) = \lim_{\eps \to 0} \epsilon^{\frac{\gamma^2}{4}} \int_{\R} e^{\frac{\gamma}{2} \phi_{\epsilon}(z)} dz.
 \end{equation}
Now for $\gamma \in (0,2)$, $\mu, \mu_B >0$, set 
\begin{equation}\label{eq:1pt}
 \left \langle  e^{\alpha \phi(z)}  \right \rangle:= \LF_\bbH^{(\alpha, z)}[e^{-\mu \mu_\phi(\bbH) - \mu_B \nu_\phi(\R)} -1], \quad \textrm{for }z\in \bbH.
\end{equation}
We will explain in Section~\ref{subsec:pre-LCFT} that $ \left| \left \langle  e^{\alpha \phi(z)}  \right \rangle \right|<\infty$  when $\alpha \in (\frac{2}{\gamma}, Q)$, 
thanks to the $-1$ in~\eqref{eq:1pt}.
Moreover, for $\Delta_\alpha=\frac{\alpha}{2}(Q-\frac{\alpha}{2})$, the quantity $\left|\Im  z\right|^{2\Delta_\alpha} \left \langle  e^{\alpha \phi(z)}  \right \rangle$ 
does not depend on $z\in \bbH$. For concreteness, we take $z=i$ and set 
 \begin{equation}\label{eq:def_U_alpha}
U(\alpha) := \left \langle  e^{\alpha \phi(i)}  \right \rangle.
 \end{equation}
Now we are ready to state our main result.

 \begin{theorem}\label{thm:FZZ-physics}
 	For $\gamma \in (0,2)$, $\alpha \in (\frac{2}{\gamma}, Q)$, $\mu, \mu_B >0$, we have 	$U(\alpha) = U_{\mathrm{FZZ}}(\alpha)$.
 \end{theorem}
 The condition $\alpha \in (\frac{2}{\gamma}, Q)$ is required for \eqref{eq:1pt} to be finite, but one can extend the probabilistic definition of $U(\alpha)$ and the result to $\alpha \in (\frac{\gamma}{2}, Q)$; see  Theorem~\ref{thm:inverse-gamma} and Corollary \ref{cor_truncation}. So far in the probability literature the exact formulas on LCFT have all been derived by implementing the BPZ equations and the operator product expansion of~\cite{BPZ1984},
 as first performed in \cite{DOZZ_proof} proving the DOZZ formula. In the setup of a domain with boundary the same technique has been applied in the works \cite{remy-fb-formula, RZ_interval, RZ_boundary}, which all compute different cases of boundary Liouville correlations with $\mu =0, \mu_B>0$. This method has a major obstruction to prove Theorem \ref{thm:FZZ-physics}. 
Indeed, in order to define an observable satisfying the BPZ equation, the range of $\alpha$ needs to contain  an interval of length strictly greater than $\frac{2}{\gamma}$. The best range of $\alpha$ for a GMC definition of $ \left \langle  e^{\alpha \phi(z)}  \right \rangle$ is $(\frac{\gamma}{2}, Q)$ - see Corollary \ref{cor_truncation} -  which has length exactly $\frac{2}{\gamma}$ and thus is not sufficient. Another less fundamental but technically challenging issue is to reveal  the intricate  dependence on $\mu, \mu_B$ in $ U_{\mathrm{FZZ}}(\alpha)$. 
In the next subsection, we explain our strategy based on the conformal welding of quantum surfaces that  circumvents  these difficulties.

\subsection{A proof  strategy based on the conformal welding of quantum surfaces}\label{subsec:sketch}
By definition, the FZZ formula describes the joint law of  $\nu_\phi(\R)$ and $\mu_\phi(\bbH)$ in~\eqref{def:GMC_epsilon}  where $\phi$ is a sample from $\LF_\bbH^{(\alpha,i)}$. 
Here although $\LF_\bbH^{(\alpha,i)}$ is an infinite measure we adopt the probability terminology such as ``sample'' and ``law''.  
The law  of $\nu_\phi(\R)$ is encoded in the limiting case of the FZZ formula where $\mu=0$ and $\mu_B>0$, which has been obtained in~\cite{remy-fb-formula}.
Given this result, it turns out that the FZZ formula is equivalent to the statement that conditioning on $\nu_\phi(\R)=1$, the conditional law of  $\mu_\phi(\bbH)$ is the inverse gamma distribution with certain parameters. Here the inverse gamma distribution with shape parameter $a$ and scale parameter $b$ has the following density:
\(1_{x>0} \frac{b^a}{\Gamma(a)} \frac{1}{x^{a+1}} e^{-b/x}\). 
The crux of this paper is to derive the desired inverse gamma distribution using conformal welding of quantum surfaces. 
In this section we sketch this  strategy.

Quantum surfaces are the generalization of 2D Riemannian manifolds in the LQG random geometry.
For a fixed $\gamma\in (0,2)$, consider triples  $(D,h,z)$ where $D$ is a domain, $h$ is a variant of Gaussian free field on $D$, and $z\in D$.
We say that $(D,h,z)$ is equivalent to $(\wt D,\wt h,\wt z)$ if there exists a conformal map 
$\psi: \wt D \to D$ such that  \(\wt h = h \circ \psi + Q \log |\psi'|\) and $\psi(z)=\wt z$. Under this equivalence relation, the intrinsic geometric quantities in $\gamma$-LQG such 
as the quantum area and length measures transform covariantly under conformal maps. Here the quantum area and length are defined by Gaussian multiplicative chaos as in~\eqref{def:GMC_epsilon}. 
A quantum surface with one interior marked point is an equivalence class under this equivalence relation. We can similarly define quantum surfaces with more marked points or decorated with other natural structures such as  curves. 

For $\alpha\in (\frac{\gamma}{2},Q)$, sample $\phi$ from $\LF_\bbH^{(\alpha,i)}$ and condition on $\nu_\phi(\R)=\ell>0$. 
(This conditioning makes sense; see Lemma~\ref{lem:disint-alpha}.)
We write the conditional law of the quantum surface corresponding to $(\bbH, \phi,i)$ as $\cM_{1,0}^\disk(\alpha; \ell)^\#$. 
With this notion, the FZZ formula  can be reduced to the following.
\begin{theorem}\label{thm:inverse-gamma}
	For $\alpha\in (\frac{\gamma}{2},Q)$  
	the law of the quantum area of a sample from $\cM_{1,0}^\disk(\alpha; 1)^\#$ is the inverse gamma distribution with shape $\frac2\gamma(Q-\alpha)$ and scale $\frac1{4\sin \frac{\pi \gamma^2}4}$.
\end{theorem}

When $\alpha=\gamma$, by~\cite{cercle-quantum-disk} and~\cite{AHS-SLE-integrability},  $\cM_{1,0}^\disk(\alpha; 1)^\#$ describes the law of the so-called \emph{quantum disk} with unit boundary length and one interior marked point. In this case, based on the mating-of-trees theory of Duplantier, Miller, and Sheffield~\cite{wedges}, 
Gwynne and the first author of this paper~\cite{ag-disk} proved	that the law of the quantum area 
is the inverse gamma distribution with shape $\frac2\gamma(Q-\gamma)$ and scale $\frac{1}{2 \BB a^2 \sin^2(\pi\gamma^2/4) }$, where $\BB a^2$
is the unknown variance in the mating-of-trees theory, which first appeared in~\cite[Theorem 8.1]{wedges}.

Let us review the mating-of-trees theory.
In our proofs  we will only use some of its consequences  that can be stated without explicit reference to  it. 
Hence we will keep our discussion brief and  refer to the survey~\cite{ghs-mating-survey} for more background, especially on its fundamental role in the recent development on the scaling limit of random planar maps.
Recall that the Schramm-Loewner evolution ($\SLE_\kappa$) with a parameter $\kappa>0$ is a canonical family of conformal invariant random planar curves 
discovered by Schramm~\cite{schramm0}. In a nutshell, mating-of-trees theory says that if we run a space-filling variant of an $\SLE_{16/\gamma^2}$
curve on top of an independent $\gamma$-LQG surface, then this curve-decorated quantum surface can be encoded by a  two-dimensional correlated Brownian motion 
$(L_t, R_t)$ such that 
\begin{equation}\label{eq:MOT}
\Var[L_1]=\Var[R_1]=\BB a^2 \textrm{ and } \Cov(L_1 R_1)=  -\BB a^2 \cos (\frac{\gamma^2\pi}{4}).
\end{equation}

Here the \emph{mating-of-trees variance} $\BB a^2$ is an unknown function of the parameter $\gamma$. As a first step towards proving Theorem~\ref{thm:inverse-gamma}, we identify  the value of $\BB a^2$.

\begin{theorem}\label{thm:MOT-var}
	For $\gamma\in (0,2)$, the mating-of-trees variance $\mathbbm a^2 = \mathbbm a^2(\gamma)$ is given by 
	\[\mathbbm a^2 = \frac2{\sin (\frac{\pi \gamma^2}4)}.\]
\end{theorem}
We will prove Theorem~\ref{thm:MOT-var} in Section~\ref{sec:variance}. 
Our proof has two ingredients: a systematic understanding of the relation between canonical quantum surfaces  and LCFT 
developed by the first and the third authors with Holden in~\cite{AHS-SLE-integrability}; the explicit boundary LCFT correlation functions computed by the second author in~\cite{remy-fb-formula} and in his joint work ~\cite{RZ_boundary} with Zhu. 

To prove Theorem~\ref{thm:inverse-gamma}, we use the idea of conformal welding which we recall now.
For $\gamma\in (0,2)$ and $\kappa=\gamma^2\in (0,4)$, 
if we run an independent $\SLE_\kappa$ on top of a certain type of 
$\gamma$-LQG quantum surface, the two quantum surfaces on the two sides of the SLE curve are independent quantum surfaces. 
Quantum boundary lengths from the two sides agree on the curve, defining an unambiguous  notion of quantum length on the SLE curve. 
Moreover, the original curve-decorated quantum surface can be recovered by gluing the two smaller quantum surfaces according to the quantum boundary lengths.
This recovering procedure is called \emph{conformal welding}.
Such results were first established by Sheffield~\cite{shef-zipper} and later extended in~\cite{wedges,ahs-disk-welding}.

In this paper we prove a new conformal welding result that enables us to derive the inverse gamma distribution in  Theorem~\ref{thm:inverse-gamma}. 
It asserts  the existence of an  $\SLE_\kappa$ type curve $\eta$ with the following properties.  See Figure~\ref{fig-bubble} for an illustration. 
\begin{itemize}
	\item $\eta$ is a simple closed random curve surrounding $i$ that visits 0 and otherwise is in $\bbH$. 
	\item Suppose $\phi$ is independent of $\eta$ and  is sampled from the conditional law of $\LF^{(\alpha,i)}_\bbH$ conditioning on $\nu_\phi(\R)=1$. 
	Let $D_\eta(0)$  and  $D_\eta(\infty)$ be the bounded and unbounded component of $\bbH\setminus \eta$, respectively. 
	Then as quantum surfaces with marked points, 
	$(D_\eta(0),\phi,i,0)$  and  $(D_\eta(\infty),\phi,0^-, 0^+)$ are conditionally independent given the quantum length of $\eta$.
	\item  The conditional law  of $(D_\eta(0),\phi,i)$ conditioning  on  the quantum length $\ell$ of $\eta$ is $\cM_{1,0}^\disk(\alpha; \ell)^\#$.  
	\item The law of the quantum area and quantum lengths of the two boundary arcs of  $(D_\eta(\infty),\phi,0^-, 0^+)$ can be 
	explicitly described in terms of  the mating-of-trees Brownian motion $(L_t, R_t)$.
\end{itemize}
\begin{figure}[ht!]
	\begin{center}
		\includegraphics[scale=0.65]{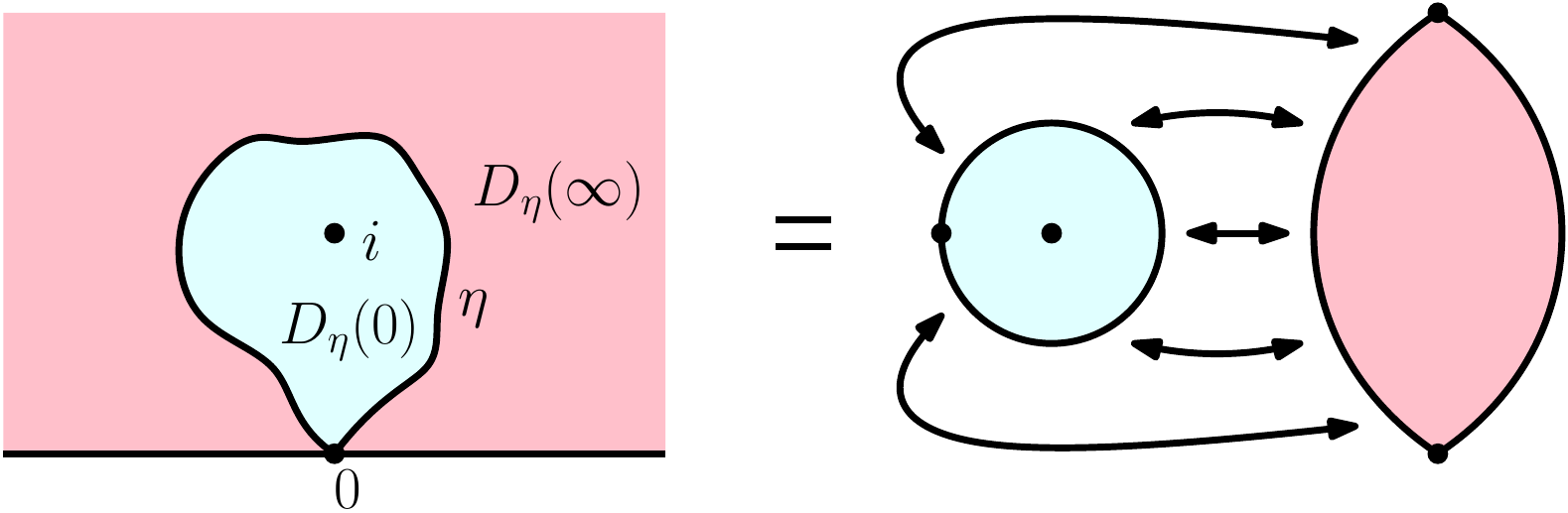}%
	\end{center}
	\caption{\label{fig-bubble} \textbf{Left:} A  curve $\eta$  independent of a sample $\phi$ from $\LF_\bbH^{(\alpha ,i)}$. \textbf{Right: }The quantum surfaces $(D_\eta(0), \phi, i, 0)$ and $(D_\eta(\infty), \phi, 0^-, 0^+)$ are conditionally independent given the quantum length of $\eta$.
	Moreover,  $(\phi, \eta)$ can be recovered from the two quantum surfaces by conformal welding.}
\end{figure}
This result is stated as Theorem~\ref{thm-FZZ-weld} and proved in Sections~\ref{subsec:bubble2} and~\ref{sec:bubble}. 
It relies on the conformal welding results for finite area quantum surfaces proved in~\cite{ahs-disk-welding}.
The law of $(D_\eta(\infty),\phi,0^-, 0^+)$  is given by a variant of the so-called \emph{two-pointed quantum disk with weight $\gamma^2/2$}, whose quantum length and area distribution is obtained in~\cite{ahs-disk-welding}  in terms of the mating-of-trees Brownian motion; see Section~\ref{subsec:special disks}.

Using the above conformal welding result, we can obtain a recursive relation on the law of the quantum area of a sample from $\cM_{1,0}^\disk(\alpha; 1)^\#$. Using a path decomposition for  Brownian motion in cones, we prove that the only solution to this recursion is the inverse gamma distribution in~Theorem~\ref{thm:inverse-gamma}, which is identified as the law of the duration of a certain  Brownian motion in cones. 
For technical reasons, we carry out  this argument for $\alpha\in (\frac{\gamma}{2},Q-\frac{\gamma}{4})$ first. This gives Theorem~\ref{thm:FZZ-physics} for 
$\alpha\in (\frac{2}{\gamma},Q-\frac{\gamma}{4})$. Since $\frac{\gamma}{2}<\frac2{\gamma}<Q-\frac{\gamma}{4}$,
by the analyticity of $U(\alpha)$ in $\alpha\in (\frac{2}{\gamma},Q)$, 
we obtain Theorem~\ref{thm:FZZ-physics} in its full range $(\frac2{\gamma},Q)$, 
which in turn gives Theorem~\ref{thm:inverse-gamma} in its full range $(\frac{\gamma}{2},Q)$.
It also allows us to extend Theorem~\ref{thm:FZZ-physics} to $\alpha\in (\frac{\gamma}{2},Q)$ for a suitably defined $U(\alpha)$; see Corollary \ref{cor_truncation}.
See  Section~\ref{sec:FZZ} for the detailed argument.

\subsection{Exact solvability for moments of a  Gaussian multiplicative chaos}\label{subsub:GMC-disk}
It is natural to look at the limits $\mu \rightarrow 0$ or $\mu_B \rightarrow 0$ in the exact formula of Theorem \ref{thm:FZZ-physics}. These limits have the effect of deleting one of the two Liouville potentials in \eqref{liouville_action} and the probabilistic expression for $U(\alpha)$ then reduces up to an explicit prefactor to a moment of GMC either on the bulk or on the boundary of the domain.  In the case of $\mu \rightarrow 0$, see equation \eqref{eq:U0} for this reduction. Our main result then reduces in this case to Proposition \ref{prop-remy-U} giving an exact formula for the moment of GMC on $\mathbb{R}$.  This formula was derived in \cite{remy-fb-formula} and is actually used in our proof of Theorem \ref{thm:FZZ-physics}. On the other hand the limit $\mu_B \rightarrow 0$ provides a novel result on the moments of GMC on $\mathbb{H}$.

\begin{proposition} Let $\gamma \in (0,2)$, $\alpha \in (\frac{\gamma}{2}, Q)$, and $\wt \phi = h - 2Q \log |z|_+ + \alpha G_{\mathbb{H}}(z,i)$,  
	where $h $ is a sample from 
	$P_{\mathbb{H}}$.  Then
	\begin{equation}
	\E \left[ \mu_{\wt \phi}(\mathbb{H})^{\frac1 \gamma(Q-\alpha)} \right] =  \frac{ 2}{\sqrt{\pi}}  \left( \frac{ \pi}{ 2^{\gamma \alpha +2 }}  \frac{\Gamma(\frac{\gamma^2}{4})}{\Gamma(1 -\frac{\gamma^2}{4})} \right)^{\frac{Q - \alpha}{\gamma} } \Gamma( \frac{\gamma \alpha}{2} - \frac{\gamma^2}{4} ) \Gamma(\frac{\alpha}{\gamma} - \frac{2}{\gamma^2}) \cos \left(\pi \frac{\alpha -Q}{\gamma} \right).
	\end{equation}
\end{proposition}
A further corollary can be derived if one assumes the moment $\frac{1}{\gamma}(Q -\alpha)$ of the above GMC to be an integer $n \in \mathbb{N}$. It is a well-known simple Gaussian computation that a positive integer moment of GMC reduces to a Selberg type integral, see \cite{Selberg_int} for a review on these integrals.
\begin{corollary}
	Let $\gamma \in (0,2) $ and $ n \in \mathbb{N}$ such that $0 < n < \frac{2}{\gamma^2}$. Then the following holds:
	\begin{align}\label{DF_integral_H}
	& \int_{\mathbb{H}^n} dz_1 \cdots dz_n \prod_{i<j} \left( \frac{1}{|z_i -z_j|^{\gamma^2} |z_i -\overline{z}_j|^{\gamma^2} } \right)   \prod_{i=1}^n  \left( \frac{ 1 }{| z_i - \overline{z}_i |^{\frac{\gamma^2}{2}}} \frac{1}{|z_i^2  + 1|^{\gamma Q - n\gamma^2} } \right)
	\\ \nonumber
	& \quad = \frac{ 2}{\sqrt{\pi}}  \left( \frac{- \pi}{ 2^{ 4 + \frac{\gamma^2}{2}(1-2n)  }}  \frac{\Gamma(\frac{\gamma^2}{4})}{\Gamma(1 -\frac{\gamma^2}{4})} \right)^{n } \Gamma( 1 - \frac{\gamma^2}{2} n ) \Gamma(\frac{1}{2} - n).
	\end{align}
\end{corollary}
The integral \eqref{DF_integral_H} ressembles the Dotsenko-Fateev integral of \cite{DF1984} appearing in the context of CFT on the Riemann sphere, with $\mathbb{C}^n$ being replaced by $\mathbb{H}^n$. To the best of our knowledge the above evaluation of \eqref{DF_integral_H} was not known. 
Notice that for $\gamma \geq \sqrt{2}$, there are no valid $n$ as the first moment of the GMC on $\mathbb{H}$ is not finite. For any $\gamma \in (0,\sqrt{2})$, there are finitely many $n$ that satisfy $0 < n < \frac{2}{\gamma^2}$. Lastly let us note by conformal invariance it is possible to write both of these results on the unit disk $\mathbb{D}$, see for instance \cite[Section 5.3]{RZ_boundary} for how to link moments of GMC on $\mathbb{D}$ and $\mathbb{H}$.

\subsection{Perspectives and outlook}\label{subsec:outlook}
In this section we describe several perspectives, ongoing projects,  and future directions related to the FZZ formula. 

\subsubsection{Integrability and the conformal bootstrap for boundary LCFT}\label{subsub:BLCFT}

In order to carry out the conformal bootstrap for LCFT on Riemann surfaces with boundary, along with the bulk one-point function that we obtained in Thoerem~\ref{thm:FZZ-physics}, one needs to compute three other correlation functions for LCFT on $\bbH$
\begin{equation}\label{eq:GRH}
\left \langle e^{\alpha \phi(z)} e^{\frac{\beta}{2} \phi(s)} \right \rangle, \quad \left \langle e^{\frac{\beta}{2} \phi(s_1)} e^{\frac{\beta}{2} \phi(s_2)} \right \rangle, \quad \left \langle e^{\frac{\beta_1}{2} \phi(s_1)} e^{\frac{\beta_2}{2} \phi(s_2)} e^{\frac{\beta_3}{2} \phi(s_3)} \right \rangle,
\end{equation}
where $z \in \mathbb{H}$, $s, s_1, s_2, s_3 \in \mathbb{R}$.
For the boundary two-point and three-point functions one also has the freedom to choose $\mu_B$ as a function defined on the boundary 
which is constant on each arc in between boundary insertions. See \cite[Figure 1]{RZ_boundary} for a summary. 

Along with $\left \langle e^{\alpha \phi(z)} \right \rangle$, these three correlation functions are the ``basic" correlations as thanks to conformal invariance they depend trivially on $z,s_1,s_2,s_3$. On the other hand their dependence on $\gamma,\mu,\mu_B, \alpha, \beta, \beta_i$ is non-trivial. Explicit formulas for these functions have been proposed in the physics papers  \cite{FZZ, Hosomichi_2001, PT_boundary_3pt}.
In the limiting case where $\mu =0$, they reduce to moments of the boundary GMC measure and they have been explicitly computed in the work \cite{RZ_boundary}.  	
In a work in progress with Zhu, we plan to verify the proposed formulas for all three correlations for general $\mu>0$.
The techniques will be a combination of the tools of the present paper and of \cite{RZ_boundary}.
Indeed, the approach in  \cite{RZ_boundary} is still applicable to the second and third correlation functions in~\eqref{eq:GRH} as there is no bulk insertion.  
By a conformal welding statement  similar to the one in this paper, we will get the first one from the second one and the FZZ formula.

Once the above correlations have been evaluated, the next natural step is to compute a correlation function with more marked points or on a non simply connected surface with boundary. For this purpose one needs to implement the conformal bootstrap procedure first proposed in physics in \cite{BPZ1984}. At the level of mathematics this has been achieved to compute an $N$-point function on the sphere in the recent breakthrough \cite{GKRV_bootstrap}. One can expect to adapt the methods of \cite{GKRV_bootstrap} to the case of a surface with boundary, where the FZZ formula along with the correlations \eqref{eq:GRH} are a crucial input. More concretely in \cite{Martinec_2003} a bootstrap equation involving the FZZ formula is proposed to compute the partition function of LCFT on an annulus. We state it here as a conjecture.
\begin{conjecture}
	Consider an annulus represented as a cylinder of length $\pi \tau$ and of radius $1$. Let $ q = e^{- 2 \pi \tau} $ and $\mathcal{Z}_{\mathrm{Annulus}}$ be the partition function (no insertion points) of LCFT on this annulus defined probabilistically in an analogous way to \eqref{eq:1pt}, see also \cite{Remy_annulus}. Then the following equality should hold: 
	\begin{equation}\label{bootstrap_annulus}
	\mathcal{Z}_{\mathrm{Annulus}} = \int_{\mathcal{C}} U_{\mathrm{FZZ}}(Q + i P) U_{\mathrm{FZZ}}(Q - i P) q^{P^2 - \frac{1}{24}} \prod_{n \geq 1} (1 - q^n) dP.
	\end{equation}
	One has the degree of freedom to choose different values of $\mu_B$ for each of the two boundaries of the annulus, in which case the two $U_{\mathrm{FZZ}}$ functions in the right hand side of \eqref{bootstrap_annulus} must be computed respectively with those two values. The $\mu$ parameter is the same for $\mathcal{Z}_{\mathrm{Annulus}} $ and both $U_{\mathrm{FZZ}}$ functions. Lastly the contour of integration $\mathcal{C}$ is a suitable deformation of $\mathbb{R}$ avoiding the pole at $P =0$.
\end{conjecture}

\subsubsection{Interplay between three types of integrability in conformal probability}\label{subsub:interplay}
Our paper is part of an ongoing program of the first and the third authors 
to prove integrable results for SLE, LCFT, and mating-of-trees via two connections between these subjects:
1.\ equivalent but complementary descriptions of canonical random surfaces  in the path integral (e.g. \cite{dkrv-lqg-sphere}) and mating-of-trees (e.g. \cite{wedges}) perspectives; 
2.\ conformal welding of these surfaces with SLE curves as the interface. 
We now describe other aspects in this program that are closely related to this paper. 
\begin{itemize}
\item  Our paper relies on the recent work of the first and the third authors with Holden~\cite{ahs-disk-welding,AHS-SLE-integrability}. 
In particular, our conformal welding result is built on the one proved in~\cite{ahs-disk-welding}  for two-pointed quantum disks. 
Our evaluation of the mating-of-trees variance uses the technique from~\cite{AHS-SLE-integrability}  on the LCFT description of quantum surfaces. 

One of the main results in~\cite{AHS-SLE-integrability} is an exact formula for a variant of SLE curves called the chordal $\SLE_\kappa(\rho_-;\rho_+)$ (see Section~\ref{subsec:welding}).
Its proof shares the same starting point with our proof of the FZZ formula: a conformal welding result for quantum disks. 

\cite{AHS-SLE-integrability} demonstrated how to get integrability results for SLE using LCFT, mating-of-trees and conformal welding.  
 Using  the same methodology  and taking the FZZ formula as a crucial input,  the first and the third authors proved two integrable results 
 on the conformal loop ensemble in~\cite{AS-CLE}.
 The first relates the three-point correlation function of the conformal loop ensemble on the sphere ~\cite{JLS-CLE} with the DOZZ formula.
 The other addresses a conjecture of Kenyon and Wilson (recorded in~\cite[Section 4]{ssw-radii}) on its electrical thickness.

\item Our proof of the FZZ formula does not rely on conformal field theoretical techniques such as the BPZ equations and operator product expansion,
except for  two exact formulas from~\cite{remy-fb-formula} and~\cite{RZ_boundary}; see~\eqref{eq:U0-explicit} and~\eqref{eq-R}. 
In an ongoing project with Da Wu, the first and the third authors aim at using the conformal welding and stochastic calculus methods in SLE  to recover these two formulas.
Combined with our paper, this will provide a proof of the FZZ formula which is fully based on SLE, mating-of-trees, and conformal welding. 
It is an interesting open question  whether such a proof can be given for the DOZZ formula.

\item The first and the third authors are working on verifying our belief 
that in a very general sense the conformal welding of two quantum surfaces defined by LCFT can be described by LCFT, where
the conformal welding result (Theorem~\ref{thm-bubble-zipper})  in this paper is a special case. 
For another instance of this statement,
consider a pair of independent Liouville fields on a disk with three boundary insertions. If we conformally weld them along the three boundary arcs,
then the resulting surface should be given by a Liouville field on the sphere with three bulk insertions.
Moreover, the magnitude of an  insertion on the sphere is determined by  the local rule for conformal welding described in~\cite{wedges}.
More precisely, in the terminology of~\cite{wedges}, if we zoom in around an insertion on the sphere the picture looks like the  conformal welding of two \emph{quantum wedge}s into a \emph{quantum cone}, which is understood in~\cite{wedges}. In our present case,  if we zoom in around $0$ in Figure~\ref{fig-bubble}, the picture will converge to
 the  conformal welding of three quantum wedges into a single quantum wedge as established in~\cite{wedges}. 
\end{itemize}
 
\subsubsection{FZZT branes and related models}
We review here several models related to the FZZ formula in the theoretical physics literature. Boundary Liouville theory admits two kinds of boundary states, the one studied in the present paper being the so-called Fateev-Zamolodchikov-Zamolodchikov-Teschner (FZZT) brane \cite{FZZ}, \cite{Teschner_brane}. The other is the Zamolodchikov-Zamolodchikov (ZZ) brane which corresponds to LCFT on a disk with the hyperbolic metric as background metric; see \cite{ZZ_brane}. In this ZZ brane setup there is also a formula for the bulk one-point function, see \cite[Equation (2.16)]{ZZ_brane}. Both the FZZT and ZZ branes play a role in the relation between matrix models, integrable heirarchies, and Liouville quantum gravity; see for instance \cite{AleshkinMLG} and references therein. Lastly one can consider LCFT on a non-orientable surface, the simplest model being the projective plane, see \cite[Equation (3.10)]{HikidaLCFT} for the analogue of $U(\alpha)$ in that case. 
We hope to adapt our methods to study these directions in the future, see also the review \cite{Nakayama_review} for more details.

\bigskip

\noindent\textbf{Organization of the paper.} After providing background in Section~\ref{sec:pre},  we obtain the mating of trees variance in Section~\ref{sec:variance}.
Then in Section~\ref{sec:FZZ}, we carry out the strategy outlined in Section~\ref{subsec:sketch} modulo the proof of the conformal welding result, whose proof is supplied in Section~\ref{sec:bubble}.

\bigskip

\noindent\textbf{Acknowledgements.} 
We thank Nina Holden, R\'emi Rhodes, Vincent Vargas, and Tunan Zhu for early discussions on alternative approaches towards proving the FZZ formula. 
 We also thank Scott Sheffield for suggesting the mating-of-trees variance problem and Ewain Gwynne for helpful discussions on this. We thank Konstantin Aleshkin for many explanations on the theoretical physics background behind the FZZ formula.
We are grateful to Wendelin Werner for enlightening discussions. 
We thank an anonymous referee for helpful comments on an earlier version of this paper.
This project was initiated when G.R. visited MIT in the Spring of 2018. He would like to very warmly thank Scott Sheffield for his hospitality.
M.A. was partially supported by NSF grant DMS-1712862. 
G.R. was supported by an NSF mathematical sciences postdoctoral research fellowship, NSF Grant DMS-1902804.
X.S.\ was supported by  the Simons Foundation as a Junior Fellow at the Simons Society of Fellows,  and by the  NSF grant DMS-2027986 and the Career award 2046514.

%% file: pre.tex
\section{Preliminaries}\label{sec:pre}
In this section we provide the necessary background for our proofs. We start by fixing some global notations and conventions in Section~\ref{subsec:notation}.
We then review   Liouville CFT and quantum disks in Section~\ref{subsec:pre-LCFT} and~\ref{subsec:QD}, respectively. 
We will only need these two sections for the evaluation of the mating of trees variance in Section~\ref{sec:variance}. We recall the conformal welding of quantum disks 
in Section~\ref{subsec:welding} and some properties on two special quantum disks in Section~\ref{subsec:special disks}. 
\subsection{Notations and conventions}\label{subsec:notation}
Throughout the paper we assume $\gamma\in (0,2)$ is the LQG coupling constant. Moreover, 
\begin{equation}\label{key}
Q=\frac\gamma2+\frac2\gamma \quad \textrm{and} \quad \kappa=\gamma^2.
\end{equation}

We will work with planar domains including the upper half plane $\bbH=\{ z \in \C: \Im (z)>0 \}$, the unit disk $\D=\{z\in \C: |z|<1 \}$,
and the horizontal strip $\cS=\R\times (0,\pi)$. For a domain $D\subset \C$, we write $\bdy D$ of $D$ as the set of prime ends of $D$ 
and call it the \emph{boundary} of $D$. For example, $\bdy\bbH=\R\cup\{\infty \}$
and $\bdy \cS= \{ z\in \C: \Im z=0 \textrm{ or }1 \} \cup \{\pm \infty \}$.

We write  $\Gamma$ as the unique meromorphic function such that $\Gamma (\alpha) =\int_{0}^{\infty} x^{\alpha-1}e^{-x} dx$ for $\alpha>0$. 
For $a>0$ and $b>0$,  the inverse gamma distribution with shape parameter $a$ and scale parameter $b$ has the following density:
\begin{equation}\label{eq:inverse-gamma-parameter}
1_{x>0} \frac{b^a}{\Gamma(a)} \frac{1}{x^{a+1}} e^{-b/x}.
\end{equation}

We will frequently consider non-probability measures and extend the terminology of probability theory to this setting. 
In particular, suppose $M$ is a measure on a measurable space $(\Omega, \cF)$ such that $M(\Omega)$ is not necessarily $1$, and $X$ is a $\cF$-measurable function. 
Then we say that $(\Omega, \cF)$ is a sample space, $X$ is a random variable.  We call  the pushforward measure $M_X = X_*M$  the \emph{law} of $X$.
We say that $X$ is \emph{sampled} from $M_X$. We also write $\int f(x) M_X(dx)$ as $M_X[f]$ for simplicity.
For a finite positive measure $M$, we denote as its total mass by $|M|$
and write $M^{\#}=|M|^{-1}M$ as the corresponding probability measure.

	Let $g$ be a smooth metric on $\bbH$ such that 
	the metric completion of $(\bbH, g)$ is a compact Riemannian manifold.
	Let $H^1(\bbH, g)$ be the Sobolev  space whose norm is the sum of the Dirichlet energy and  the $L^2$-norm with respect to $(\bbH, g)$.
	Let $H^{-1}(\bbH)$ be the dual space of $H^1(\bbH, g)$. 
	Then the function space $H^{-1}(\bbH)$ and its topology do not depend on the choice of $g$, and it is a Polish  (i.e.\ complete separable metric) space. All random functions on $\bbH$ considered in this paper will belong to $H^{-1}(\bbH)$.

\subsection{Liouville conformal field theory on the upper half plane}\label{subsec:pre-LCFT}

For convenience we present LCFT on
 domains conformally equivalent to a unit disk,   with the upper half plane $\bbH$ as our base domain. 
Let $h$ be the centered Gaussian process on $\bbH$ with covariance kernel given by
\begin{equation}\label{covariance}
\E[h(x)h(y)]= G_\bbH(x,y) :=\log \frac{1}{|x-y||x-\bar{y}|} + 2 \log |x|_+ + 2 \log |y|_+,
\end{equation}
where  $|x|_+ := \max(|x|,1)$ and in the sense that  $\E[(h,f) (h,g)]= \iint f(x)\E[h(x)h(y)]g(y) dxdy$, for smooth test functions $f,g$.
Let $P_\bbH$ be the law of $h$. Using the arguments in  \cite{shef-gff,dubedat-coupling} it can be shown that $P_\bbH$ is a probability measure on  the space $H^{-1}(\bbH)$ defined in Section~\ref{subsec:notation}.
For smooth test functions $g$ and $f$ such that $\int_{\bbH} f(z)\,d^2 z=\int_{\bbH} g(z)\, d^2z=0$, we have $\E[(h,f) (h,g)]=(2\pi)^{-1}\int_{\bbH} \nabla g \cdot \nabla f\, d^2z$. This is the characterizing property of the free boundary Gaussian free field, which is only uniquely defined modulo  additive constant. 
The field $h$ is the  particular variant where the additive constant is fixed by requiring the average around the upper half unit circle to be 0.

Given a function $f\in H^{-1}(\bbH)$ and $z\in \bbH\cup  \R$, let $f_\eps(z)$ be the average of $f$ over $\partial B_\eps(z)\cap \bbH$.  
For $h\sim P_\bbH$, define the random measures $\mu_h=\lim_{\eps\to 0}\eps^{\gamma^2/2} e^{\gamma h_\eps} d^2z$ and $\nu_h=\lim_{\eps\to 0}\eps^{\gamma^2/4} e^{\gamma h_\eps/2}dz$, where the convergence holds in probability in weak topology.  See \cite{Ber_GMC, rhodes-vargas-review} and references therein for more details on this construction.
We call $\mu_h$ and $\nu_h$ the quantum area and quantum boundary length measures, respectively, corresponding to $\phi$. 
For functions which can be written as a sum of the GFF and a continuous function, quantum length and area can be defined similarly.

\begin{definition}[Liouville field]\label{def_GFF} 
	Sample  $(h, \mathbf c)$ from the measure $P_\bbH \times [e^{-Qc} \, dc]$ 
	on $H^{-1}(\bbH)\times \R$, and let $\phi(z) = h(z) - 2Q \log |z|_+ + \mathbf c$ be a random field on $\bbH$.
	Let $\LF_\bbH$  be the measure on $H^{-1}(\bbH)$ which describes the law of $\phi$. We call a sample from $\LF_\bbH$ 
	a Liouville field on $\bbH$.
\end{definition}

	The following lemma defines Liouville fields with insertions by making sense of $e^{\alpha \phi(z_0)} \LF_\bbH(d\phi)$.
\begin{lemma}\label{lem:Girsanov}
	For $\alpha \in \R$ and $z_0\in\bbH$, the limit $\LF_\bbH^{(\alpha, z_0)} := \lim_{\eps \to 0} \eps^{\alpha^2/2} e^{\alpha \phi_\eps(z_0)} \LF_\bbH(d\phi)$ exists in the vague topology. Moreover, sample $(h, \mathbf c) $ from $(2 \Im z_0)^{-\alpha^2/2} |z_0|_+^{-2\alpha(Q-\alpha)}P_\bbH \times [e^{(\alpha - Q)c} dc]$ and let 
	\[\phi(z) = h(z) - 2Q \log |z|_+ + \alpha G_\bbH(z,z_0) +\mathbf c \quad \textrm{for }z\in \bbH.\]
	Then  the law of $\phi$ is given by  $\LF_\bbH^{(\alpha,z_0)}$. 
\end{lemma}

\begin{proof}
	Sample $h$ from $P_\bbH$, fix $c \in \R$ and set $\wt \phi(z) = h(z) - 2Q \log |z|_+ + c$. For a compactly supported continuous function $F$ on $H^{-1}(\bbH) \times \mathbb{R}$, Girsanov's theorem and $\Var h_\eps(z_0) = -\log \eps - \log (2\Im z_0) + 4 \log |z_0|_+ + o_\eps(1)$ give
	\begin{align*}
	\lim_{\eps \to 0} \E \left[ \eps^{\alpha^2/2} e^{\alpha \wt \phi_{\eps}(z_0)} F(h,  c)  \right] &= (2 \Im z_0)^{-\alpha^2/2} |z_0|_+^{-2\alpha(Q-\alpha)} \lim_{\eps \to 0} \E \left[ e^{\alpha  c} e^{\alpha h_{\eps}(z_0) - \frac{\alpha^2}{2} \E[h_{\eps}(z_0)^2] } F(h,  c)  \right]\\
	&= (2 \Im z_0)^{-\alpha^2/2} |z_0|_+^{-2\alpha(Q-\alpha)} \lim_{\eps \to 0} \E \left[  e^{\alpha  c} F\left(h(\cdot) + \alpha \E[h(\cdot) h_{\eps}(i) ],  c \right)  \right] \\
	&= (2 \Im z_0)^{-\alpha^2/2} |z_0|_+^{-2\alpha(Q-\alpha)}  \E \left[ e^{\alpha  c}  F\left(h(\cdot) + \alpha G_{\bbH}(\cdot,i),  c \right)  \right].
	\end{align*} 
Integrating over $[e^{-Qc}dc]$ yields the  lemma.
\end{proof}

\begin{definition}\label{def:LFalpha}
		
	We call a sample from $\LF_\bbH^{(\alpha, z)} $ 
	a Liouville field on $\bbH$  with an $\alpha$-insertion at $z$.
\end{definition}

The following lemma gives the change of coordinate for the Liouville field under a conformal map.
\begin{lemma}\label{lem-change-coord}
	For $(\alpha, z_0) \in \R\times \bbH$, let $\psi:\bbH\to\bbH$ be a conformal map such that $\psi(z_0)=i$. Sample $\phi$ from $\LF_\bbH^{(\alpha, z_0)}$, then the law of $\phi \circ \psi^{-1} + Q \log |(\psi^{-1})'|$ is $\left|\Im z_0\right|^{-2\Delta_\alpha} \LF_\bbH^{(\alpha,i)}$ where $\Delta_\alpha=\frac{\alpha}{2}(Q-\frac{\alpha}{2})$. 
\end{lemma}
\begin{proof} 
	 \cite[Theorem 3.5]{hrv-disk} gives this result when we parametrize the Liouville field in the disk $\D$, and is adapted from \cite[Theorem 3.5]{dkrv-lqg-sphere}. The same argument as in \cite[Theorem 3.5]{hrv-disk} applies to the upper half plane case. Alternatively, the result for $\D$ can be transferred to $\bbH$ via the coordinate change explained in \cite[Section 5.3]{RZ_boundary}. We omit the details. 
\end{proof}

Suppose $f$ is a measurable function on $H^{-1}(\bbH)$. We recall the   convention $M_X[f]$ of Section~\ref{subsec:notation} and write
$\LF_\bbH^{(\alpha, i)}[f]= \int f (\phi) \LF_\bbH^{(\alpha, i)} (d\phi )$.
\begin{definition}\label{def:U}
	For $\alpha\in (\frac{2}{\gamma},Q)$, $\mu\ge 0$ and  $\mu_B \in \mathbb{C}$ with $\mathrm{Re}(\mu_B)>0$, let 
	\[
	\left \langle  e^{\alpha \phi(z)}  \right \rangle= \left \langle  e^{\alpha \phi(z)}  \right \rangle_{\gamma, \mu,\mu_B}:= \LF_\bbH^{(\alpha, z)}[e^{-\mu \mu_\phi(\bbH) - \mu_B \nu_\phi(\R)} -1] \quad \textrm{for }z\in \bbH.
	\]
\end{definition}
We include in the definition the case of complex $\mu_B$ which will be used in Section \ref{sec:FZZ}.

\begin{lemma}\label{lem:U-bound}
	Suppose $\alpha\in (\frac{2}{\gamma},Q)$,  $\mu\ge 0$ and  $\mu_B \in \mathbb{C}$, $\mathrm{Re}(\mu_B)>0$. Then $|\left \langle  e^{\alpha \phi(i)}  \right \rangle|<\infty$. Moreover, the value of 
		$\left|\Im z\right|^{2\Delta_\alpha}\left \langle  e^{\alpha \phi(z)}  \right \rangle$ does not depend on $z\in \bbH$. 
	\end{lemma}
\begin{proof}
By Lemma~\ref{lem:Girsanov}, take $\wt \phi(z) = h(z) - 2Q \log |z|_+ +\alpha G_\bbH(z,i)$ where $h$ is sampled from $P_\bbH$. By integration by parts on the $c$ integral one has:
	\begin{align*}
	\left| \left \langle  e^{\alpha \phi(i)}  \right \rangle \right|  & = \left \vert \int_{\mathbb{R}} dc e^{(\alpha -Q)c} \E \left[ 1- e^{- \mu e^{\gamma c} \mu_{\tilde \phi}(\mathbb{H}) - \mu_B e^{\frac{\gamma c}{2}} \nu_{\tilde \phi}(\mathbb{R}) } \right] \right \vert \\
	&= \frac{1}{Q -\alpha} \left \vert \int_{\mathbb{R}} dc e^{(\alpha -Q)c} \E \left[ \left(  \mu \gamma e^{\gamma c} \mu_{\tilde \phi}(\mathbb{H}) +  \mu_B \frac{\gamma}{2} e^{\frac{\gamma c}{2}} \nu_{\tilde \phi}(\mathbb{R})  \right) e^{- \mu e^{\gamma c} \mu_{\tilde \phi}(\mathbb{H}) - \mu_B e^{\frac{\gamma c}{2}} \nu_{\tilde \phi}(\mathbb{R}) } \right] \right \vert \\
	& \leq \frac{1}{Q -\alpha}  \int_{\mathbb{R}} dc e^{(\alpha -Q)c} \E \left[ \left(  \mu \gamma e^{\gamma c} \mu_{\tilde \phi}(\mathbb{H}) +  | \mu_B | \frac{\gamma}{2} e^{\frac{\gamma c}{2}} \nu_{\tilde \phi}(\mathbb{R})  \right) e^{- \mu e^{\gamma c} \mu_{\tilde \phi}(\mathbb{H}) - \mathrm{Re}(\mu_B) e^{\frac{\gamma c}{2}} \nu_{\tilde \phi}(\mathbb{R}) } \right]  \\
	& \leq \frac{1}{Q -\alpha} \int_{\mathbb{R}} dc e^{(\alpha -Q)c} \E \left[  \mu \gamma e^{\gamma c} \mu_{\tilde \phi}(\mathbb{H}) e^{- \mu e^{\gamma c} \mu_{\tilde \phi}(\mathbb{H})  }  +  | \mu_B | \frac{\gamma}{2} e^{\frac{\gamma c}{2}} \nu_{\tilde \phi}(\mathbb{R})   e^{ - \mathrm{Re}(\mu_B) e^{\frac{\gamma c}{2}} \nu_{\tilde \phi}(\mathbb{R}) } \right]\\
	& \leq c_1 \E \left[ \mu_{\tilde \phi}(\mathbb{H})^{\frac{1}{\gamma}(Q -\alpha)} \right] + c_2 \E \left[  \nu_{\tilde \phi}(\mathbb{R})^{\frac{2}{\gamma}(Q -\alpha)} \right].
	\end{align*} 
	In the last line above $c_1$ and $c_2$ are explicit positive constants coming from evaluating the integral over $c$. The two expectations of GMC moments are finite for $\alpha \in (\frac{2}{\gamma},Q)$ thanks to~\cite[Corollary 6.11]{hrv-disk} for the first and~\cite[Lemma 3.10]{dkrv-lqg-sphere} adapted to the one dimensional case for the second. 
	Hence the claim $|\left \langle  e^{\alpha \phi(i)}  \right \rangle|<\infty$ holds. Finally, by Lemma~\ref{lem-change-coord} , $\left|\Im z\right|^{2\Delta_\alpha}\left \langle  e^{\alpha \phi(z)}  \right \rangle$ does not depend on $z\in \bbH$. 
\end{proof}

The next two statements give the law of the total quantum length $\nu_\phi(\R)$ under  $\LF_\bbH^{(\alpha, i)}$.
\begin{lemma}\label{lem-len-LF}
	For $\alpha > \frac\gamma2$, let $h \sim P_\bbH$ and $\wt \phi(z) = h(z) - 2Q \log |z|_+ +\alpha G_\bbH(z,i)$. Let  \(\ol U_0(\alpha):= \E[ \nu_{\wt \phi}(\R)^{\frac2\gamma(Q-\alpha)}]\) where the expectation $\E$ is with respect to $P_\bbH$. Then
	\begin{equation}\label{eq:length-density}
	\LF_\bbH^{(\alpha, i)}[f(\nu_\phi(\R)) ]= \int_{0}^\infty f(\ell)  \frac2\gamma 2^{-\frac{\alpha^2}2} \ol U_0(\alpha)\ell^{\frac2\gamma(\alpha-Q)-1} \, d\ell
	\end{equation}
	for each non-negative measurable function $f$ on $(0, +\infty)$.  Moreover,
	\begin{equation}\label{eq:U0}
\left \langle  e^{\alpha \phi(i)}  \right \rangle_{\gamma, 0,\mu_B}= \frac{2}{\gamma} 2^{-\frac{\alpha^2}2} \mu_B^{\frac{2}{\gamma}(Q -\alpha)} \Gamma \left(\frac{2}{\gamma}(\alpha - Q) \right) \ol U_0(\alpha) \quad \textrm{for } \alpha\in (\frac\gamma2, Q).
	\end{equation}
\end{lemma}
\begin{proof}
	
	We sample a real random number $\mathbf c $ from $2^{-\frac{\alpha^2}2}e^{(\alpha - Q)c}dc$ independently of $\wt \phi$, then the law of the field $\phi = \wt \phi + \mathbf c$ is 
	$\LF_\bbH^{(\alpha,i)}$. To prove~\eqref{eq:length-density}, it suffices to consider the case $f(\ell)=1_{a<\ell<b}$ with $0<a<b$. In this case, we have 
	\begin{align*}
	\LF_\bbH^{(\alpha,i)} [f(\nu_\phi(\R)) ] &=  \E\left[\int_{0}^\infty 1_{e^{\frac{\gamma c}{2}} \nu_{\wt \phi}(\R) \in (a, b)}2^{-\frac{\alpha^2}2} e^{(\alpha - Q)c} \, dc\right]\\ &= \E\left[\int_a^{b} \nu_{\wt \phi} (\R)^{\frac2\gamma(Q-\alpha)} 2^{-\frac{\alpha^2}2} \ell^{\frac2\gamma (\alpha - Q)} \cdot \frac2\gamma \ell^{-1} \, d\ell\right], 
	\end{align*}
	where we have done the change of variables $\ell= e^{\frac\gamma2c} \nu_{\wt \phi}(\R)$. Since $\ol U_0(\alpha)= \E[\nu_{\wt \phi}(\R)^{\frac2\gamma (Q-\alpha)}]$, by interchanging the integral and expectation we get~\eqref{eq:length-density}.
	
	Note that \(\left \langle  e^{\alpha \phi(i)}  \right \rangle_{\gamma, 0,\mu_B}= \LF_\bbH^{(\alpha, i)}[e^{ - \mu_B \nu_\phi(\R)} -1]= \int_{0}^\infty (e^{ - \mu_B \ell}  -1 ) \frac2\gamma 2^{-\frac{\alpha^2}2} \ol U_0(\alpha)\ell^{\frac2\gamma(\alpha-Q)-1} \, d\ell\). Now~\eqref{eq:U0} follows from 
	\[
	\int_{0}^\infty (e^{ - \mu_B \ell}  -1 ) \ell^{\frac2\gamma(\alpha-Q)-1} \, d\ell
	=  \mu_B^{\frac{2}{\gamma}(Q -\alpha)} \Gamma \left(\frac{2}{\gamma}(\alpha - Q) \right),
	\]which is a simple application of integration by parts to the definition of the $\Gamma$ function.
\end{proof}

The following explicit expression of $\ol U_0(\alpha)$ was proven in \cite{remy-fb-formula}.
\begin{proposition}[{\cite{remy-fb-formula}}]\label{prop-remy-U}
	For $\alpha > \frac{\gamma}{2}$ and $\ol U_0(\alpha)$  defined as in  Lemma~\ref{lem-len-LF}, we have
	\eqb\label{eq:U0-explicit}
	\ol U_0(\alpha) = \left( \frac{2^{-\frac{\gamma\alpha}2} 2\pi}{\Gamma(1-\frac{\gamma^2}4)} \right)^{\frac2\gamma(Q-\alpha)} 
	\Gamma( \frac{\gamma\alpha}2-\frac{\gamma^2}4)\quad \textrm{for all }\alpha>\frac{\gamma}2.
	\eqe
\end{proposition}
\begin{proof}
	This formula is proved in~\cite{remy-fb-formula} in a different appearance, namely the disk domain $\mathbb{D}$ is used instead of $\mathbb{H}$. To recover our setup, consider the following calculation:
	\begin{align*}
	\E[ \nu_{\wt \phi}(\R)^{\frac2\gamma(Q-\alpha)}] &= \lim_{\eps \rightarrow 0} \E \left[ \left( \int_{\mathbb{R}} \eps^{\frac{\gamma^2}{4}} e^{\frac{\gamma}{2} h_{\eps}(z)} e^{ \frac{\gamma}{2} (- 2 Q \log |z|_+ + \alpha G_{\mathbb{H}}(z,i) ) } dz \right)^{\frac2\gamma(Q-\alpha)} \right] \\
	&= \lim_{\eps \rightarrow 0} \E \left[ \left( \int_{\mathbb{R}} e^{\frac{\gamma}{2} h_{\eps}(z) - \frac{\gamma^2}{4} \E[ h_{\eps}(z)^2]} e^{ \frac{\gamma}{2} (- \frac{4}{\gamma} \log |z|_+ + \alpha G_{\mathbb{H}}(z,i) ) } dz \right)^{\frac2\gamma(Q-\alpha)} \right].
	\end{align*}
This last equation then gives equation (1.15) in \cite[Definition 1.5]{RZ_boundary} and \eqref{eq:U0-explicit} is simply \cite[Theorem 1.6]{RZ_boundary}.	
\end{proof}

\subsection{Quantum surface and quantum disks}\label{subsec:QD}
In this section we review the definition of a few variants of quantum disks that will be used in our paper.
Let $\cD\cH = \{ (D,h)\: : \: D \subset \C \text{ open, } h \text{ a distribution on }D\}$. We define an equivalence relation on $\cD \cH$ by saying $(D, h) \sim_\gamma (\wt D, \wt h)$ if there is a conformal map $\psi: D \to \wt D$ such that $\wt h = \psi \bullet_\gamma h$, where
\eqb\label{eq-QS}
\psi \bullet_\gamma h := h \circ \psi^{-1} + Q \log |(\psi^{-1})'|. 
\eqe

A \emph{quantum surface} is an equivalence class of pairs $(D, h) \in \cD \cH$ under the equivalence relation $\sim_\gamma$, 
where $D$ is a disk domain and $h$ is a distribution on $D$ (i.e.\ $h\in C_0^\infty(D)'$).
An \emph{embeddding} of a quantum surface is a choice of representative $(D,h)$.  
Consider tuples $(D, h, z_1,\cdots,z_m, w_1,\cdots, w_n )$
$z_i \in D $ and $ w_j \in \partial D$.
We say 
$$(D, h, z_1,\cdots,z_m, w_1,\cdots, w_n )  \sim_\gamma (\wt D, \wt h, \wt z_1,\cdots,\wt z_m, \wt w_1,\cdots, \wt w_n )$$
if there is a conformal map $\psi:  D \to \wt D$ 
such that~\eqref{eq-QS} holds  and $\psi( z_i) = \wt z_i, \psi(w_j) = \wt w_j$. 
Let $\mathfrak D_{m,n}$ be  the set of equivalent classes of such tuples under $\sim_\gamma$.
We write $\mathfrak D_{0,0}$ as $\mathfrak D$ for simplicity. 
The reason we define quantum surface using \eqref{eq-QS} is because the $\gamma$-LQG quantum area  $\mu_{\wt h}$ is the pushforward of $\mu_h$. 
The same holds for the quantum length measure as long as it is well defined.

The set $\mathfrak D_{m,n}$ can be viewed as the quotient space of 
$$\{(h, z_1,\cdots,z_m, w_1,\cdots, w_n) : h \textrm{ is  a distribution on }\bbH, z_1,\cdots,z_m\in \bbH,  w_1,\cdots, w_n\in \R\cup\{\infty\} \}$$ 
under $\sim_\gamma$. Therefore the Borel  $\sigma$-algebra of $H^{-1}(\bbH)$ induces a $\sigma$-algebra on $\mathfrak D_{m,n}$.
Moreover, a random distribution on $\bbH$ such as a variant of the GFF induces a random variable valued in $\mathfrak D_{m,n}$.

We now define the 2-pointed quantum disk  introduced in \cite[Section 4.5]{wedges}, which is a family of measures on $\mathfrak D_{0,2}$.
It is most convenient to  describe it using the horizontal strip $\cS=\R\times (0,\pi)$. Let $\exp:\cS\to \bbH$ be the exponential map $z\mapsto e^z$.
Let  $h_\cS=h_\bbH \circ \exp$ where $h_\bbH$ is sampled from $P_\bbH$.
We call  $h_\cS$ be a free-boundary GFF on $\cS$. Its covariance kernel is given by  $G_\cS(z,w) =G_\bbH (e^z,e^w)$.
The field $h_\cS$ can be written as $h_\cS=h^{\op c}+h^{\ell}$, where 
$h^{\op c}$ is constant on vertical lines of the form $u + [0,i\pi]$ for $u\in \R$, and $h^{\ell}$  has mean zero on all such
lines~\cite[Section 4.1.6]{wedges}. We call $h^{\ell}$ the \emph{lateral component} of the free-boundary GFF on $\cS$.
\begin{definition}
	\label{def-thick-disk} 
	For $W \geq \frac{\gamma^2}2$, let $\beta =  Q + \frac\gamma2  - \frac W\gamma$. Let 
	\[Y_t =
	\left\{
	\begin{array}{ll}
	B_{2t} - (Q -\beta)t  & \mbox{if } t \geq 0 \\
	\wt B_{-2t} +(Q-\beta) t & \mbox{if } t < 0
	\end{array}
	\right. , \]
	where $(B_s)_{s \geq 0}, (\wt B_s)_{s \geq 0}$ are independent standard Brownian motions conditioned on $B_{2s} - (Q-\beta)s<0$ and $\wt B_{2s} - (Q-\beta)s < 0$ for all $s>0$.\footnote{Here we condition on a zero probability event. This can be made sense of via a limiting procedure.}
	Let  $h^1(z)=Y_t$ for each $z\in \cS$ and $t\in \R$ with $\op{Re}z=t$.  
	Let $h^2$ be a random distribution on $\cS$ independent of $Y_t$ which has the law of the lateral component $h^{\ell}$ of the free-boundary GFF on $\cS$.
	Let  $\mathbf c$ be a real number  sampled from $\frac\gamma2 e^{(\beta-Q)c}dc$ independent of $(h^1,h^2)$ and $\phi=h^1+h^2+\mathbf c$.
	Let $\cMtwo(W)$ be the infinite measure on $\mathfrak D_{0,2}$ describing the law of $(\cS, \phi, -\infty, +\infty)$.
	We call a sample from $\cMtwo(W)$ a \emph{weight-$W$ quantum disk}.
\end{definition}

The weight 2 quantum disk $\cMtwo(2)$ is special because its two marked points are typical with respect to the quantum boundary length measure \cite[Proposition A.8]{wedges}; see Proposition~\ref{prop-QD2-field}. 
Based on this  we can define the  family of quantum disks marked with quantum typical points.  We will use our convention that $M^{\#}=|M|^{-1}M$.
\begin{definition}\label{def-QD}
	Let  $(\cS, \phi, +\infty,-\infty) $ be the embedding of a sample from $\cMtwo(2)$ as in Definition~\ref{def-thick-disk}. Let $A=\mu_\phi(\cS)$ and $L=\nu_\phi(\partial \cS)$.
	Let $\QD$ be the law of $(\cS, \phi) $ under the reweighted measure 
	$L^{-2}\cMtwo(2)$, viewed as a measure on $\mathfrak D$. 
	For non-negative integers $m,n$,  let $(\cS,\phi)$ be a sample from $A^mL^n\QD$, and then
	independently sample $z_1,\cdots, z_m$ and $w_1,\cdots, w_n$ according to $\mu_\phi^\#$ and $\nu_\phi^\#$,  respectively.
	Let $\QD_{m,n}$ be  the law of $(\cS, \phi, z_1,\cdots, z_m, w_1,\cdots, w_n)$  viewed as a measure on  $\mathfrak D_{m,n}$.
	We  call a sample from $\QD_{m,n}$ \emph{a quantum disk with $m$ interior and $n$ boundary marked points}. 
\end{definition}

\begin{proposition}[{Proposition A.8 of \cite{wedges}}] \label{prop-QD2-field}
	We have $\cMtwo(2) = \QD_{0,2}$.
\end{proposition}

For $W \ge \gamma^2/2$, we can define the family of finite  measures $\{ \cMtwo(W; \ell, \ell')\}_{\ell, \ell'>0}$ such that $\cMtwo(W; \ell, \ell')$ is supported on quantum surfaces with left and right boundary arcs having quantum lengths $\ell$ and $\ell'$, respectively and 
\eqb \label{eq-disint-W}
\cMtwo(W) = \iint_0^\infty \cMtwo(W; \ell, \ell') \,d \ell\, d \ell'.
\eqe
In words,  $|\cMtwo(W; \ell, \ell')|d\ell\, d\ell'$ describes the distribution of the left and right boundary lengths of a sample from $\cMtwo(W)$, 
and $\cMtwo(W; \ell, \ell')^{\#}$ is the probability measure obtain by conditioning  $\cMtwo(W)$ on specific boundary length values. 
The general theory of disintegration~\eqref{eq-disint-W}  only specifies $\cMtwo(W; \ell, \ell')$ for almost every  $(\ell, \ell')$. In \cite[Section 2.6]{ahs-disk-welding}  this ambiguity is removed by introducing a suitable topology  for which $\cMtwo(W, \ell, \ell')$ is continuous in $\ell,\ell'$.

We can also define the measure $\QD_{m,n}(\ell)$ on $\mathfrak D_{m,n}$ which corresponds to restricting $\QD_{m,n}$ to the event that  the boundary length is $\ell$.  
Since $\cMtwo(2)=\QD_{0,2}$, we set
\[\QD_{0,2}(\ell)=\int_0^\ell \cMtwo(2; x,\ell-x)\, dx.\] 
From here  $\QD_{m,n}(\ell)$ for general $m,n$ 
can be   specified by Definition~\ref{def-QD} and the requirement that
\begin{equation}\label{eq:disint-QD}
\QD_{m,n}= \int_0^\infty \QD_{m,n}(\ell)  \,d \ell.
\end{equation}

If we ignore the boundary marked points of a sample from $\QD_{0,n}(\ell)$, its law is given by $\ell^n\QD(\ell)$. Therefore, ignoring boundary marked points, the probability measure $\QD_{0,n}(\ell)^\#$ does not depend on $n$ and agrees with $\QD(\ell)$. The following theorem gives us a precise way to specify the mating-of-trees variance  $\BB a^2$ in terms of $ \QD(1)^{\#}$, which we will use to evaluate $\BB a^2$.
\begin{theorem}[{\cite[Theorem 1.2]{ag-disk}}]\label{thm: intro area disk}
	The law of the total quantum area of a sample from $\QD(1)^{\#}$  is the inverse gamma distribution with shape parameter $\frac{4}{\gamma^2}$ and scale parameter $\frac{1}{2 \BB a^2 \sin^2(\pi\gamma^2/4) }$ as in~\eqref{eq:inverse-gamma-parameter}, where $\BB a^2$ is the mating-of-trees variance in~\cite{wedges}.
	\end{theorem}

\subsection{SLE and conformal welding of quantum disks}\label{subsec:welding}
We now review the conformal welding result proved in \cite{ahs-disk-welding}. We need an important variant of SLE called $\SLE_\kappa(\rho_-;\rho_+)$, introduced in~\cite{lsw-restriction} and  studied in e.g.\ \cite{dubedat-rho,ig1}. 
The parameter range relevant for us is $\kappa\in (0,4)$, $\rho_-\ge \frac{\kappa}{2}-2$ and $\rho_+\ge \frac{\kappa}{2}-2$. 
In this range,  the $\SLE_\kappa(\rho_-;\rho_+)$ on $(\bbH, 0,\infty)$ is a probability measure on simple curves on $\bbH$ from $0$ to $\infty$, which does not touch $\bdy\bbH$
except at the endpoints. For a general simply connected domain $D$ with boundary points $a$ and $b$, let $\psi: \bbH \rta D$ be a conformal map such that $\psi(0)=a$ and
$\psi(\infty)=b$. The $\SLE_\kappa(\rho_-;\rho_+)$  on $(D,a,b)$ is defined as the pushforward by $\psi$ of the $\SLE_\kappa(\rho_-;\rho_+)$ on $(\bbH, 0,\infty)$. 
Although there is a degree of freedom in choosing $\psi$, this definition is independent of such choices. 
We omit the definition of  $\SLE_\kappa(\rho_-;\rho_+)$ via the Loewner equation as it will not be needed.

For $W\ge \gamma^2/2$, recall  $\cMtwo(W;\ell, \ell')$ from~\eqref{eq-disint-W}. 
For a fixed $\ell>0$, let
\begin{equation}\label{eq:gamma^22}
\cMtwo(W;\cdot, \ell')=\int_0^\infty \cMtwo(W;\ell, \ell') \,d\ell \textrm{ and }\cMtwo(W; \ell,\cdot)=\int_0^\infty \cMtwo(W;\ell, \ell') \,d\ell'.
\end{equation}
Then we have $\cM_{0,2}^\disk(W)=\int_0^\infty\cM_{0,2}^\disk(W;\ell, \cdot)\,d\ell$ hence $\cM_{0,2}^\disk(W;\ell,\cdot)$ is the disintegration of $\cM_{0,2}^\disk(W)$ over the left boundary length. The same hold for $\cM_{0,2}^\disk(W;\cdot,\ell)$ with right boundary instead. 
For $W_-$ and $W_+\in [\gamma^2/2,\infty)$, we will consider the following measure
\begin{equation}\label{eq:two-disks}
\int_{0}^\infty \cM_{0,2}^\disk(W_-; \cdot, \ell)\times \cM_{0,2}^\disk(W_+;\ell, \cdot)\, d\ell.
\end{equation}
For a fixed $\ell>0$,  $\cM_{0,2}^\disk(W_-; \cdot, \ell)\times \cM_{0,2}^\disk(W_+;\ell, \cdot)$ is a product measure on 
$\mathfrak D_{0,2}\times \mathfrak D_{0,2}$. 
Therefore the integration in~\eqref{eq:two-disks} gives another  measure  on $\mathfrak D_{0,2}\times \mathfrak D_{0,2}$.
\begin{theorem}[{Theorem 2.2 of \cite{ahs-disk-welding}}]\label{thm:disk-welding} 
	Let $\gamma\in (0,2)$ and $\kappa=\gamma^2$. For $W_-$ and $W_+\in [\gamma^2/2,\infty)$, let $W=W_-+W_+$, $\rho_{-}=W_- -2$ and $\rho_+=W_+-2$.
	Suppose $(\bbH,\phi,0,\infty)$ is an embedding of a sample from  $\cM_{0,2}^\disk(W)$ given by Definition~\ref{def-thick-disk}.
	Let $\eta$ be a $\SLE_\kappa(\rho_-;\rho_+)$ curve on $(\bbH, 0,+\infty)$  independent of $\phi$. 
	Let $\bbH_\eta^-$ and  $\bbH_\eta^+$ be the connected components of $\bbH\setminus \eta$ on the left and right side of $\eta$ respectively.
	There exists $C \in (0,\infty)$ such that the law of $(\bbH_\eta^-,\phi,0,\infty)$  and  $(\bbH_\eta^+,\phi,0, \infty)$ 
	viewed as a pair of elements in $\mathfrak D_{0,2}$ equals
	\begin{equation}\label{eq:cut-disk}
	C\int_{0}^\infty \cM_{0,2}^\disk(W_-; \cdot, \ell)\times \cM_{0,2}^\disk(W_+;\ell, \cdot)\, d\ell.
	\end{equation}
\end{theorem}

If we are only given the information  of $(\bbH_\eta^-,\phi,0,\infty)$  and  $(\bbH_\eta^+,\phi,0, \infty)$ as marked quantum surfaces
we can identify the right boundary of the former to the latter to get a curve-decorated topological disk, where on the complement of the curve there is a conformal structure. 
From the classical work of Sheffield~\cite{shef-zipper},outside of a measure-zero event, this uniquely determines a conformal structure on the entire disk,  under which we get $(\bbH,\phi,  \eta, 0,\infty)$ modulo the equivalence relation $\sim_\gamma$. This recovering procedure is called \emph{conformal welding}. 
The reason for this uniqueness is a property satisfied by  $\SLE_\kappa(\rho_-;\rho_+)$ almost surely called the conformal removability.
For more details on Sheffield's work and conformal welding, see~\cite[Section 6]{berestycki-lqg-notes}.

In words, Theorem~\ref{thm:disk-welding} says that 
modulo a multiplicative constant, conformally welding $\cM_{0,2}^\disk(W_-)$  and $\cM_{0,2}^\disk(W_-)$,
we get $\cM_{0,2}^\disk(W)$  decorated with an independent $\SLE_\kappa(\rho_-;\rho_+)$. 

\subsection{Quantum disks of weight $2$ and $\gamma^2/2$}\label{subsec:special disks}
Our proof of the FZZ formula relies on the conformal welding of quantum disks with weight $W\in \{2,  \gamma^2/2 \}$.
For these two weights the mating-of-trees theory  for  quantum disks \cite{sphere-constructions,ag-disk,wedges}
 allow us to  describe the quantum area and quantum  length distributions of  $\cMtwo(W)$  in terms of Brownian motion. 
The case $W=2$ is essentially Theorem~\ref{thm: intro area disk}. The case $W=\gamma^2/2$ is proved in~\cite{ahs-disk-welding}, which we now review. 

 For $\theta \in (0,2\pi)$, let $\cC_\theta := \{z \: : \: \arg(z) \in (0, \theta)\}$ be the cone with angle $\theta$. 
For $z \in \cC_\theta$, let $\sm_{\cC_\theta}(z)$ denote the probability measure corresponding to Brownian motion started at $z$ and killed when it exits $\cC_\theta$. For $y >0$, let $E_{y, \eps}$ be the event that BM exits $\cC_\theta$ on the boundary interval $(ye^{i\theta}, (y+\eps)e^{i\theta})$, and let $\sm_{\cC_\theta}(z,ye^{i\theta}) = \lim_{\eps \to 0} \eps^{-1} \sm_{\cC_\theta}(z)|_{E_{y, \eps}}$. For $x > 0$ define $\sm_{\cC_\theta}(x,ye^{i\theta}) = \lim_{\eps \to 0} \eps^{-1}\sm_{\cC_\theta}(x + \eps i, ye^{i\theta})$. For more details on these limits see Appendix~\ref{sec:BM-cone}.
The following result  from~\cite{ahs-disk-welding} describes the joint law of the quantum boundary lengths and quantum area of the weight $\frac{\gamma^2}2$ quantum disk, in terms of the measure $\sm_{\cC_\theta}(x, ye^{i\theta})$ with $\theta = \frac{\pi \gamma^2}4$.

\begin{proposition}[{\cite{ahs-disk-welding}}]\label{lem-mot}
Let $\theta = \frac{\pi \gamma^2}4$ and $u = \frac1{\mathbbm a \sin \theta}$. There is a constant $C \in (0, \infty)$ such that for all $\ell, r>0$ we have $\left|\cMtwo(\frac{\gamma^2}2; \ell, r)\right|= C \left| \sm_{\cC_\theta}(\frac\ell u, \frac ru e^{i\theta})\right|$.   Moreover, the quantum area of a sample from $\cMtwo(\frac{\gamma^2}2; \ell, r)^\#$  agrees in law with the duration of a path sampled from $\sm_{\cC_\theta}(\frac\ell u, \frac ru e^{i\theta})^\#$.
\end{proposition}

\begin{proof}
 	Consider the shear transformation $\Lambda = \mathbbm{a} \begin{pmatrix} \sin \theta & - \cos \theta \\ 0 & 1 \end{pmatrix}$ that maps $\cC_\theta$ to $\R_+^2$. It sends a standard 2D Brownian motion to a Brownian motion with covariance $\mathbbm a^2 \begin{pmatrix} 1 & -\cos \theta \\ -\cos \theta & 1 \end{pmatrix}$, and maps $\frac{\ell}u \mapsto \ell$ and $\frac ru e^{i\theta} \mapsto ri$. Let $\mu^{\gamma}_{\R_+^2}(\ell, ri)$ be the law of $\wt Z_t  := \Lambda Z_t$ where $t\mapsto Z_t$ 
 	is sampled from the  path measure $\sm_{\cC_\theta}(\frac\ell u, \frac ru e^{i\theta})$. It is proved in 
\cite[Proposition 7.7]{ahs-disk-welding} that for some constant $C$ and all $\ell, r > 0$ we have $\left|\cMtwo(\frac{\gamma^2}2; \ell, r) \right| = C \left| \mu_{\R_+^2}^\gamma (\ell, r i) \right|$, and that the quantum area of a sample from $ \cMtwo(\frac{\gamma^2}2; \ell, r)^\#$ agrees in law with the duration of a sample from $\mu_{\R_+^2}^\gamma (\ell, r i)^\#$. Transforming back by $\Lambda^{-1}$, we conclude the proof.
\end{proof}

We will also need the length distributions of the weight $\frac{\gamma^2}2$ and weight 2 quantum disk. 
\begin{lemma}[{\cite[Propositions 7.7 and 7.8]{ahs-disk-welding}}]\label{lem:lengh}
	There are constants $C_1,C_2 \in (0,\infty)$ such that
	\eqb\label{eq-MOT-length-original}
	|\cMtwo(\frac{\gamma^2}2;\ell, r)| = C_1 \frac{(\ell r)^{4/\gamma^2-1}}{(\ell^{4/\gamma^2} + r^{4/\gamma^2})^2}, \quad \textrm{and}\quad  |\cMtwo(2;\ell, r)| = C_2(\ell + r)^{-4/\gamma^2-1}.
	\eqe
\end{lemma} 

%% file: sec3.tex
\section{Embeddings of quantum disks and the  mating-of-trees variance}\label{sec:variance}

In this section we prove Theorem~\ref{thm:MOT-var}, which says $\BB a^2= \frac2{\sin (\frac{\pi \gamma^2}4)}$. Our starting point is the following observation, which expresses $\BB a^2$ in terms of the quantum  length distribution of $\QD_{1,0}$ and $\QD_{0,2}$.

\begin{lemma}\label{lem:QD-ratio}
	Recall Definition~\ref{def-QD} and~\eqref{eq:disint-QD}. We have
	\begin{equation}\label{eq:exp-area}
		\BB a^2=\frac{|\QD_{0,2}(\ell)|  }{(\frac{4}{\gamma^2}-1) 2 \sin^2(\pi\gamma^2/4) |\QD_{1,0}(\ell)|} \quad \textrm{for each }\ell>0.
	\end{equation} 
\end{lemma}
\begin{proof}
	By Definition~\ref{def-QD}, for a non-negative measurable  function $f$ on $\R$ we have
	\begin{equation}\label{eq:11-length}
		\QD_{1,1}[f(L)]= \QD_{0,1}[f(L)A] =\int_0^\infty f(\ell)  \QD_{0,1}(\ell)[A]d\ell.
	\end{equation}
	By~\eqref{eq:disint-QD} we have $\QD_{1,1}[f(L)]=\int_0^\infty f(\ell) |\QD_{1,1}(\ell)| d\ell$. Therefore $|\QD_{1,1}(\ell)|=  \QD_{0,1}(\ell)[A]$.

	Note that $ \QD_{0,1}(\ell)[A]= |\QD_{0,1}(\ell)|\QD_{0,1}(\ell)^{\#}[A]$. By the scaling property of the quantum disk, we have $\QD_{0,1}(\ell)^{\#}[A]=\ell^2\QD_{0,1}(1)^{\#}[A]$, which equals $\ell^2\QD(1)^{\#}[A]$ since $\QD_{0,1}(1)^{\#}=\QD(1)^{\#}$. 
	By Theorem~\ref{thm: intro area disk} and~\eqref{eq:inverse-gamma-parameter}, we have  $\QD(1)^{\#}[A]=\frac{b}{a-1}$ where $a = \frac4{\gamma^2}$ and $b = (2\mathbbm a^2 \sin^2 (\pi \gamma^2/4))^{-1}$. Therefore 
	\begin{equation}\label{eq:01-length}
		|\QD_{1,1}(\ell)|=\QD_{0,1}(\ell)[A]=|\QD_{0,1}(\ell)| \ell^2  \frac{1}{2 \BB a^2 \sin^2(\pi\gamma^2/4)}. \frac{1}{\frac{4}{\gamma^2}-1}.
	\end{equation}
	On the other hand,  $|\QD_{1,1}(\ell)|= |\QD_{1,0}(\ell)|\ell$ and $|\QD_{0,2}(\ell)| = |\QD_{0,1}|\ell$. By re-arranging~\eqref{eq:01-length} we conclude the proof.
\end{proof}

The next lemma gives the explicit evaluation of $|\QD_{0,2}(\ell)| $.
\begin{lemma}\label{lem-disk-perim-law}
For $\ell>0$, we have 	\( |\QD_{0,2}(\ell)|=  \ol R(\gamma;1,1) \ell^{-\frac4{\gamma^2} }  \), where 
\begin{align}
\label{eq-R}
\ol R(\gamma; 1,1) =\frac{(2\pi)^{\frac4{\gamma^2}-1}}{(1-\frac{\gamma^2}4)\Gamma(1-\frac{\gamma^2}4)^{\frac4{\gamma^2}}}.
\end{align}
\end{lemma}
\begin{proof}

Recall that $\QD_{0,2}=\cM^{\disk}_{0,2}(W)$ with $W=2$ by Definition~\ref{def-QD}.
Let $L_1$ and $L_2$ be the left and right boundary lengths of a sample from 
$\cM_{0,2}^\disk(2)$. By~\cite[Lemma 3.3]{AHS-SLE-integrability}, for $\mu_1, \mu_2 > 0$, the law of $\mu_1 L_1 + \mu_2 L_2 $ is 
 $1_{\ell>0}\ol R(\gamma; \mu_1, \mu_2) \ell^{-\frac4{\gamma^2}} d\ell$ where $\ol R(\gamma; \mu_1, \mu_2)$ is explicitly computed by the second  author of this paper  and Zhu in~\cite{RZ_boundary}. 	
 Setting now $\mu_1 = \mu_2 = 1$, since the density of $L_1+L_2$ can also be written as $ 1_{\ell>0}|\QD_{0,2}(\ell)| d\ell $ one obtains $ |\QD_{0,2}(\ell)|=  \ol R(\gamma;1,1) \ell^{-\frac4{\gamma^2} }$.
\end{proof}
\begin{remark}\label{rmk:refl}
The function $\ol R(\gamma;1,1)$ in Lemma~\ref{lem-disk-perim-law} is a  boundary variant of the reflection coefficient $\ol R(\gamma)$ considered  in~\cite{rv-tail, DOZZ_proof}.
\end{remark}
To get $\BB a^2$, it remains to  evaluate $|\QD_{1,0}|$. It  follows from the LCFT representation of $\QD_{1,0}$.
\begin{theorem}\label{thm:1disk-const}
	Let $\phi$ be a sample of $\LF_{\bbH}^{(\gamma,i)}$. Then the law of $(\bbH, \phi, i)$ viewed as  a marked quantum surface is $\frac{2\pi (Q-\gamma)^2}\gamma \QD_{1,0}$.
\end{theorem}

Cercle~\cite{cercle-quantum-disk} proved that $\QD_{0,3}(1)^{\#}$ admits a simple LCFT description, which extends the result in~\cite{ahs-sphere} for  the  quantum sphere. In~\cite[Section 2]{AHS-SLE-integrability}, a family of such results was proved in a systematic way. We will prove Theorem~\ref{thm:1disk-const} based on ideas and results from~\cite[Section~2]{AHS-SLE-integrability}. We postpone this proof to Section~\ref{subsec:embedding} and proceed to wrap up the evaluation of $\BB a^2$.

\begin{proof}[Proof of Theorem~\ref{thm:MOT-var} given Theorem~\ref{thm:1disk-const}] By~\eqref{eq:length-density} with $\alpha=\gamma$ and Theorem~\ref{thm:1disk-const}, we have
	\[|\QD_{1,0} (\ell)|=
	\frac\gamma{2\pi(Q-\gamma)^2}\times  \frac2\gamma 2^{-\frac{\gamma^2}2}\ol U_0(\gamma)\ell^{\frac2\gamma(\gamma-Q)-1} = 
	\frac{\ol U_0(\gamma) 
		\ell^{- \frac{4}{\gamma^2}}}  {2^{\gamma^2/2} \pi (Q-\gamma)^2}.
		\]
	Now by Lemmas~\ref{lem:QD-ratio} and~\ref{lem-disk-perim-law}, and plugging in~\eqref{eq:U0-explicit} and~\eqref{eq-R} and simplifying, we have
	\[
	\BB a^2= \frac{\ol R(\gamma;1,1)}{\ol U_0(\gamma)} \times \frac{2^{\frac{\gamma^2}2-1} \pi (Q-\gamma)^2}{(\frac4{\gamma^2}-1)  \sin^2(\pi\gamma^2/4)} = \frac{2\pi}{\Gamma(\frac{\gamma^2}4)\Gamma(1-\frac{\gamma^2}4)\sin^2(\frac{\pi \gamma^2}4)}.
	\]
	Using the identity $\Gamma(z) \Gamma(1-z) = \frac{\pi}{\sin(\pi z)}$ for $z \not \in \Z$ gives the result. 
\end{proof}

In the rest of this section we first prove Theorem~\ref{thm:1disk-const} in Section~\ref{subsec:embedding} and then prove related results in Section~\ref{subsec:third-point} that will be used in Section~\ref{sec:bubble}. 

\subsection{LCFT descriptions of quantum disks}\label{subsec:embedding}
In this section we prove Theorem~\ref{thm:1disk-const}. It uses the so-called uniform embedding of quantum disks via Haar measures introduced in~\cite{AHS-SLE-integrability}. Before recalling this result, we first give a concrete realization of a Haar measure on the group of conformal automorphisms of $\bbH$. 
Sample $(\mathbf p,\mathbf q,\mathbf r)$ from the measure $|(p-q)(q-r)(r-p)|^{-1} \, dp \, dq \, dr$ restricted to triples $(p,q,r)$ which are oriented counterclockwise on $ \R = \partial \bbH$. Let $\mathfrak g$ be the conformal map with $\mathfrak g(0) = \mathbf p, \mathfrak g(1) = \mathbf q, \mathfrak g(-1) = \mathbf r$, 
and let $\mathbf m_\bbH$ be the law of $\mathfrak g$.  We recall the notation $\mathfrak g \bullet_\gamma \phi = \phi \circ\mathfrak  g^{-1} + Q \log \left|(\mathfrak g^{-1})'\right|$ from~\eqref{eq-QS}. 
As explained in~\cite{AHS-SLE-integrability}, $\mathbf m_\bbH$ is a (left and right invariant) Haar measure on the conformal automorphism group $\mathrm{conf}(\bbH)$ of $\bbH$.
\begin{proposition}[Proposition 2.39 of \cite{AHS-SLE-integrability}]\label{prop-unmarked-disk-embed}
	Let $M$ be a measure on fields $\phi$ such that the law of $(\bbH, \phi)$ is $\QD$. If we sample $(\phi, \mathfrak g)$ from $M \times \mathbf{m}_\bbH$,  then the law of the field $\mathfrak g \bullet_\gamma \phi$ is 
	\[\frac\gamma{2(Q-\gamma)^2} \LF_\bbH.\]
\end{proposition}

Due to the invariance of the Haar measure, the law of $\mathfrak g \bullet_\gamma \phi$ does not depend on the choice of $M$. This is called a uniform embedding of $\QD$ via $\mathbf m_\bbH$ in~\cite[Theorem 1.2]{AHS-SLE-integrability}.

We will need to following basic fact on $\mathbf m_\bbH$ in the proof of Theorem~\ref{thm:1disk-const}.
\begin{lemma}\label{lem-haar-density-i}
	For $\mathfrak g$ sampled from $\mathbf m_\bbH$, the law of $\mathfrak g(i) \in \bbH$ is $\frac{\pi}{\left|\Im z\right|^2} d^2z$. 
\end{lemma}
We first give a representation of the Haar measure $m_\bbH$ where the density of $\mathfrak g(i)$ is transparent. 
\begin{lemma}\label{lem:ank}
Define three measures $A, N, K$ on the conformal automorphism group $\mathrm{conf}(\bbH)$ on $\bbH$  as follows. Sample $\mathbf  t$ from $\1_{t >0} \frac1t \, dt$ and set $a: z \mapsto \mathbf t z$, sample $\mathbf s$ from the Lebesgue measure on $\R$ and set $n : z \mapsto z + \mathbf s$, and sample $\mathbf  u$ from $\1_{-\frac\pi2 < u < \frac\pi2} du$ and set $k: z \mapsto \frac{z \cos \mathbf  u - \sin\mathbf   u}{z \sin \mathbf  u + \cos\mathbf   u}$. Let $A,N,K$ be the laws of $a,n,k$ respectively. 
Then the law of $ a \circ n \circ k$  under  $A\times N\times K$  equals $\mathbf m_\bbH$.
\end{lemma}
\begin{proof}
	Set $\mathfrak g=a \circ n \circ k$.
	Then the law of $ \mathfrak g$  under  $A\times N\times K$   is a Haar measure on $\mathrm{conf}(\bbH)$; see e.g.\ \cite[Theorem 11.1.3]{deitmar-echterhoff-harmonic-analysis}.
	By  the uniqueness of the Haar measure (see e.g.\ \cite[Theorem 5.1.1]{faraut-lie-groups}), the law of $\mathfrak g$ is $C \mathbf m_\bbH$ for some $C\in (0,\infty)$.
	We now check that $C=1$.

	Define  $E_\eps := \{\mathfrak g(0) \in (0,1), \mathfrak g(-1) < \mathfrak g(0) < \mathfrak g(1), \frac{\mathfrak g(1)-\mathfrak g(0)}{\mathfrak g(0)-\mathfrak g(-1)} \in (1-\eps, 1+\eps), \mathfrak g(0)-\mathfrak g(-1) \in (1, 2) \}$, then by the definition of $\mathbf m_\bbH$ we get
	\[A\times N\times K [E_\eps] =  C\int_0^1 \int_{p-2}^{p-1} \int_{p+(p-r)(1-\eps)}^{p+(p-r)(1+\eps)} \frac{1}{(p-r)(q-p)(q-r)} \, dq \, dr \, dp= C(1+o_\eps(1))\frac\eps2. \]
	
	Let 
	$\ul{E}_{\eps, \delta} = \{\mathbf  u \in (-(1-\delta)\frac\eps2, (1-\delta)\frac\eps2), \mathbf  t \in (1+\delta,2-\delta), \mathbf s \in (0, \frac1{\mathbf  t})\}$ and $\ol{E}_{\eps, \delta} = \{\mathbf  u \in (-(1+\delta)\frac\eps2, (1+\delta)\frac\eps2), \mathbf  t \in (1-\delta,2+\delta), \mathbf s \in (0, \frac1{\mathbf  t})\}$. Since $\frac{z\cos u - \sin u}{z \sin u + \cos u} = z - u - uz^2 + O(u^2)$ as $u \to 0$ with error uniform for $z \in \{-1,0,1\}$, it is easy to check that for fixed $\delta$ and sufficiently small $\eps$ we have $\ul{E}_{\eps, \delta} \subset E_\eps \subset \ol{E}_{\eps, \delta}$, and
	\begin{align*}
	A\times N \times K[\ul{E}_{\eps, \delta}] &= (1+\delta)\eps \int_{2-\delta}^{1-\delta} \frac1{t^2}\, dt = (1-\delta)(\frac1{1+\delta} - \frac1{2-\delta})\eps, \\
	\quad A\times N \times K[\ol{E}_{\eps, \delta}]  &= (1+\delta)(\frac1{1-\delta} - \frac1{2+\delta})\eps.
	\end{align*}
Sending $\eps \to 0$ we have  $\frac{C}2\in [ (1-\delta)(\frac1{1+\delta} - \frac1{2-\delta}), (1+\delta)(\frac1{1-\delta} - \frac1{2+\delta})]$. Sending $\delta \to 0$, 
 we get $C=1$.	
\end{proof}

\begin{proof}[Proof of Lemma~\ref{lem-haar-density-i}]
Under the change of variable $x =st$ and $y = t$, we have  $\frac1{y^2} dx \, dy = \frac1t ds\, dt$. By the definition of $A$ and $N$ in Lemma~\ref{lem:ank}, we see that 
the law of $a(n(i)) = \mathbf{st} + \mathbf ti$ is $\frac1{\left|\Im z\right|^2} d^2z$ if $(a,n)$ is sampled from $A \times N$.
Since $k$ fixes $i$ and $|K| = \pi$, the law of $ a \circ n \circ k(i)$ under $A\times N\times K$ is $\frac{\pi}{\left|\Im z\right|^2} d^2z$.
\end{proof}

\begin{proof}[Proof of Theorem~\ref{thm:1disk-const}]
	We first show that: 
	\eqb\label{eq-interchange}
	\LF_\bbH(d\phi) \mu_\phi(d^2z) =  \LF_\bbH^{(\gamma, z)}(d\phi)\, d^2z. 	
	\eqe
	Let $h$ be sampled from $P_\bbH$.  Note that: 
	\[\E[\mu_h(d^2z)] = (2\Im z)^{-\gamma^2/2} |z|_+^{2\gamma^2} d^2z.\]
	More precisely, for any Borel set $A$ we have $\E[\mu_h(A)] = \int_A(2\Im z)^{-\gamma^2/2} |z|_+^{2\gamma^2}  d^2z$.
	 Then by Girsanov's theorem (see e.g.~\cite[Section 3.5]{berestycki-lqg-notes}), for any nonnegative measurable functions $f$ on $H^{-1}(\bbH)$ and $g$ on $\bbH$ we have 
	\eqb\label{eq-girsanov}
	\int f(h) \left(\int_\bbH g(z) \mu_h(d^2z) \right) P_\bbH(dh) = \int_\bbH \int f(h + \gamma G_\bbH(\cdot, z)) P_\bbH(dh) (2\Im z)^{-\gamma^2/2} |z|_+^{2\gamma^2} g(z)  d^2z.
	\eqe
Now,
	 recalling the definition of $\LF_\bbH$ and using $\mu_{h -2Q \log |\cdot|_+ + \mathbf c}(d^2z) = |z|_+^{-2Q\gamma}e^{\gamma c} \mu_h(d^2z)$ gives
	\[\iint f(\phi) g(z) \mu_\phi(d^2z) \LF_\bbH(d\phi) = \iiint f(h - 2Q \log |\cdot|_+ + c) g(z) |z|_+^{-2Q\gamma}  \mu_h(d^2z) P_\bbH(dh) e^{(\gamma - Q)c} dc,  \]
	and applying~\eqref{eq-girsanov}, this is equal to
	\[\iiint f(h+\gamma G_\bbH(\cdot, z) -2Q \log |\cdot|_+ + c)P_\bbH(dh) e^{(\gamma-Q)c}dc\, g(z) (2\Im z)^{-\gamma^2/2} |z|_+^{-2\gamma(Q-\gamma)} d^2z\] 
	which can be rewritten as $\iint f(\phi) g(z) \LF_\bbH^{(\gamma, z)}(d\phi) \, d^2z$. Thus we have shown~\eqref{eq-interchange}. 
	
	Let $M$ be a measure on $H^{-1}(\bbH)$ such that when $\psi$ is sampled from $M$ the law of the marked quantum surface $(\bbH, \psi, i)$ is $\QD_{1,0}$. By Proposition~\ref{prop-unmarked-disk-embed}, if we sample $(\psi, \mathfrak g)$ from $M \times \mathbf m_\bbH$, then the law of the pair $(\mathfrak g \bullet_\gamma \psi, \mathfrak g (i))$ is the law of $(\phi,z)$ under $\frac{\gamma}{2(Q-\gamma)^2} \LF_\bbH(d\phi) \mu_\phi(d^2z)$. 
	By~\eqref{eq-interchange} this  equals 
	\eqb\label{eq-field-11-z}
	\frac{\gamma}{2(Q-\gamma)^2} \LF_\bbH^{(\gamma, z)}(d\phi)\, d^2z.
	\eqe
	
	Let $S$ be the quantum surface $(\bbH, \mathfrak g \bullet_\gamma \psi, \mathfrak g(i))$. By Lemma~\ref{lem-haar-density-i} the joint law of $(S, \mathfrak g(i))$ is $\QD_{1,0} \times [\frac{\pi}{\Im(z)^2} d^2z]$. Comparing to~\eqref{eq-field-11-z}, we see that if $\phi$ is sampled from $\LF_{\bbH}^{(\gamma,z)}$, the law of $(\bbH, \phi, z)$ viewed as quantum surfaces with a marked point is 
	$\frac{2\pi (Q-\gamma)^2}{\gamma\Im(z)^2} \QD_{1,0}$. 
\end{proof}

We now give an LCFT description of $\QD_{1,1}$, which will be used in the proof of Theorem~\ref{thm:FZZ-physics}. 
	\begin{lemma}\label{lem-def-LF11}
		For $x \in \R$, the vague limit  $\LF_\bbH^{(\gamma, i), (\gamma, x)}:=\lim_{\eps \to 0}\eps^{\gamma^2/4} e^{\frac\gamma2 \phi_\eps(x)} \LF_\bbH^{(\gamma, i)}(d\phi)$ exists. Moreover, if we sample $(h, \mathbf c)$ from $(\frac12 |x|_+ e^{G_\bbH(i,x)})^{\gamma^2/2} P_\bbH\times [e^{(\gamma + \frac\gamma2 - Q)c}dc]$, then the law of $\phi = h + \gamma G_\bbH(\cdot, i) + \frac\gamma2 G_\bbH(\cdot, x) - 2Q \log|\cdot|_+ + \mathbf c$ is $\LF_\bbH^{(\gamma, i), (\gamma, x)}$. 
	\end{lemma}
\begin{proof}
	The proof is identical to that of Lemma~\ref{lem:Girsanov}.
\end{proof}
	
Similarly as in Lemma~\ref{lem-change-coord}, for $x\neq x' \in \R$, $\LF_\bbH^{(\gamma, i), (\gamma, x)}$ and $\LF_\bbH^{(\gamma, i), (\gamma, x')}$ 
are related by  the conformal  coordinate change. Therefore they 
correspond to the same quantum surface modulo a multiplicative constant. The next proposition says that the quantum surface is $\QD_{1,1}$.
	\begin{proposition}\label{prop-embed-QD11}
		Let $x_0 \in \R$ and let $\phi$ be a sample from $\LF_\bbH^{(\gamma, i), (\gamma, x_0)}$. Then the law of the quantum surface $(\bbH, \phi, i, x_0)$ is $C_{x_0} \QD_{1,1}$ for some constant $C_{x_0}$. 
	\end{proposition}
	\begin{proof}
		Exactly as in~\eqref{eq-interchange}, we have \[\LF_\bbH^{(\gamma, i)}(d\phi) \nu_\phi(dx) = \LF_\bbH^{(\gamma, i), (\gamma, x)}(d\phi)\, dx.\]
		Thus, by Theorem~\ref{thm:1disk-const}, if we sample a pair $(\phi, \mathbf x) \in H^{-1}(\bbH) \times \R$ from the measure $\LF_\bbH^{(\gamma, i), (\gamma, x)}(d\phi)\, dx$, 
		then the law of the quantum surface $(\bbH, \phi, i, \mathbf x)$ is $\frac{\gamma}{2\pi(Q-\gamma)^2} \QD_{1,1}$. For $x \in \R$, let $g_x$  be the conformal automorphism of $\bbH$ fixing $i$ and sending $x$ to $x_0$. Similarly as in Lemma~\ref{lem-change-coord}, there is an explicit constant $C_x'$ such that when $\phi$ is sampled from $\LF_\bbH^{(\gamma, i), (\gamma, x)}$ then the law of $g_x \bullet_\gamma \phi$ is $C_x' \LF_\bbH^{(\gamma, i),(\gamma, x_0)}$. Thus, when a field $\phi$ is sampled from $(\int_\R C'_x\, dx )\LF_\bbH^{(\gamma, i), (\gamma, x_0)}$ then the law of the quantum surface $(\bbH, \phi, i, x_0)$ is $\frac{\gamma}{2\pi(Q-\gamma)^2} \QD_{1,1}$.
	\end{proof}

\subsection{Adding a bulk point to a 2-pointed quantum disk}\label{subsec:third-point}
In this section we consider the LCFT description of the weight $W$ quantum disk (i.e.\ $\cM_{0,2}^\disk(W)$) marked by a bulk point. This will be used in the proof of our conformal welding result Theorem~\ref{thm-bubble-zipper}. 
\begin{definition}\label{def:disk-bulk}
For $W\ge \frac{\gamma^2}{2}$, recall  $\cM_{0,2}^\disk(W)$ from Definition~\ref{def-thick-disk}. Let $M$ be a measure on $H^{-1}(\bbH)$ such that if $\phi$ is sampled from $M$ the law of $(\bbH, \phi, 0, \infty)$ is $\cM_{0,2}^\disk(W)$. Let $(\phi, z)$ be sampled from $M(d\phi)\mu_\phi(d^2z)$. 
We write $\cM^{\disk}_{1,2}(W)$ for the law of $(\bbH, \phi, z, 0, \infty)$ viewed as a marked quantum surface. 
\end{definition}
In Definition~\ref{def:disk-bulk}, $(\bbH, \phi, 0, \infty)$ can be replaced by any embedding of a sample from $\cM_{0,2}^\disk(W)$. Starting from such an embedding, one simply needs to first reweigh by the total quantum area and then to add a point according to the quantum area measure. Recall the horizontal strip $\cS= \R\times (0,\pi)$. Under the coordinates $(\cS, \pm\infty)$, a nice relation between the Liouville field on $\cS$ and $\cM_{0,2}^\disk(W)$ has been established in \cite{AHS-SLE-integrability}, which we recall now.

\begin{definition}
	Let $P_\cS$ be the law of the free-boundary GFF on $\cS$ defined above Definition~\ref{def-thick-disk} and let $\E_\cS$ be the expectation over $P_\cS$. 
	Let $\alpha \in \R, \beta < Q$ and $z \in \cS$. Sample $(h, \mathbf c)$ from $C_{\cS}^{(\beta, \pm\infty), (\alpha, z)}P_\cS \times [e^{(\beta+\alpha - Q)c}dc]$, where 
	$C_{\cS}^{(\beta, \pm\infty), (\alpha, z)} := \lim_{\eps \to 0} \E_\cS[\eps^{\alpha^2/2} e^{\alpha(h_\eps(z) - (Q-\beta)\left| \Re z \right|)}]$. Let 
	$$\phi = h - (Q - \beta) \left| \Re \cdot \right| + \alpha G_\cS(\cdot, z) + \mathbf c$$ where $G_\cS(w,z)=G_\bbH(e^w,e^z)$.
	We write $\LF_\cS^{(\beta, \pm\infty), (\alpha, z)}$ for the law of $\phi$. If $\alpha = 0$ we simply write it as $\LF_\cS^{(\beta, \pm\infty)}$. 
\end{definition}

\begin{theorem}[{Theorem 2.22 in \cite{AHS-SLE-integrability}}]\label{thm-two-disk-equivalence}
	Fix $W > \frac{\gamma^2}2$.
	Let  $\phi$ be as in Definition~\ref{def-thick-disk} so that $(\cS, \phi, +\infty, -\infty)$ is an embedding of  a sample from $\cM_{0,2}^\disk(W)$.
	Let $T\in \R$ be sampled from the Lebesgue measure $dt$ independently of 	$\phi$.  
	Let $\wt \phi(z)=\phi (z+T)$. 
	Then the law of $\wt \phi$ is given by $ \frac\gamma{2(Q-\beta)^{2}} \LF_\cS^{(\beta,\pm\infty)}$ where $\beta = Q + \frac\gamma2 - \frac W\gamma$.
\end{theorem}

When a sample from $\cM^{\disk}_{1,2}(W)$ is embedded as $(\cS, \phi, i\mathbf{\theta}, \pm \infty)$ for some $\theta \in (0,\pi)$, 
then  $(\phi,\theta)$ are uniquely determined by the marked quantum surface structure.  The following lemma is a straightforward variant of Theorem~\ref{thm-two-disk-equivalence} that describes the joint law of $(\phi,\theta)$.
\begin{lemma}\label{lem-pointed-disk}
	Fix $W >\frac{\gamma^2}2$ and  $\beta = Q + \frac\gamma2 - \frac W\gamma$. When a sample from $\cM^{\disk}_{1,2}(W)$ is embedded as $(\cS, \phi,i\theta,+\infty,-\infty)$, then the law of $(\phi, \theta)$ is $\frac{\gamma}{2(Q-\beta)^2} \LF_\cS^{(\beta, \pm \infty),(\gamma, u)}(d\phi) 1_{\theta\in (0,\pi)} d\theta$.
	\end{lemma}
\begin{proof}
	By Girsanov's theorem as in~\eqref{eq-interchange} in the proof of Theorem~\ref{thm:1disk-const}, we have
	\eqb\label{eq-interchange-bulk}
	\LF_\cS^{(\beta, \pm\infty)}(d\phi)\mu_\phi(d^2z) = \LF_\cS^{(\gamma, z), (\beta, \pm\infty)}(d\phi) d^2z.
	\eqe
Let $M$ be the law of $\phi$ in Theorem~\ref{thm-two-disk-equivalence}  so that the law of the marked quantum surface $(\cS, \phi,+\infty, -\infty)$ is $\cM_{0,2}^\disk(W)$.
	Now we sample $(\phi, \mathbf z, T)$ according to $\mu_\phi(d^2z)M(d\phi) dt$,  where $dt$ corresponding to the Lebesgue measure on $\mathbb{R}$.
Similarly as in the proof of Theorem~\ref{thm:1disk-const},	 let $\wt \phi(z)=\phi (z+T)$ and $\mathbf u=\mathbf z-T$. 
	Then by Theorem~\ref{thm-two-disk-equivalence}  and the definition of $\cM^\disk_{1,2}(W)$,
	the law of $(\wt \phi, \mathbf u)$ is 
	$	\frac{\gamma}{2(Q-\beta)^2} \LF_\cS^{(\beta, \pm\infty)}(d\phi)\mu_\phi(d^2z) $.
	By~\eqref{eq-interchange-bulk} this equals 
	\eqb\label{eq-field-11-t}
	\frac{\gamma}{2(Q-\beta)^2}  \LF_\cS^{(\gamma, t+i\theta), (\beta, \pm\infty)}(d\phi) \times 1_{\theta\in (0,\pi)}d\theta dt.
	\eqe
	Let $S$ be the marked quantum surface $(\cS, \wt \phi, \mathbf u,-\infty,\infty)$. Since $\Re u=\Re \mathbf z +T$ and $dt$ is translation invariant,
	the joint law of $(S,\Re \mathbf u)$ is $\cM^\disk_{1,2}(W)\times  dt$. 
	Comparing to~\eqref{eq-field-11-t}, we see that for each $t\in \R$,  if $(\phi,\theta)$ is sampled from $\LF_{\cS}^{(\gamma,t+i\theta), (\beta, \pm\infty)} (d\phi) 1_{0<\theta<\pi}d\theta$, then the law of $(\cS, \phi, t+i\theta,+\infty,-\infty)$ viewed as a marked quantum surface is 
	$\cM^\disk_{1,2}(W)$. Setting $t=0$ we conclude.
\end{proof}

 The following lemma allows us to transfer the LCFT description of $\cM_{1, 2}^\disk(W)$ from $\cS$ to $\bbH$. 
\begin{lemma}\label{lem-S-to-H}
	Suppose $\alpha \in \R$ and $u \in \cS$ such that $\Re u = 0$. Let $\exp: \cS\rta \bbH$ be the map $z\mapsto e^z$. 
	If $\phi$ is sampled from $\LF_\cS^{(\alpha, u)}$ then the law of $\exp \bullet_\gamma \phi$ is  $\LF_\bbH^{(\alpha, e^u)}$.  
\end{lemma}
\begin{proof}
	It is easy to check that for $z \in \cS$ we have $C_\cS^{(0, \pm\infty), (\alpha, z)} = e^{-\alpha(Q+\alpha)\left| \Re z \right| + \frac{\alpha^2}2 \Re z} (2 \Im e^z)^{-\alpha^2/2}$, so for $u \in \cS$ with $\Re u = 0$ we have $C_\cS^{(0, \pm\infty), (\alpha, u)} = C_\bbH^{(\alpha, e^u)}$. 
	 Here $ C_\bbH^{(\alpha, z_0)}=(2 \Im z_0)^{-\alpha^2/2} |z_0|_+^{-2\alpha(Q-\alpha)} $ is the prefactor in front of $P_\bbH \times [e^{(\alpha - Q)c} dc]$
	in the description of  $\LF_\bbH^{(\alpha,z_0)}$ from Lemma~\ref{lem:Girsanov}. 
	Finally, since the law of  $\wt h := h \circ \exp^{-1}$ is  $P_\bbH$  and $G_\cS(z,w) = G_\bbH(e^z, e^w)$, we see that 
	\[\exp\bullet_\gamma(h - Q\left| \Re \cdot \right| + \alpha G_\cS(\cdot, z))  = \wt h - 2Q \log |\cdot|_+ + \alpha G_\bbH(\cdot, e^z). \]
	Adding a random constant $\mathbf c$ sampled from $e^{(\alpha - Q)c}dc$ completes the proof. 
\end{proof}

We only record the LCFT description of $\cM^{\disk}_{1,2}(2+\gamma^2)$ since it is particularly simple, 
and it is also the only case we need for our conformal welding result Theorem~\ref{thm-bubble-zipper}. 
\begin{lemma}\label{lem:hx-embed}
Sample $ (\phi,\mathbf x)$ from $\LF_{\bbH}^{(\gamma,i)} \times dx$ where $dx$ is the Lebesgue measure on $\R$.
Then the law of $(\bbH, h,i,\infty, \mathbf x)$  
viewed as marked quantum surface  is $	\frac{2Q^2}{\gamma}\cM^{\disk}_{1,2}(2+\gamma^2)$.
\end{lemma}
\begin{proof}

	Since $W = 2+\gamma^2$, we have $\beta = Q+ \frac\gamma2 - \frac W\gamma =  0$. 
	By Lemma~\ref{lem-pointed-disk}, there is a constant $C \in (0, \infty)$ such that if $(\phi_1, \theta)$ is sampled from $ \LF_\cS^{(\gamma, \theta i)} (d\phi_1)\1_{0 < \theta < \pi} d\theta$, then the law of the marked quantum surface $(\cS, \phi_1, \theta i,  +\infty, -\infty)$ is $	\frac{2Q^2}{\gamma} \cM^{\disk}_{1,2}(2+\gamma^2)$. By Lemma~\ref{lem-S-to-H}, if we set $\phi_2 := \exp\bullet_\gamma \phi_1$, then the law of $(\phi_2, \theta)$ is $  \LF_\bbH^{(\gamma, e^{i\theta})} (d\phi_2)\1_{0<\theta<\pi} d\theta$, and the law of marked quantum surface $(\bbH, \phi_2, e^{i\theta}, \infty, 0)$ is $	\frac{2Q^2}{\gamma}\cM^{\disk}_{1,2}(2+\gamma^2)$. 
	 
	Finally, let $f_\theta:\bbH \to \bbH$ be the conformal automorphism fixing $\infty$ and sending $e^{i\theta} \mapsto i$ (i.e.\ $f_\theta(z) = \frac z{\sin \theta} - \cot \theta$). 
	Setting $\phi := f_\theta \bullet_\gamma \phi_2$ and $\mathbf x = f_\theta(0)$, the law of $(\phi, \theta)$ is $\frac{1}{\sin^2\theta}  \LF_\bbH^{(\gamma, i)}(d\phi)\1_{0 < \theta < \pi} d\theta$ by Lemma~\ref{lem-change-coord}, so it is a calculus exercise to check that  the law of $ (\phi,\mathbf x)$ is $\LF_{\bbH}^{(\gamma,i)} \times dx$. On  the other hand, 
	the law of the marked quantum surface  $(\bbH, \phi, i, \infty, \mathbf x)$ is  $	\frac{2Q^2}{\gamma}\cM^{\disk}_{1,2}(2+\gamma^2)$.  
\end{proof}

%% file: sec4.tex
\section{Proof of the FZZ formula}\label{sec:FZZ}

In this section we carry out the strategy outlined in Section~\ref{subsec:sketch} to prove Theorem~\ref{thm:FZZ-physics}.
In Section~\ref{subsec:bubble} we make a precise statement  (Theorem~\ref{thm-bubble-zipper})
about the conformal welding results for $\QD_{1,1}$ and $\cMtwo(\gamma^2/2)$  that we discussed in Section~\ref{subsec:sketch}. 
We postpone its proof to Section~\ref{sec:bubble}. 
In Section~\ref{subsec:bubble2}, we prove the extension of it where the weight of the  bulk insertion is generic. 
Based on this, for $\alpha \in (\frac\gamma2, Q-\frac\gamma4)$ we identify the inverse gamma distribution for the area law in Section~\ref{subsec:inverse-gamma}.
In Section~\ref{subsec:final}, by relating  the inverse gamma distribution and  FZZ formula, we  
prove that	 $U(\alpha) = U_{\mathrm{FZZ}}(\alpha)$ for $\alpha \in (\frac{2}{\gamma}, Q-\frac{\gamma}{4})$. 
In Section~\ref{subsec:analytic} we  show that $U(\alpha)$ has an analytic extension 
on a complex neighborhood of  $(\frac{2}{\gamma}, Q)$, which implies $U(\alpha) = U_{\mathrm{FZZ}}(\alpha)$ for $\alpha \in (\frac{2}{\gamma}, Q)$, and hence the inverse gamma law for this range also.
Finally in Corollary \ref{cor_truncation} we extend the probabilistic definition of $U(\alpha)$ to the range $\alpha \in (\frac{\gamma}{2}, Q)$ in a way that matches $U_{\mathrm{FZZ}}(\alpha)$.

\subsection{SLE bubble zipper with a quantum typical bulk insertion}\label{subsec:bubble}
\newcommand{\bub}{\mathrm{Bubble}}

Let $\bub_{\bbH}(i, 0)$ be  the space of counterclockwise simple loops  on $\bbH$ that pass through $0$ and surround $i$.
More precisely, an oriented  simple closed 
loop $\eta$ on $\C$ is in $\bub_{\bbH}(i, 0)$ if and only if $0\in \eta$, $(\eta\setminus\{0\})  \subset \bbH$, 
$i$ is inside the bounded component of $\bbH\setminus\eta$, and $\eta$ surrounds $i$ counterclockwise. 
For $\eta\in \bub_\bbH(i,0)$, let $D_\eta(i)$ and  $D_\eta(\infty)$ be the bounded and unbounded components of $\bbH\setminus\eta$, respectively.  
The point $0$ corresponds to two boundary  points on $D_\eta(\infty)$ which we denote by $0^{-}$ and $0^+$ such that $\eta$ goes from $0^+$ to $0^-$.

Recall $\QD_{1,1}(\ell) $ and $ \cMtwo(\frac{\gamma^2}2;\ell,\cdot)$  as defined in Sections~\ref{subsec:QD} and~\ref{subsec:welding}.
Our proof of the FZZ formula relies on the following conformal welding equation in the same spirit as Theorem~\ref{thm:disk-welding}.
\begin{theorem}\label{thm-bubble-zipper} 
	There exists a unique probability measure $\sm$ on $\bub_{\bbH}(i, 0)$ such that the following holds.
	Suppose $(\phi, \eta)$ is sampled from $\LF_\bbH^{(\gamma, i)} \times \sm$. Then the law of $(D_\eta(0),\phi,i,0)$  and  $(D_\eta(\infty),\phi,0^-, 0^+)$ 
	viewed as a pair of marked quantum surfaces is given by  
	\[
	C\int_{0}^\infty \QD_{1,1}(r)\times \cMtwo(\frac{\gamma^2}2;\cdot,r)\, dr   \textrm{ for some constant } C \in (0,\infty).
	\]
\end{theorem}
Similarly as in Theorem~\ref{thm:disk-welding},  Theorem~\ref{thm-bubble-zipper} says that 
if we conformally weld $\QD_{1,1} $ and $ \cMtwo(\frac{\gamma^2}2)$,  we get the curve decorated quantum surface whose embedding can be described by $\LF_\bbH^{(\gamma, i)} \times \sm$. 
The proof of Theorem~\ref{thm-bubble-zipper} uses ideas which are orthogonal to the rest of the proof of Theorem~\ref{thm:FZZ-physics}.
Therefore we postpone the proof of Theorem~\ref{thm-bubble-zipper}  to Section~\ref{sec:bubble}. As part of that proof we will show that the measure $\sm$
comes from collapsing the end points of a family of $\SLE_\kappa(\rho_-;\rho_+)$ curves while conditioning on surrounding $i$.
For the purpose of proving Theorem~\ref{thm-bubble-zipper}, we do not need such a detailed description of $\sm$.

\subsection{SLE bubble zipper with a generic bulk insertion}\label{subsec:bubble2}
In this section we extend Theorem~\ref{thm-bubble-zipper} from the $\gamma$-bulk insertion to the case of a generic one. 
We start by defining quantum disks with an $\alpha$-bulk insertion.
\begin{definition}\label{def:Malpha}
	For $\alpha \in \R$, let $\phi$ be sampled from $\LF_\bbH^{(\alpha, i)}$. We write
	$\cM_{1,0}^\disk(\alpha)$ as the infinite measure describing  the law of  $(\bbH, \phi, i)$ as a marked quantum surface. Similarly, when $\phi$ is sampled from $\LF_\bbH^{(\alpha, i), (\gamma, 0)}$ as defined in Lemma~\ref{lem-def-LF11}, we write $\cM^\disk_{1,1}(\alpha)$ as the law of  $(\bbH, \phi, i,  0)$ as a marked quantum surface.
\end{definition}
\begin{remark}\label{rmk:alpha}(Choice of parametrization)
Weight and log-singularity are two different parameterizations of vertex insertions for quantum surfaces; see~\cite[Tables~1.1 and 1.2]{wedges} for more choices of parameters. We parametrize   $\mathcal M_{0,2}^\disk(W)$ using weight $W$ because  the most important property we need for $\mathcal M_{0,2}^\disk(W)$ is the  conformal welding identity, where the weights are additive.  
We parameterize $\cM_{1,0}^\disk(\alpha)$ by the log singularity $\alpha$ because this is the one used in the FZZ formula  and it behaves nicely under Girsanov transform; see Lemma~\ref{lem-reweighting-LF}.
\end{remark}

By Proposition~\ref{prop-embed-QD11},
$\cM_{1,1}^\disk(\gamma)= C \QD_{1,1}$ for some constant $C$.
Therefore, Theorem~\ref{thm-bubble-zipper} says that 
modulo a multiplicative constant, conformally welding samples from  $\cM_{1,1}^\disk(\gamma)$  and $ \cMtwo(\frac{\gamma^2}2)$,
we get a sample from $\LF_\bbH^{(\gamma,i)}$ 
decorated with an independent SLE curve. 
In this section we show that the same holds with $\gamma$ replaced by $\alpha$.

We first give an explicit description of the disintegration of $\cM_{1,0}^\disk(\alpha)$ over the boundary length.
\begin{lemma}\label{lem:disint-alpha}
For $\alpha \in \R$ and $h $ sampled from $P_\bbH$, let  $\wt h(z) = h(z) -2Q \log |z|_+ + \alpha G_\bbH(z,i)$ and $L = \nu_{\wt h}(\R)$. 
For $\ell > 0$, let $\LF_\bbH^{(\alpha, i)}(\ell)$ be  the law of  $\wt h + \frac2\gamma \log \frac\ell L$ under the reweighted measure $2^{-\alpha^2/2} \frac2\gamma \frac{\ell^{\frac2\gamma(\alpha - Q) -1}}{L^{\frac2\gamma(\alpha-Q)}} P_\bbH$, 
and let $\cM_{1,0}^\disk(\alpha; \ell)$ be the measure on quantum surfaces $(\bbH, \phi, i)$ with $\phi$ sampled from $\LF_\bbH^{(\alpha, i)}(\ell)$. Then $\cM_{1,0}^\disk(\alpha; \ell)$ is a measure on quantum surfaces with boundary length $\ell$, and 
\begin{equation}\label{eq:disint-alpha}
	\LF_\bbH^{(\alpha, i)} = \int_0^\infty \LF_\bbH^{(\alpha, i)}(\ell) \, d\ell, \quad 
	\cM_{1,0}^\disk(\alpha) = \int_0^\infty \cM_{1,0}^\disk(\alpha; \ell) \, d\ell.
\end{equation}
\end{lemma}
\begin{proof}
	It is clear that $\LF_\bbH^{(\alpha, i)}(\ell)$-a.e.\ we have $\nu_\phi(\R) = \ell$. 
To see that $\int_0^\infty \LF_\bbH^{(\alpha, i)}(\ell) \, d\ell = \LF_\bbH^{(\alpha, i)}$, we note that for any nonnegative measurable function $F$ on $H^{-1}(\bbH)$ we have
\[\int_0^\infty \int F(\wt h + \frac2\gamma \log \frac\ell L) 2^{-\alpha^2/2} \frac2\gamma \frac{\ell^{\frac2\gamma(\alpha - Q) -1}}{L^{\frac2\gamma(\alpha-Q)}} P_\bbH(dh) \, d\ell = \int  \int_\R F(\wt h + c) 2^{-\alpha^2/2} e^{(\alpha-Q)c}\,dc\, P_\bbH(dh) \] 
using Fubini's theorem and the change of variable $c = \frac2\gamma \log \frac\ell L$. 
This yields $\int_0^\infty \LF_\bbH^{(\alpha, i)}(\ell) \, d\ell = \LF_\bbH^{(\alpha, i)}$, since the left hand side of the above  equation characterizes $\int_0^\infty \LF_\bbH^{(\alpha, i)}(\ell) \, d\ell$ and the right hand side gives $\LF_\bbH^{(\alpha, i)}$ thanks to Lemma \ref{lem:Girsanov}. The second identity in~\eqref{eq:disint-alpha} then  follows from definition. 
\end{proof}

We state without proof the variant of this lemma for $\cM_{1,1}^\disk(\alpha)$.
\begin{lemma}\label{lem:disint-alpha-11}
	For $\alpha \in \R$ and $h $ sampled from $P_\bbH$, let  $\wt h(z) = h(z) -2Q \log |z|_+ + \alpha G_\bbH(z,i) + \frac\gamma2G_\bbH(z,0)$ and $L = \nu_{\wt h}(\R)$. 
		For $\ell > 0$, let $\LF_\bbH^{(\alpha, i)}(\ell)$ be  the law of  $\wt h + \frac2\gamma \log \frac\ell L$ under the reweighted measure $2^{-\alpha^2/2} \frac2\gamma \frac{\ell^{\frac2\gamma(\alpha - Q)}}{L^{\frac2\gamma(\alpha-Q)+1}} P_\bbH$,
	and let $\cM_{1,1}^\disk(\alpha; \ell)$ be the measure on quantum surfaces $(\bbH, \phi, i)$ with $\phi$ sampled from $\LF_\bbH^{(\alpha, i)}(\ell)$. Then $\cM_{1,1}^\disk(\alpha; \ell)$ is a measure on quantum surfaces with boundary length $\ell$, and 
	\begin{equation}\label{eq:disint-alpha-11}
		\LF_\bbH^{(\alpha, i), (\gamma, 0)} = \int_0^\infty \LF_\bbH^{(\alpha, i), (\gamma, 0)}(\ell) \, d\ell, \quad 
		\cM_{1,1}^\disk(\alpha) = \int_0^\infty \cM_{1,1}^\disk(\alpha; \ell) \, d\ell.
	\end{equation}
\end{lemma}

To generalize Theorem~\ref{thm-bubble-zipper},  we also need to deform our measure on curves. 
Given $\eta\in \bub_\bbH(i,0)$, let $\psi_\eta:\bbH\to D_\eta(i)$ be the unique conformal map fixing $i$ and $0$. 
Let $\sm$ be the measure on $\bub_\bbH(i,0)$ in Theorem~\ref{thm-bubble-zipper}  and $\Delta(\alpha)=\frac\alpha2(Q-\frac\alpha2)$. Let 
$\sm_\alpha$ be the  measure on $\mathrm{Bubble}_\bbH(i, 0)$ obtained by reweighting $\sm$ as follows:  
\eqb\label{eq-m-alpha}
\frac{d\sm_\alpha}{d\sm}(\eta) = |\psi_\eta'(i)|^{2\Delta_\alpha - 2}.
\eqe

We now are ready to state the bubble zipper result for quantum disk with $\alpha$-bulk insertion. We consider only $\alpha > \frac\gamma2$; for such $\alpha$ we have  $|\LF_\bbH^{(\alpha, i)}(1)|< \infty$ by Lemma~\ref{lem-len-LF}.
\begin{theorem}\label{thm-FZZ-weld}
	
	There exists $C \in (0,\infty)$ such that the following holds. For $\alpha  > \frac\gamma2$,
	suppose $(\phi, \eta)$ is sampled from $\LF_\bbH^{(\alpha, i)}(1) \times \sm_\alpha$. 
	Then the law of $(D_\eta(0),\phi,i,0)$  and  $(D_\eta(\infty),\phi,0^-, 0^+)$ 
	viewed as a pair of marked quantum surfaces is given by  
	\eqb\label{eq-FZZ-weld}
	C\int_{0}^\infty \cM^\disk_{1,1}(\alpha; r)\times \cMtwo(\frac{\gamma^2}2; 1,r)\, dr. 
	\eqe
\end{theorem}
The proof of Theorem~\ref{thm-FZZ-weld} is similar in spirit to that of \cite[Proposition 4.5]{AHS-SLE-integrability}. 
We begin by explaining how a Liouville field with specified boundary length changes under a suitable reweighting.
	\begin{lemma}\label{lem-reweighting-LF}
	Let $\alpha > \frac\gamma2$. For any $\ell > 0$, $\eps \in (0,1)$ and for any nonnegative measurable function $f$ on $H^{-1}(\bbH)$ for which $\phi \mapsto f(\phi)$ only depends on $\phi|_{\bbH\backslash B_\eps(i)}$,
\begin{equation}\label{eq:LF-eps}
	\int f(\phi) \times \eps^{\frac12(\alpha^2-\gamma^2)} e^{(\alpha - \gamma)\phi_\eps(i)}d \LF_\bbH^{(\gamma, i)}(\ell) = \int f(\phi)d\LF_\bbH^{(\alpha, i)}(\ell).
\end{equation}
	 Moreover, the same holds when we replace $\LF_\bbH^{(\gamma, i)}$ and $\LF_\bbH^{(\alpha, i)}$ with $\LF_\bbH^{(\gamma, i), (\gamma, 0)}$ and $\LF_\bbH^{(\alpha, i), (\gamma, 0)}$.
\end{lemma}
\begin{proof}
	We only explain the proof of~\eqref{eq:LF-eps} since the same argument works for $\LF_\bbH^{(\gamma, i), (\gamma, 0)}$ and $\LF_\bbH^{(\alpha, i), (\gamma, 0)}$. 
	For a GFF $h$ sampled from $P_\bbH$, let $\wt h := h -2Q \log |\cdot|_+ + \gamma G_\bbH(\cdot, i)$. 
	Let $\theta_\eps$ be the uniform probability measure on $\partial B_\eps(i)$ and let $\delta := (2 \log |\cdot|_+, \theta_\eps)$, 
	where $(\cdot,\cdot)$ means pairings of generalized functions hence $(\cdot,\theta_\eps)$ means average over $\theta_\eps$.  
	Recall $G_\bbH (x,y)=\E[h(x)h(y)]$ from~\eqref{covariance}. We see that $(G_\bbH(\cdot, i), \theta_\eps) = -\log(2\eps) +\delta$. 
	Thus, $\wt h_\eps(i) = h_\eps(i) +(\gamma-Q)\delta - \gamma \log (2\eps)$. Let $\E =\E_\bbH$ be the expectation over $P_\bbH$.
	By the description of~$\LF_\bbH^{(\gamma, i)}$ from Lemma~\ref{lem:Girsanov},
	we have
	\begin{align}
	&	\int f(\phi) \times \eps^{\frac12(\alpha^2-\gamma^2)} e^{(\alpha - \gamma)\phi_\eps(i)}d \LF_\bbH^{(\gamma, i)}=2^{-\frac{\gamma^2}2} \eps^{\frac12(\alpha^2-\gamma^2)} \int_\R \E[e^{(\alpha - \gamma)(\wt h_\eps(i)+c)} f(\wt h + c)] e^{(\gamma - Q)c}\, dc\nonumber
	\\
	&=2^{-\frac{\alpha^2}2} (2\eps)^{\frac12(\alpha-\gamma)^2} e^{(\alpha - \gamma) (\gamma - Q) \delta} 
	 \int_\R {\E[e^{(\alpha - \gamma)h_\eps(i)} f(\wt h + c)]}  e^{(\alpha - Q)c}\, dc. \label{eq:h-int}
	\end{align}
	 Define $G_\bbH^\eps(z, i) := \E[h(z)h_\eps(i)] = (G_\bbH(z, \cdot), \theta_\eps)$. Then $G_\bbH^\eps(z, i)|_{\bbH\backslash B_\eps(i)} = G_\bbH(z, i)|_{\bbH\backslash B_\eps(i)} + \delta$.
	By the Girsanov Theorem and the fact that $f(\phi)$ only depends on $\phi|_{\bbH\setminus B_\eps(i)}$, we have 
\begin{align}
&\int_\R {\E[e^{(\alpha - \gamma)h_\eps(i)} f(\wt h + c)]}  e^{(\alpha - Q)c}\, dc
=\E[e^{(\alpha - \gamma)h_\eps(i)}]\int_\R \E[f(\wt h + (\alpha - \gamma)G_\bbH^\eps(\cdot, i) + c)] e^{(\alpha - Q)c}\, dc\nonumber\\
&= \E[e^{(\alpha - \gamma)h_\eps(i)}]\int_\R \E[f(\wt h + (\alpha - \gamma)G_\bbH(\cdot, i)+ (\alpha-\gamma)\delta+ c)] e^{(\alpha - Q)c}\, dc. \label{eq:Girsanov1}
\end{align} 
Since $\Var(h_\eps(i)) = \iint G_\bbH(z,w) \theta_\eps(dz) \theta_\eps(dw) = -\log(2\eps) +2\delta$, we have 
\begin{equation}\label{eq:varh}
\E[e^{(\alpha- \gamma)h_\eps(i)}] = (2\eps)^{-\frac12(\alpha-\gamma)^2} e^{(\alpha-\gamma)^2 \delta}.
\end{equation}
On the other hand, with the change of variable $c'=(\alpha-\gamma)\delta+c$, we get
\begin{align*}
&\int_\R \E[f(\wt h + (\alpha - \gamma)G_\bbH(\cdot, i)+ (\alpha-\gamma)\delta+ c)] e^{(\alpha - Q)c}\, dc\\
&=e^{-(\alpha-\gamma)(\alpha - Q)\delta}
\int_\R \E[f(\wt h + (\alpha - \gamma)G_\bbH(\cdot, i)+ c'] e^{(\alpha - Q)c'}\, dc' = e^{-(\alpha-\gamma)(\alpha - Q)\delta}  2^{\frac{\alpha^2}2}   \int f(\phi)d\LF_\bbH^{(\alpha, i)}.
\end{align*}
Combining this with~\eqref{eq:h-int},~\eqref{eq:Girsanov1} and~\eqref{eq:varh}, and since 
\[
2^{-\frac{\alpha^2}2} (2\eps)^{\frac12(\alpha-\gamma)^2} e^{(\alpha - \gamma) (\gamma - Q) \delta} \times (2\eps)^{-\frac12(\alpha-\gamma)^2} e^{(\alpha-\gamma)^2 \delta}
\times e^{-(\alpha-\gamma)(\alpha - Q)\delta}  2^{\frac{\alpha^2}2}=1,
\]
we have 
	\[
	\int f(\phi) \times \eps^{\frac12(\alpha^2-\gamma^2)} e^{(\alpha - \gamma)\phi_\eps(i)}d \LF_\bbH^{(\gamma, i)} = \int f(\phi)d\LF_\bbH^{(\alpha, i)}.
	\]
	Disintegrating over $\ell$ completes the proof. 
\end{proof}

We also need the following analog of Lemma~\ref{lem-reweighting-LF} which is proved in the same way. 
\begin{lemma}\label{lem-reweighting-LF-loop}
Let $\alpha > \frac\gamma2$, and let $\eta \in \mathrm{Bubble}_\bbH(i,0)$. Let $\psi_\eta: \bbH \to D_\eta(i)$ be the conformal map fixing $0$ and $i$, and let $\bbH_{\eta, \eps} := \bbH\backslash \psi_\eta(B_\eps(i))$. Let $\theta_\eps$ be the uniform probability measure on $\partial B_\eps(i)$ and $\wh \theta_\eps := (\psi_\eta)_* \theta_\eps$. For any $\ell > 0$, $\eps \in (0,1)$ and  for any nonnegative measurable function $f$ on $H^{-1}(\bbH)$ for which $\phi \mapsto f(\phi)$ only depends 
on  $\phi|_{\bbH_{\eta, \eps}}$, we have 
		\[
		\int f(\phi|_{\bbH_{\eta, \eps}}) \times \eps^{\frac12(\alpha^2-\gamma^2)} e^{(\alpha - \gamma)(\phi, \wh \theta_\eps)}d \LF_\bbH^{(\gamma, i)}(\ell) = \int f(\phi|_{\bbH_{\eta, \eps}}) \times  |\psi_\eta'(i)|^{\frac12(\gamma^2-\alpha^2)}d\LF_\bbH^{(\alpha, i)}(\ell).
		\]
\end{lemma}
\begin{proof}
	Let $\delta := (2\log|\cdot|_+, \wh \theta_\eps)$  and $G_\bbH^\eps(z, i) := (G_\bbH(z, \cdot) , \wh \theta_\eps)$. 
	Then for $z \in \bbH_{\eta, \eps}$, 
	\begin{align*}
	&G_\bbH^\eps(z, i) = (-\log|z - \psi_\eta(\cdot)| - \log |z + \psi_\eta(\cdot)|,  \theta_\eps) + 2 \log |z|_+ + \delta\\
	&= -\log |z-i| - \log |z+i| + 2\log|z|_+ + \delta	 = G_\bbH(z,i) + \delta.
	\end{align*}
	Indeed, since $-\log|z - \cdot| - \log |z + \cdot|$ is harmonic on $D$ and $\psi_\eta$ is conformal, the map $-\log|z - \psi_\eta(\cdot)| - \log |z + \psi_\eta(\cdot)|$ is harmonic on $B_\eps(i)$. 
	
	Let $\eta_\eps = \psi(\partial B_\eps(i))$, then $\wh \theta_\eps$ is the harmonic measure on $\eta_\eps$ viewed from $i$. It is well known that  $\int \log |z - i|\, \wh \theta_\eps(dz)$ is the conformal radius of $\eta_\eps$ viewed from $i$, and this conformal radius can alternatively be computed as $\eps |\psi_\eta'(i)|$. Thus $(G_\bbH(\cdot, i), \wh \theta_\eps) = -\log|2 \eps \psi_\eta'(i)| +\delta$. 
	
	Finally, $\Var((h, \wh \theta_\eps)) = \int G_\bbH^\eps(z,i) \, \wh \theta_\eps(dz) = -\log |2\eps \psi_\eta'(i)| + 2\delta$. 
	Now the exact same computation  in Lemma~\ref{lem-reweighting-LF} with $(2\eps)$ in~\eqref{eq:h-int} and~\eqref{eq:varh} replaced by 
	$(2\eps\psi'_\eta(i))$ gives Lemma~\ref{lem-reweighting-LF-loop}.  
\end{proof}
\begin{proof}[{Proof of Theorem~\ref{thm-FZZ-weld}}]
	With slight abuse of notation, Theorem~\ref{thm-bubble-zipper} gives 
	\eqb\label{eq-unweighted-weld}
	\LF_\bbH^{(\gamma, i)}(1) \times \sm = C \int_0^\infty \cM_{1,1}^\disk(\gamma; r) \times \cM_{0,2}^\disk(\frac{\gamma^2}2; 1,r) \, dr,
	\eqe
	where~\eqref{eq-unweighted-weld} should be interpreted as saying that when  a pair of quantum surfaces is sampled from the right hand side, conformally welded, then embedded by sending the bulk and boundary points to $i$ and $0$ in $\bbH$, the law of the resulting field and curve $(\phi, \eta)$ is the left hand side. In our argument we will treat the right hand side of~\eqref{eq-unweighted-weld} as a measure on pairs $(\phi, \eta)$. 
	
	We use the notation of Lemma~\ref{lem-reweighting-LF-loop}, so $\psi_\eta: \bbH \to D_\eta(i)$ is a conformal map. 	Define
	\[
	X = \phi \circ \psi_\eta + Q \log |\psi_\eta'|.
	\]
	For $\eps \in (0,1)$ and $f$ a nonnegative measurable function of $\phi|_{\bbH_{\eta, \eps}}$, and $g$ a nonnegative measurable function of $\eta$,  weighting~\eqref{eq-unweighted-weld} gives
	\begin{align}
	&\int f(\phi|_{\bbH_{\eta, \eps}}) g(\eta) \eps^{\frac12(\alpha^2-\gamma^2)} e^{(\alpha - \gamma)X_\eps(i)} \LF_\bbH^{(\gamma, i)}(1) \times \sm\label{eq-LHS} \\
	&= C\int_0^\infty \left( \int f(\phi|_{\bbH_{\eta, \eps}}) g(\eta) \eps^{\frac12(\alpha^2-\gamma^2)} e^{(\alpha - \gamma)X_\eps(i)} \cM_{1,1}^\disk(\gamma; r) \times \cM_{0,2}^\disk(\frac{\gamma^2}2; 1,r) \right) \, dr. \label{eq-RHS}
	\end{align}
	Since $\psi_\eta'$ is holomorphic, $\log |\psi_\eta'|$ is harmonic so $(\log|\psi_\eta'|, \theta_\eps) = \log |\psi_\eta'(i)|$. Thus
	\eqb \label{eq-circle-av-X-Y}
	X_\eps(i) = (X, \theta_\eps) = (\phi \circ \psi_\eta + Q \log |\psi_\eta'|, \theta_\eps) = (\phi, \wh\theta_\eps) + Q \log |\psi_\eta'(i)|.
	\eqe
	By~\eqref{eq-circle-av-X-Y} and Lemma~\ref{lem-reweighting-LF-loop}, the expression~\eqref{eq-LHS} equals
	\[\int f(\phi|_{\bbH_{\eta, \eps}}) g(\eta) |\psi_\eta'(i)|^{-\frac{\alpha^2}2 +Q\alpha - 2} \LF_\bbH^{(\alpha, i)}(1) \times \sm .\]
Recall that $\sm_\alpha=|\psi_\eta'(i)|^{-\frac{\alpha^2}2 +Q\alpha - 2}  \sm$.	By Lemma~\ref{lem-reweighting-LF} the integral~\eqref{eq-RHS} equals
	\[C\int_0^\infty \left(\int f(\phi|_{\bbH_{\eta, \eps}}) g(\eta) \cM_{1,1}^\disk(\alpha; r) \times \cM_{0,2}^\disk(\frac{\gamma^2}2; 1,r) \right) \, dr. \] 
	The result follows by equating the above two integrals and sending $\eps \to 0$. 
\end{proof}

A priori it is not clear whether $\sm_\alpha$ is finite. Using Theorem~\ref{thm-FZZ-weld} we have the following. 
\begin{lemma}
	For $\alpha\in (\frac\gamma2, Q + \frac2\gamma)$, the measure $\sm_\alpha$ is finite.
\end{lemma}
\begin{proof}
	Consider the event that $\nu_\phi (\R)\in [1,2] $ where $(\phi, \eta)$ is sampled $\LF_\bbH^{(\alpha, i)} \times \sm_\alpha$.
	Evaluating the size of this event using the two descriptions of the same measure of Theorem~\ref{thm-FZZ-weld}, since $|\cM_{1,1}^\disk(\alpha;r)| \propto r^{\frac2\gamma(\alpha-Q)}$ by Lemma~\ref{lem:disint-alpha-11} 
		and $\cMtwo(\frac{\gamma^2}2;\ell,r) \propto \frac{\ell^{4/\gamma^2-1}r^{4/\gamma^2-1}}{(\ell^{4/\gamma^2} + r^{4/\gamma^2})^2}$  by Lemma~\ref{lem:lengh}, we get for some constant $C \in (0, \infty)$
	\[\LF_\bbH^{(\alpha, i)}[\{ \nu_\phi(\R) \in (1,2)\}] |\sm_\alpha| = C  \int_0^\infty \int_1^2 r^{\frac2\gamma(\alpha - Q)} \times \frac{\ell^{4/\gamma^2-1} r^{4/\gamma^2-1}}{(\ell^{4/\gamma^2} + r^{4/\gamma^2})^2} \,d\ell \, dr.\]
	
	Since $\frac{\ell^{4/\gamma^2-1} r^{4/\gamma^2-1}}{(\ell^{4/\gamma^2} + r^{4/\gamma^2})^2} < \frac{2^{4/\gamma^2-1} r^{4/\gamma^2-1}}{(1+r^{4/\gamma^2})^2}$ for $\ell \in (1,2)$, we conclude that 
	\[\LF_\bbH^{(\alpha, i)}[\{ \nu_\phi(\R) \in (1,2)\}] |\sm_\alpha|  < C 2^{4/\gamma^2} \int_0^\infty \frac{r^{\frac2\gamma  \alpha - 2}}{(1+r^{4/\gamma^2})^2} \, dr < \infty, \]
	where finiteness 
	follows from $\alpha \in (\frac\gamma2, Q + \frac2\gamma)$. Since $\LF_\bbH^{(\alpha, i)}[\{ \nu_\phi(\R) \in (1,2)\}]>0$ we get $|\sm_\alpha|<\infty$. 
\end{proof}

The following proposition is a rephrasing of Theorem~\ref{thm-FZZ-weld} which is more convenient for our argument in Section~\ref{subsec:inverse-gamma}.

\begin{proposition}\label{prop:unit-length} 
	Fix $\alpha\in (\frac\gamma2, Q + \frac2\gamma)$ and sample $\phi$ from $\LF_\bbH^{(\alpha, i)}(1)^\#$ (so $(\bbH, \phi, i)$ has law $\cM^\disk_{1,0}(\alpha;1)^{\#}$). 
	Let $\eta$ be a sample of $\sm_\alpha^{\#}$ independent of $\phi$.
	Let $\cL$ be the quantum length of $\eta$ and $\phi_0=\phi-\frac{2}{\gamma}\log \cL$. Then 
	$(D_\eta(i),\phi_0,i,0)$  and  $(D_\eta(\infty),\phi,0^-, 0^+)$ viewed as marked quantum surfaces are independent.  The law of the former is 
	$\cM^\disk_{1,1}(\alpha; 1)^{\#}$. The law of the latter is the probability measure proportional to  $\int_0^\infty r^{\frac2\gamma(\alpha - Q)}\cMtwo(\frac{\gamma^2}2;1,r) dr$.
\end{proposition}
\begin{proof}
From the definition of $\cM_{1,0}^\disk(\alpha; r)$, we see that if $(D, \psi, z)$ is sampled from $\cM_{1,0}^\disk(\alpha; r)^\#$ then $(D, \psi -\frac2\gamma \log r, z)$ has law $\cM_{1,0}^\disk(\alpha; 1)^\#$. Since $\cM_{1,0}^\disk(\alpha; 1)^\#$ and $\cM_{1,1}^\disk(\alpha; 1)^\#$ are the same if we ignore the boundary marked point for $\cM_{1,1}^\disk(\alpha; 1)^\#$,
by Theorem~\ref{thm-FZZ-weld}, the joint law of $(D_\eta(i),\phi_0,i,0)$  and  $(D_\eta(\infty),\phi,0^-, 0^+)$ is the probability measure proportional to 
\[ \cM_{1,1}^\disk(\alpha; 1)^\# \times \int_0^\infty  \left|\cM_{1,1}^\disk(\alpha; r)\right| \cMtwo(\frac{\gamma^2}2; 1, r) \, dr. \] Finally, since $\left|\cM_{1,1}^\disk(\alpha; r)\right| = r \left|\cM_{1,0}^\disk(\alpha; r) \right| \propto r^{\frac2\gamma(\alpha-Q)}$ by Lemma~\ref{lem-len-LF}, we are done. 
\end{proof}

\subsection{Identification of the inverse gamma distribution}\label{subsec:inverse-gamma}
By definition, $\cM_{1,0}^\disk(\alpha; 1)^\#$ and $\cM_{1,1}^\disk(\alpha; 1)^\#$  are the same if we ignore the boundary marked point. 
Therefore, Proposition~\ref{prop:unit-length} can be viewed as a recursive property for $\cM_{1,0}^\disk(\alpha; 1)^\#$.
The main goal of this subsection is to use this property  to  identify the law of its quantum area.
\begin{proposition}\label{prop:inverse-a}
	For $\alpha \in (\frac\gamma2, Q - \frac\gamma4)$, 	the law of the quantum area
	of a sample from $\cM_{1,0}^\disk(\alpha; 1)^\#$ is the inverse gamma distribution with shape $\frac2\gamma(Q-\alpha)$ and scale $\frac1{4\sin \frac{\pi \gamma^2}4}$.
\end{proposition}

In the setting of Proposition~\ref{prop:unit-length}, we write $\cA_0=\mu_\phi(\bbH)$ and $\cA_1=\mu_{\phi_0}(D_\eta(i))$. Let $\cA=\mu_\phi(D_\eta(\infty))$ and recall that $\cL$ 
is  the quantum length of $\eta$. Then $\cA_0$ and $\cA_1$ agree in law, $\cA_1$ is independent of $(\cA, \cL)$, and
\begin{equation}\label{eq:area-sum}
\cA_0=\cA+ \cL^2\cA_1
\end{equation}
since $\cL^2\cA_1$ (resp. $\cA$) is the quantum area of the region inside (resp. outside) $\eta$.
Note that $(\cA,\cL)$ is determined by  $(D_\eta(\infty),\phi,0^-, 0^+)$, whose law is given by boundary reweighting of  $\cMtwo(\frac{\gamma^2}2)$  at the end of Proposition~\ref{prop:unit-length}. 
By Proposition~\ref{lem-mot}, we can describe $(\cA, \cL)$ in terms of Brownian motion in cones.

Recall from Section~\ref{subsec:special disks} that for the cone $\cC_\theta = \{ z : \arg z \in (0, \theta)\}$ with $\theta \in (0,2\pi)$, the measure $\sm_{\cC_\theta}(z)$ is the probability measure corresponding to Brownian motion started at $z$ and killed when it exits $\cC_\theta$.
For $x,y>0$ we define using a limiting procedure a measure $\sm_{\cC_\theta}(x, ye^{i\theta})$ corresponding to Brownian motion started at $x$ and restricted to the event that it exits $\cC_\theta$ at $ye^{i\theta}$. Using a similar limiting procedure, we can define $\sm_{\cC_\theta}(x, 0)$ for $x  > 0$ and $\sm_{\cC_\theta}(z, 0)$ for $z \in \cC_\theta$ as well; see  Appendix~\ref{sec:BM-cone} for more details.

As stated in Lemma~\ref{lem-markov}, essentially by the Markov property of Brownian motion,  for $0 < \theta < \phi < 2\pi$ there is a constant $c >0$ such that for $u>0$ we have
	\eqb \label{eq-markov}
	\sm_{\cC_\phi}(u, 0)^\# = c \int_0^\infty \sm_{\cC_\theta}(u, r e^{i\theta}) \times  \sm_{\cC_\phi}(re^{i\theta}, 0)\, dr.
	\eqe
	More precisely, if we sample $(Z^1,Z^2)$ from the right hand side of~\eqref{eq-markov}, the concatenation of $Z^1$ (a path from $u$ to $r e^{i\theta}$) and $Z^2$ (a path from $r e^{i\theta}$ to 0) yields a path $Z$ from $u$ to $0$ whose law is $\sm_{\cC_\phi}(u,0)^\#$. We refer to Figure~\ref{fig-cone-recursion} for an illustration. 
	
For $(Z^1, Z^2)$ sampled from~\eqref{eq-markov}, define the random variables 
\begin{equation}\label{eq-def-AL}
A = \text{duration of }Z^1, \qquad L = Z^1(A)/e^{i\theta}>0.
\end{equation}

\begin{lemma}\label{lem:AL}
	The law of $(\cA,\cL)$ is the same as $(A,\frac{L}{u})$ in~\eqref{eq-def-AL} with $\phi = \frac{\gamma \pi}{2(Q-\alpha)}, \theta = \frac{\pi \gamma^2}4$ and $u = \frac1{\sqrt{2\sin \theta}}$. 
\end{lemma}
\begin{proof}
	By~\eqref{eq-markov}, since $\left|\sm_{\cC_\phi}(re^{i\theta}, 0) \right| \propto r^{-\frac\pi\phi} = r^{\frac2\gamma(\alpha - Q)}$ (Corollary~\ref{cor-cone-duration-law}), the marginal law of $Z^1$ is given by $\left(\int_0^\infty r^{\frac2\gamma(\alpha - Q)} \sm_{\cC_\theta}(u, re^{i\theta}) \, dr \right)^\#$. The claim with $u$ replaced by $\frac{1}{\mathbbm a \sin \theta}$ then follows from Proposition~\ref{prop:unit-length} and Lemma~\ref{lem-mot}. Finally Theorem~\ref{thm:MOT-var} gives $\frac{1}{\mathbbm a \sin \theta} = \frac1{\sqrt{2\sin\theta}}$.
\end{proof}

\begin{figure}[ht!]
	\begin{center}
		\includegraphics[scale=0.7]{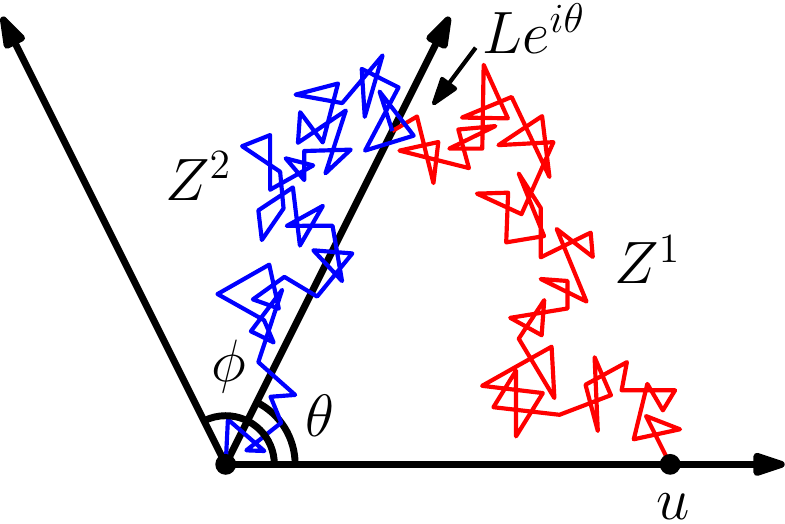}%
	\end{center}
	\caption{\label{fig-cone-recursion}  We split the Brownian path $Z$ in $\cC_\phi$ from $u$ to $0$ at the time it first hits the ray $e^{i\theta} \R_+$ for some $\theta\in (0,\phi)$ to get the subpaths $Z^1$ (red) and $Z^2$ (blue). We let $A$ be the duration of $Z^1$ and $L$ be such that $Le^{i\theta}$ is the endpoint of $Z^1$. 
	}
\end{figure}

We now give a  characterization of the inverse gamma distribution that  implies Proposition~\ref{prop:inverse-a}.
\begin{proposition}\label{prop-characterization}
	Let $0<\theta<\phi<2\pi$ and $u > 0$. Suppose $X$ is  a real-valued random variable independent of $(A,L)$ as sampled in~\eqref{eq-def-AL}. Then
	\eqb\label{eq-characterization}
	X \stackrel d= A + \frac{L^2}{u^2} X
	\eqe if and only if the law of $X$ is the inverse gamma distribution with shape $\frac\pi\phi$ and scale $\frac{u^2}2$  as described in Equation~\eqref{eq:inverse-gamma-parameter}.
\end{proposition}

\begin{proof}[Proof of Proposition~\ref{prop:inverse-a} given Proposition~\ref{prop-characterization}]
	In~\eqref{eq:area-sum}, each of $\cA_0$ and $\cA_1$ agrees in law with the quantum area $X$ of a sample from $\cM_{1,0}^\disk(\alpha; 1)^\#$. Moreover, by Proposition~\ref{prop:unit-length} we have $\cA_1$ independent of $(\cA, \cL)$, so by Lemma~\ref{lem:AL} we have $X \stackrel d= A + \frac{L^2}{u^2} X$. Then Proposition~\ref{prop-characterization} gives the law of $X$. Here the constraint $\alpha<Q-\frac{\gamma}4$ comes from $\phi = \frac{\gamma \pi}{2(Q-\alpha)}\in(0,2\pi)$.
\end{proof}

To prove Proposition~\ref{prop-characterization}, we need some basic properties about Brownian motion in cones whose full proofs we defer to Appendix~\ref{sec:BM-cone}.
See Figure~\ref{fig-cone-recursion} for an illustration.
\begin{lemma}\label{lem-recursion-invgamma}
	Suppose $0 < \theta < \phi < 2\pi$ and $u>0$, and define $Z^1, Z^2$ by the right hand side of~\eqref{eq-markov}. Let $A,L$ be defined as in~\eqref{eq-def-AL}, let $T$ be the sum of the durations of $Z^1$ and $Z^2$, and let $Y = \frac{u^2}{L^2} (T-A)$.
	Then both the laws of $T$ and $Y$ are the  inverse gamma distribution with shape $\frac\pi\phi$ and scale $\frac{u^2}2$, and $Y$ is independent of $(A,L)$. 
	Moreover, for $\eps \in (0,\frac{\pi}{\phi})$, we have $\E[ (\frac L u)^\eps] < 1$. 
\end{lemma}
\begin{proof}
	By~\eqref{eq-markov}, the law of $T$ agrees with that of the duration of a sample from $\sm_{\cC_\phi}(u, 0)^\#$, which by  
	Corollary~\ref{cor-cone-duration-law} is the inverse gamma distribution with the desired parameters. By~\eqref{eq-markov} and Brownian scaling, the conditional law of $Y$ given $(A,L)$ agrees with the law of the duration of a sample from $\sm_{\cC_\phi}(ue^{i\theta},0)^\#$, and so by Corollary~\ref{cor-cone-duration-law} it is the inverse gamma with shape $\frac\pi\phi$ and scale $\frac{u^2}2$. Finally, the claim $\E[(\frac L u)^\eps] < 1$ is Lemma~\ref{lem-moment-small}.
\end{proof}
\begin{proof}[Proof of Proposition~\ref{prop-characterization}]
	With notation as in Lemma~\ref{lem-recursion-invgamma}, we have $T = A + \frac{L^2}{u^2} Y$. Since each of $T, Y$ are distributed as inverse gamma with shape $\frac\pi\phi$ and scale $\frac{u^2}2$, and $Y$ is independent of $(A,L)$, we obtain the ``if'' direction of the proposition.
	
	Now we check the converse. Let $(A_1, L_1), \dots, (A_n, L_n)$ be independent copies of $(A,L)$ in~\eqref{eq-def-AL}. Suppose $X$ satisfies~\eqref{eq-characterization}. For any positive integer $n$,  iterating~\eqref{eq-characterization} $n$ times yields that $X$ has the same distribution as
	\[A_1 + \frac{L_1^2}{u^2} \left(A_2 + \frac{L_2^2}{u^2} \left(\dots \left(A_n + \frac{L_n^2}{u^2} X \right) \dots\right)  \right) =: S_n + \left( \prod_{i=1}^n\frac{L_i^2}{u^2}\right) X ,\]
where $S_n$ is a function of $\{(A_i, L_i)\}_{1\leq i \leq n}$. Since Lemma~\ref{lem-recursion-invgamma} gives $\E[\left(\frac Lu \right)^\eps] < 1$ for some $\eps >0$, by Markov's inequality we have $\prod_{i=1}^n \frac{L_i^2}{u^2} \to 0$ in probability as $n \to \infty$, and hence $S_n \to X$ in distribution as $n \to \infty$. Thus any solution to~\eqref{eq-characterization} is the distributional limit of $S_n$, which is unique.
\end{proof}

 \subsection{From the inverse gamma distribution to the FZZ formula}\label{subsec:final}
 Recall $ \left \langle  e^{\alpha \phi(i)}  \right \rangle$ from Definition~\ref{def:U}. 
 By Definitions~\ref{def:U} and~\ref{def:Malpha}, for $\alpha\in (\frac{2}{\gamma},Q)$, $\mu\ge 0$ and $\mu_B>0$
 \begin{align*}
 \left \langle  e^{\alpha \phi(i)}  \right \rangle &=  \LF_\bbH^{(\alpha, i)}[e^{-\mu \mu_\phi(\bbH) - \mu_B \nu_\phi(\R)} -1]\\
 &= \int_0^\infty \cM_{1,0}^\disk(\alpha; \ell)[e^{-\mu A-\mu_B \ell}-1] \, d\ell \\
 &= \int_0^\infty | \cM_{1,0}^\disk(\alpha; \ell)| \cM_{1,0}^\disk(\alpha; \ell)^{\#}[e^{-\mu A-\mu_B \ell}-1] \, d\ell.
 \end{align*} 
 Here $A$  represents the quantum  area  of  a quantum surface sampled from  $\cM_{1,0}^\disk(\alpha; \ell)$.  
 Recall $\ol U_0(\alpha)$ from  Proposition~\ref{prop-remy-U}.  By Lemma~\ref{lem:disint-alpha}, we have 	
 \begin{equation}\label{eq:FZZ-final1}
 \left \langle  e^{\alpha \phi(i)}  \right \rangle=C  \int_0^\infty \ell^{\frac2\gamma(\alpha-Q)-1} \cM_{1,0}^\disk(\alpha; 1)^{\#}[e^{-\mu \ell^2 A   -\mu_B \ell}-1] \, d\ell \quad \textrm{ for }\alpha\in (\frac2{\gamma}, Q),
 \end{equation}
 where $C =  \frac2\gamma 2^{-\frac{\alpha^2}2} \ol U_0(\alpha) $. By~\eqref{eq:FZZ-final1} and Proposition~\ref{prop:inverse-a}, 
 we can prove Theorem~\ref{thm:FZZ-physics} for $\alpha \in (\frac{2}{\gamma}, Q-\frac{\gamma}{4})$.
 \begin{proposition}\label{prop:inverse-gamma-a}
  $U(\alpha) =U_{\mathrm{FZZ}}(\alpha)$ for $\alpha \in (\frac{2}{\gamma}, Q-\frac{\gamma}{4})$.
 \end{proposition}
In order to prove Proposition \ref{prop:inverse-gamma-a}, we first  record two simple calculus facts which will allow us to relate the inverse gamma distribution to the FZZ formula.

\begin{lemma}\label{lem:special1}
	Let $x \in (-1,1) $ and $a \in \mathbb{R} \setminus \mathbb{Z} $. The following expansion holds:
	\begin{align}
	\cos (a \arccos x) = \frac{a}{2} \sin(\pi a) \sum_{n=0}^{+\infty} \frac{(-2)^{n-1}}{\pi n!} \Gamma(\frac{n +a}{2}) \Gamma(\frac{n -a}{2}) x^n.
	\end{align}
\end{lemma}
\begin{proof}
	The function $T_a(x) = \cos (a \arccos x)$ is a generalization of Chebyshev polynomials, which reduces to the usual Chebyshev polynomials when $a$ is a positive integer. For $x \in (-1, 1)$,   Taylor expansion gives:
	\begin{align*}
	T_a(x) &=  a \sum_{n=0}^{+\infty} \frac{(2x)^n}{2(n !)} \cos(\frac{\pi}{2}(a-n)) \frac{\Gamma(\frac{a+n}{2})}{\Gamma(\frac{a-n}{2}+1)} \\
	& = a\sum_{n=0}^{+\infty} \frac{(2x)^n}{2 \pi n !} \cos(\frac{\pi}{2}(a-n)) \Gamma(\frac{a+n}{2}) \Gamma( - \frac{a-n}{2}) \sin(\pi \frac{a-n}{2} + \pi) \\
	& =  \frac{a}{2} \sin(\pi a) \sum_{n=0}^{+\infty} \frac{(-2)^{n-1}}{\pi n!} \Gamma(\frac{n +a}{2}) \Gamma(\frac{n -a}{2}) x^n. \qedhere
	\end{align*}

\end{proof}

\begin{lemma}\label{lem:special2}
	Let $a,b,c,\lambda \in \mathbb{R}$ satisfying $a>0$, $b > -1$, $c < \frac{b}{2} - \frac{1}{2}$, $0 < \lambda < 2 a$. Then the following identity holds:
	\begin{align}
	\int_0^{\infty} dy \; y^b  e^{- \lambda y}  \int_0^{\infty} dt \; t^c \exp \left( - y^2 t - \frac{a^2}{t} \right) = \frac{1}{2} \sum_{n=0}^{+\infty} \frac{(-\lambda)^n}{ n!} a^{-n + 2c - b + 1} \Gamma(\frac{n}{2} + \frac{b}{2} - c - \frac{1}{2})  \Gamma(\frac{n}{2} + \frac{b}{2} + \frac{1}{2}).
	\end{align}
\end{lemma}

\begin{proof}
	To compute the double integral we expand $e^{- \lambda y}$ and then perform the change of variable $u = y^2 t$, $v = \frac{a^2}{t}$. Hence $y = \frac{\sqrt{uv}}{a}$, $t = \frac{a^2}{v}$, $dy dt  = \frac{a}{2 v \sqrt{uv}} du dv$. Therefore:
	\begin{align*}
	& \int_0^{\infty} dy \; y^b  e^{- \lambda y}  \int_0^{\infty} dt \; t^c \exp \left( - y^2 t - \frac{a^2}{t} \right) \\
	& = \sum_{n =0}^{+\infty} \frac{(-\lambda)^n}{n!} \int_0^{\infty} dy \; y^{b+n}    \int_0^{\infty} dt \; t^c \exp \left( - y^2 t - \frac{a^2}{t} \right) \\
	& = \sum_{n =0}^{+\infty}  \frac{(-\lambda)^n}{n!} \int_0^{\infty} \int_0^{\infty} du  dv  \frac{a}{2 v \sqrt{uv}} a^{2c} v^{-c} a^{-b -n} (uv)^{\frac{b}{2} + \frac{n}{2}} e^{-u -v}\\
	& = \frac{1}{2} \sum_{n =0}^{+\infty}  \frac{(-\lambda)^n}{n!} a^{-n + 2c - b + 1} \int_0^{\infty} \int_0^{\infty} du  dv \; u^{\frac{b}{2} + \frac{n}{2} - \frac{1}{2}} v^{\frac{b}{2} + \frac{n}{2} - c -\frac{3}{2}} e^{-u -v}\\
	& = \frac{1}{2} \sum_{n=0}^{+\infty} \frac{(-\lambda)^n}{ n!} a^{-n + 2c - b + 1} \Gamma(\frac{n}{2} + \frac{b}{2} - c - \frac{1}{2})  \Gamma(\frac{n}{2} + \frac{b}{2} + \frac{1}{2}).
	\end{align*}
	The conditions on the parameters are such to make the above integrals and series converge. For $n=0$ in the sum, the integrals over $u$ and $v$ converge respectively if $ b > -1$ and $ c < \frac{b}{2} - \frac{1}{2}$. The condition $\lambda >0$ is obvious. The condition $ \lambda < 2 a$ is required for the series over $n$ to converge, as can be checked on the last line.
\end{proof}

 \begin{proof}[Proof of Proposition~\ref{prop:inverse-gamma-a}]
 To start off we assume that the parameters $\mu, \mu_B >0$ are chosen such that 
 \begin{equation}\label{eq:cond_mu_mu_B}
 \frac{\mu^2_B}{\mu} \sin \frac{\pi \gamma^2}{4} \in (0,1).
 \end{equation}
  This constraint will then be lifted by an analyticity argument in the variable $\mu_B$. Applying integration by parts to the integration over $\ell$ in~\eqref{eq:FZZ-final1},   we get
 	\begin{align}
 	 \left \langle  e^{\alpha \phi(i)}  \right \rangle
 	&= \frac{\gamma C}{2(\alpha-Q)} \int_0^\infty \ell^{\frac2\gamma(\alpha -Q)} \cM_{1,0}^\disk(\alpha; 1)^{\#}[( 2 \mu \ell A + \mu_B ) e^{-\mu \ell^2 A-\mu_B \ell}] \, d\ell \nonumber\\
 	& = C' \int_0^{\infty} y^{\frac{2}{\gamma}(\alpha-Q)} \cM_{1,0}^\disk(\alpha; 1)^{\#} \left[(\lambda + 2 y A) e^{-\lambda y - y^2 A} \right] dy,
 	\end{align}
 	where $y = \sqrt \mu \ell$, $C' = - \frac{\gamma C}{2(Q- \alpha)} \mu^{\frac{Q -\alpha}{\gamma}}$ and  $\lambda = \frac{\mu_B}{\mu^{1/2}} $.
 	Now we use Proposition~\ref{prop:inverse-a} which requires to assume $\alpha < Q -\frac{\gamma}{4}$ and we apply twice Lemma \ref{lem:special2} with $(a,b,c,\lambda)$ equal to 
 	$$  \left( \frac{1}{ 2  \sqrt{\sin \frac{\pi \gamma^2}{4}} }, \frac{2}{\gamma}(\alpha - Q), \frac{2}{\gamma}(\alpha -Q) -1, \lambda \right) \quad \text{and} \quad  \left( \frac{1}{ 2 \sqrt{ \sin \frac{\pi \gamma^2}{4}} }, \frac{2}{\gamma}(\alpha - Q) + 1, \frac{2}{\gamma}(\alpha -Q), \lambda \right).$$
	
	The constraints on $(a,b,c,\lambda)$ required by Lemma \ref{lem:special2} are satisfied provided that $\alpha \in (\frac{2}{\gamma}, Q)$ and that equation \eqref{eq:cond_mu_mu_B} holds.
 Hence one arrives at the following power series expansions in the parameter $\lambda$: 
\begin{align}\label{sum1}
&C' \int_0^{\infty} y^{\frac{2}{\gamma}(\alpha-Q)} \cM_1^\disk(\alpha; 1)^{\#} \left[ \lambda e^{-\lambda y - y^2 A} \right] dy \\
& \quad = -\frac{C''}{2} \sum_{n=0}^{\infty} \frac{(-\lambda)^{n+1}}{n!} \left( 2 \sqrt{ \sin \frac{\pi \gamma^2}{4}} \right)^{n + \frac{2}{\gamma}(Q -\alpha) + 1} \Gamma( \frac{n}{2} + \frac{1}{\gamma}(\alpha - Q) + \frac{1}{2}) \Gamma( \frac{n}{2} + \frac{1}{\gamma}( Q -\alpha) + \frac{1}{2}), \nonumber
\end{align}
and 
\begin{align}\label{sum2}
&C' \int_0^{\infty} y^{\frac{2}{\gamma}(\alpha-Q)} \cM_1^\disk(\alpha; 1)^{\#} \left[ 2 y A e^{-\lambda y - y^2 A} \right] dy \\
& \quad = C'' \sum_{n=0}^{\infty} \frac{(- \lambda)^n }{n!} \left( 2 \sqrt{ \sin \frac{\pi \gamma^2}{4} } \right)^{n + \frac{2}{\gamma}(Q -\alpha) } \Gamma( \frac{n}{2} + \frac{1}{\gamma}(\alpha - Q) + 1) \Gamma( \frac{n}{2} + \frac{1}{\gamma}( Q -\alpha) ), \nonumber
\end{align}
where $C'' = \frac{ (4 \sin \frac{\pi \gamma^2}{4} )^{\frac{2}{\gamma}(\alpha -Q)} }{\Gamma( \frac{2}{\gamma}(Q -\alpha) )} C'$.
By shifting the sum in \eqref{sum1} to start at $n=1$ and applying in \eqref{sum2} the following identity on the Gamma function
$$ \Gamma( \frac{n}{2} + \frac{1}{\gamma}(\alpha - Q) + 1) = \left(  \frac{n}{2} + \frac{1}{\gamma}(\alpha - Q)  \right)  \Gamma( \frac{n}{2} + \frac{1}{\gamma}(\alpha - Q) ), $$
we obtain
 	\begin{align*}
 	 \left \langle  e^{\alpha \phi(i)}  \right \rangle &= -\frac{C''}{2} \sum_{n=0}^{\infty} \frac{(-\lambda)^{n+1}}{n!} \left( 2 \sqrt{ \sin \frac{\pi \gamma^2}{4}} \right)^{n + \frac{2}{\gamma}(Q -\alpha) + 1} \Gamma( \frac{n}{2} + \frac{1}{\gamma}( Q -\alpha) + \frac{1}{2}) \Gamma( \frac{n}{2} + \frac{1}{\gamma}(\alpha - Q) + \frac{1}{2})  \\
 	& +C'' \sum_{n=0}^{\infty} \frac{(- \lambda)^n }{n!} \left( 2 \sqrt{ \sin \frac{\pi \gamma^2}{4} } \right)^{n + \frac{2}{\gamma}(Q -\alpha) } \Gamma( \frac{n}{2} + \frac{1}{\gamma}( Q -\alpha) ) \Gamma( \frac{n}{2} + \frac{1}{\gamma}(\alpha - Q) + 1) \\
 	& = C'' \frac{\alpha -Q}{\gamma}  \sum_{n=0}^{\infty} \frac{(- \lambda )^n }{n!}  \left( 2 \sqrt{ \sin \frac{\pi \gamma^2}{4} } \right)^{n + \frac{2}{\gamma}(Q -\alpha) } \Gamma( \frac{n}{2} + \frac{1}{\gamma}(\alpha - Q))  \Gamma( \frac{n}{2} + \frac{1}{\gamma}( Q -\alpha) ),
 	\end{align*}
 To then obtain the claimed formula $U_{\mathrm{FZZ}}(\alpha)$, one simply needs to apply Lemma \ref{lem:special1} with $a = \frac{2}{\gamma}(\alpha -Q)$ and $x=  \lambda \sqrt{\sin \frac{\pi \gamma^2}{4}}$. Thanks again to \eqref{eq:cond_mu_mu_B} and the fact that $\alpha \in (\frac{2}{\gamma}, Q)$, the conditions to apply the lemma are satisfied. This gives
 	\begin{align*}
 	 \left \langle  e^{\alpha \phi(i)}  \right \rangle 
 	& = 2 C'' \frac{\alpha -Q}{\gamma} \frac{\pi \gamma}{(\alpha - Q)} \frac{1}{\sin(\frac{2 \pi}{\gamma}(Q -\alpha))} \left( 2 \sqrt{ \sin \frac{\pi \gamma^2}{4} } \right)^{ \frac{2}{\gamma}(Q -\alpha) } \cos \left(\frac{2}{\gamma}(\alpha - Q) \arccos \left(\lambda \sqrt{\sin \frac{\pi \gamma^2}{4}}\right) \right)\\
 	& = 2  C'' \Gamma(\frac{2}{\gamma}(Q -\alpha)) \Gamma(1 - \frac{2}{\gamma}(Q -\alpha)) \left( 2 \sqrt{ \sin \frac{\pi \gamma^2}{4} } \right)^{ \frac{2}{\gamma}(Q -\alpha) } \cos ((\alpha - Q) \pi s),
 	\end{align*}
 where the definition of $s$ was given in \eqref{eq:def_s}; see also  \eqref{recall_s}. We compute:
 	\begin{align*}
 	&2C'' \Gamma(\frac{2}{\gamma}(Q -\alpha)) \Gamma(1 - \frac{2}{\gamma}(Q -\alpha)) \left( 2 \sqrt{ \sin \frac{\pi \gamma^2}{4} } \right)^{ \frac{2}{\gamma}(Q -\alpha) } \\
 	&= \frac{4}{\gamma} ( \sin \frac{\pi \gamma^2}{4} )^{\frac{1}{\gamma}(\alpha -Q)}     \mu^{\frac{Q -\alpha}{\gamma}}   2^{-\frac{\alpha^2}2} \left( 2^{-\frac\gamma2\alpha }   \right)^{\frac2\gamma(Q-\alpha)}   \left( \frac{\pi}{\Gamma(1-\frac{\gamma^2}4)} \right)^{\frac2\gamma(Q-\alpha)} \Gamma( \frac\gamma2(\alpha-\frac\gamma2))  \Gamma( \frac{2\alpha}{\gamma} - \frac{4}{\gamma^2} -1) \\
 	&= \frac{4}{\gamma}     2^{-\frac{\alpha^2}2} \left( \frac{ \pi \mu}{ 2^{\gamma \alpha }}  \frac{\Gamma(\frac{\gamma^2}{4})}{\Gamma(1 -\frac{\gamma^2}{4})} \right)^{\frac{Q - \alpha}{\gamma} } \Gamma( \frac{\gamma \alpha}{2} - \frac{\gamma^2}{4} )  \Gamma( \frac{2\alpha}{\gamma} - \frac{4}{\gamma^2} -1)=U_{\mathrm{FZZ}}(\alpha)/\cos ((\alpha - Q) \pi s).
 	\end{align*}
 	 Putting everything together we conclude that $\left \langle  e^{\alpha \phi(i)}  \right \rangle = U(\alpha) =U_{\mathrm{FZZ}}(\alpha)$ provided that the condition \eqref{eq:cond_mu_mu_B} holds. 
	 
	 We will now extend the result by analyticity in $\mu_B$ to the full range of $\mu_B \in(0, +\infty)$. Notice $U_{\mathrm{FZZ}}(\alpha)$ contains the factor $\cos( (\alpha - Q) \pi s)$ where $s$ and $\mu_B$ are related by:
	\begin{align}\label{recall_s}
	\mu_B = \mu_B(s) = \sqrt{\frac{\mu}{ \sin \frac{\pi \gamma^2}{4}}} \cos \frac{\pi \gamma s}{2}.
	\end{align}
We want to find an open domain $D \subset \mathbb{C}$ of $s$ where the map $s \mapsto \mu_B(s)$ will be  biholomorphic onto its image. Since $\frac{d}{ds}_{\vert s=0} \cos \frac{\pi \gamma s}{2} =0  $, we cannot include $0$ in $D$. Since $\cos \frac{\pi \gamma s}{2} = 1 - \frac{\pi^2 \gamma^2}{8} s^2 +o(s^2)$ and the second order $s^2$ is non-zero, we can include the sector $S_{\epsilon, \delta} := \{ r e^{i \theta} \vert \theta \in (-\epsilon, \frac{\pi}{2} + \epsilon), 0 < r< \delta \} $ for two small parameters $\epsilon, \delta >0$. Indeed in $S_{\epsilon, \delta}$ the angle $\theta$ is contained in an interval of length strictly less than $\pi$. We also want $D$ to contain an open neighborhood in $\mathbb{C}$ of $i(\frac{\delta}{2}, +\infty)$ where we wish to extend the equality $U(\alpha) = U_{\mathrm{FZZ}}(\alpha) $ and an open neighborhood of $(\frac{\delta}{2}, \frac{1}{2\gamma})$ were we have already established it. From all these considerations we choose:
\begin{align*}
D := S_{\epsilon, \delta} \cup  \left \{ x + i y \vert x \in (\frac{\delta}{2}, \frac{1}{2 \gamma}), y \in  (-\frac{\delta}{4}, \frac{\delta}{4})  \right \} \cup \left \{ x + i y \vert  x \in (-\frac{\delta}{4}, \frac{\delta}{4}), y \in (\frac{\delta}{2}, + \infty)  \right \}.
\end{align*}
For $ \epsilon, \delta$ small, the map $s \mapsto \mu_B(s)$ is then a biholomorphic map from $D$ onto its image, an open domain we call $D'$.
Now the exact formula $U_{\mathrm{FZZ}}(\alpha) $ is clearly an analytic function of $\mu_B$ on $D'$. We must argue the same is true for the probabilistic definition $U(\alpha)= \left \langle  e^{\alpha \phi(i)}  \right \rangle$. For all $\mu_B \in D'$, one clearly has $\Re(\mu_B)>0$ which by Lemma \ref{lem:U-bound} gives finiteness of $\left| \left \langle  e^{\alpha \phi(i)}  \right \rangle \right|$.
Moreover, $\LF_\bbH^{(\alpha, i)}\left[  \left|\nu_\phi(\R)e^{-\mu \mu_\phi(\bbH) -\mu_B \nu_\phi(\R)}\right|\right] $ is also finite in the same region. 
This implies that $\LF_\bbH^{(\alpha, i)}[ - \mu_B \nu_\phi(\R)e^{-\mu \mu_\phi(\bbH) - \mu_B \nu_\phi(\R)}]$ is analytic in $\mu_B$ when $\Re (\mu_B)>0$.
Taking anti-derivative in $\mu_B$ and use Fubini theorem to interchange the integral in $\mu_B$ and $\LF_\bbH^{(\alpha, i)}$, 
we see that $\left \langle  e^{\alpha \phi(i)}  \right \rangle=\LF_\bbH^{(\alpha, i)}[e^{-\mu \mu_\phi(\bbH) - \mu_B \nu_\phi(\R)}-1]$ is analytic in $\mu_B$ when $\Re (\mu_B)>0$. 
Since the equality of analytic functions $U(\alpha) = U_{\mathrm{FZZ}}(\alpha)$ in the variable $\mu_B$ holds on the interval given by \eqref{eq:cond_mu_mu_B} which is contained in the open set $D'$, it holds on all of $D'$. In particular it holds for $ \frac{\mu_B}{\sqrt{\mu}} \sqrt{\sin \frac{\pi \gamma^2}{4} } \in (1,+\infty)$. By continuity we recover the equality at the special value $ \frac{\mu_B}{\sqrt{\mu}} \sqrt{\sin \frac{\pi \gamma^2}{4} } =1$. This completes the proof.
  \end{proof}

\subsection{Analyticity in $\alpha$ and the proof of Theorem~\ref{thm:FZZ-physics}}\label{subsec:analytic}
In this section we prove that $U(\alpha)$ is complex analytic in $\alpha$ around   $(\frac2{\gamma}, Q)$. The proof utilizes the method given first in \cite{DOZZ_proof} and in \cite{RZ_boundary}, adapted to the case where there are both area and boundary GMC measures in the correlation function. Together with Proposition~\ref{prop:inverse-gamma-a} this will conclude the proof of Theorem \ref{thm:FZZ-physics}. We will then prove Corollary~\ref{cor_truncation} and Theorem~\ref{thm:inverse-gamma} that extend the range of $\alpha$ from  $(\frac2{\gamma}, Q)$ to  $(\frac{\gamma}2, Q)$.
\begin{proposition}\label{prop:analytic}
 	For any compact set $K \subset (\frac{2}{\gamma}, Q)$, the function $\alpha \rightarrow U(\alpha)$ is complex analytic on a complex neighborhood of the set $K$.
\end{proposition}  
\begin{proof} To start we use the following identity coming from Lemma \ref{lem:U-bound}
 \begin{equation}
 \left \langle  e^{\alpha \phi(2 i)}  \right \rangle = \frac{U(\alpha)}{\left|\Im 2i \right|^{2 \Delta_{\alpha}}},
 \end{equation}
 to place the bulk insertion at the location $2i$, and will prove that $ \left \langle  e^{\alpha \phi(2 i)}  \right \rangle$ is analytic in $\alpha$. 
Fix a $r_0$ such that $e^{-r_0} <\frac12$.  Let $h$ be the free boundary GFF on $\mathbb{H}$, which is normalized to have zero average over the upper half unit circle. 
Thanks to the Markov property of the field, see for instance \cite[Theorem 5.5]{berestycki-lqg-notes}, we can write the decomposition  $h = h_1 + h_2 $, where $h_1$ is a Dirichlet GFF on $ B(2i, e^{-r_0})$ and $h_2$ is a field which is independent of $h_1$ and smooth in a neighborhood of $2i$. Notice we placed the insertion at $2i$ so $ B(2i, e^{-r_0})$ does not intersect the upper half unit circle. 
 Consider now the one-dimensional process $h_{1,r}(2i)$ obtained by taking the mean of $h_1(2i)$ over the circles of radius $e^{-r}$ centered at $2i$, assuming $r \geq r_0$. The main fact we will use about this field is that $h_{1,r+t}(2i)-h_{1,r}(2i)$ is a Brownian motion independent of $(h(z))_{z \in \mathbb{H}_r }$, where $\mathbb{H}_r := \mathbb{H} \backslash B(2i, e^{-r})$. Record also that $\E[h_{1,r}(2i)^2] =r -r_0$  and the notation $h_r := h_{1,r} + h_2 $. From these facts we can deduce that:
\begin{align*}
&\mathbb{E} \left[ e^{\alpha h_{1,r+1}(2i) - \frac{\alpha^2}{2} \mathbb{E}[h_{1,r+1}(2i)^2] } F\left( (h(z))_{z \in \mathbb{H}_r } \right) \right] \\ 
&= e^{ \frac{\alpha^2}{2} \mathbb{E}[h_{1,r}(2i)^2] - \frac{\alpha^2}{2} \mathbb{E}[h_{1,r+1}(2i)^2] } \mathbb{E} \left[ e^{\alpha h_{1,r+1}(2i) - \alpha h_{1,r}(2i)} \right] \mathbb{E} \left[ e^{\alpha h_{1,r}(2i) - \frac{\alpha^2}{2} \mathbb{E}[h_{1,r}(2i)^2] } F\left( (h(z))_{z \in \mathbb{H}_r }\right) \right]\\
& = \mathbb{E} \left[ e^{\alpha h_{1,r}(2i) - \frac{\alpha^2}{2} \mathbb{E}[h_{1,r}(2i)^2] } F\left( (h(z))_{z \in \mathbb{H}_r } \right) \right]\\
\Rightarrow \: & \mathbb{E} \left[ e^{\alpha h_{r}(2i) - \frac{\alpha^2}{2} \mathbb{E}[h_{r}(2i)^2] } F\left( (h(z))_{z \in \mathbb{H}_r } \right) \right] = \mathbb{E} \left[ e^{\alpha h_{r+1}(2i) - \frac{\alpha^2}{2} \mathbb{E}[h_{r+1}(2i)^2] } F\left( (h(z))_{z \in \mathbb{H}_r } \right) \right].
\end{align*}
Now we can obtain $ U(\alpha) $ from the following limit
$U(\alpha) = 2^{-\frac{\alpha^2}{2}} \lim_{r \to \infty} U_r(\alpha)$,
where we have introduced:
\begin{align*}
&U_r(\alpha) = \int_{\mathbb{R}} dc e^{(\alpha - Q) c} \mathbb{E} \left[ e^{\alpha h_{r}(2i) - \frac{\alpha^2}{2} \mathbb{E}[h_{r}(2i)^2] }   \left( \exp \left( -  e^{\gamma c}  \mu_{h}(\mathbb{H}_r ) - e^{\frac{\gamma}{2} c}  \nu_{h}(\mathbb{R})  \right) -1 \right) \right].
\end{align*}
When $\alpha$ is a complex number, we write $\alpha = a + i b$. We want to prove there exists a complex neighborhood $V$ containing the set $K$ such that for any compact set $K'$ contained in $ V$, $U_r(\alpha)$ converges uniformly as $r \rightarrow + \infty$ over $K'$. 
Setting $\tilde h(z) := h(z) + a G_{\mathbb{H}}(z, 2 i)$, we have  the following:
\begin{align*}
&|U_{r+1}(\alpha) - U_{r}(\alpha)|\\ \nonumber
&= \int_{\mathbb{R}} dc  \left| e^{(\alpha - Q) c} \mathbb{E} \left[ e^{\alpha h_{r+1}(2i) - \frac{\alpha^2}{2} \mathbb{E}[h_{r+1}(2i)^2] } \exp \left( - e^{\frac{\gamma}{2} c} \nu_{h}(\mathbb{R}) \right)   \left( \exp \left( -  e^{\gamma c}  \mu_{h}(\mathbb{H}_{r+1} )     \right) - \exp \left( -  e^{\gamma c}  \mu_{h}(\mathbb{H}_r )  \right)  \right) \right] \right|\\ \nonumber
&\le  C \,e^{\frac{r+1}{2} b^2} \int_{\mathbb{R}} dc e^{(a - Q) c} \mathbb{E} \left[ \exp \left( - e^{\frac{\gamma}{2} c} \nu_{ \tilde h}(\mathbb{R}) \right)   \left| \exp \left( -  e^{\gamma c}  \mu_{\tilde h}(\mathbb{H}_{r+1} )     \right) - \exp \left( -  e^{\gamma c}  \mu_{\tilde h}(\mathbb{H}_r ) \right)  \right|   \right]\\
&\le  C \,e^{\frac{r+1}{2} b^2} \int_{\mathbb{R}} dc e^{(a - Q) c} \mathbb{E} \left[ \exp \left( -  e^{\gamma c}  \mu_{\tilde h}(\mathbb{H}_{r+1} )     \right) - \exp \left( -  e^{\gamma c}  \mu_{\tilde h}(\mathbb{H}_r ) \right)  \right]\\
&= C' e^{\frac{r+1}2b^2} \E\left[ \mu_{\wt h}(\mathbb{H}_{r+1})^{\frac{Q-a}\gamma} - \mu_{\wt h}(\mathbb{H}_{r})^{\frac{Q-a}\gamma}  \right].
\end{align*}
 From the second line to the third, we have applied the Girsanov theorem to the real part of $\alpha h_{r+1}(2i)$, before moving the absolute value inside the expression. The last equality follows from changing the integration variable from $c$ to $y=e^{\gamma c}\mu_{\wt h}(\bbH_i)$ for $i = r, r+1$. Thus,
\[|U_{r+1}(\alpha) - U_{r}(\alpha)| \leq   C' e^{\frac{r+1}2b^2} \E\left[ \mu_{\wt h}(\bbH_{r+1} \backslash \bbH_r)^{\frac{Q-a}\gamma} \right] \le C'' e^{\frac{r+1}2b^2} e^{ra(Q-a)}\E\left[ \mu_{ h}(\bbH_{r+1} \backslash \bbH_r)^{\frac{Q-a}\gamma} \right]. \]
The first inequality follows from $(a+b)^s \leq a^s + b^s$ for $s \in (0,1)$, and the second since $\wt h$ and $h$ differ by roughly $ar$ on $\bbH_{r+1} \backslash \bbH_r$. 

By the multifractal scaling of GMC, see e.g.~\cite[Section 3.6]{berestycki-lqg-notes} or \cite[Section 4]{rhodes-vargas-review}, we have
\[ \E[\mu_h(\bbH_{r+1} \backslash \bbH_r)^{\frac{Q-a}\gamma}] \asymp e^{(-\gamma Q q + \frac{\gamma^2q^2}2)r}, \qquad q = \frac{Q-a}\gamma.\]
Combining with the previous inequality, we deduce that $|U_{r+1}(\alpha) - U_r(\alpha)| \lesssim e^{r(\frac{b^2}2 -\frac12(Q-a)^2)}$. Choosing the open set $V$ in such a way that $\frac{b^2}{2}  < \frac12(Q-a)^2$ always holds, all the inequalities we have done before hold true and hence we have shown that $U_r(\alpha)$ converges locally uniformly. Since $U_r(\alpha) $ is complex analytic in $\alpha$, this proves the analyticity result of $U(\alpha)$.
\end{proof}

\begin{proof}[Proof of Theorem \ref{thm:FZZ-physics} ]
	From the exact formula,  $U_{\mathrm{FZZ}}(\alpha)$ is a meromorphic function of $\alpha$ on $\C$. By uniqueness of the meromorphic continuation, Propositions~\ref{prop:inverse-gamma-a}
	and~\ref{prop:analytic}, we have $U(\alpha) = U_{\mathrm{FZZ}}(\alpha)$ for all  $\alpha \in  (\frac{2}{\gamma}, Q)$.
\end{proof}

 \begin{proof}[Proof of  Theorem~\ref{thm:inverse-gamma}] Proposition~\ref{prop:inverse-a} proves the theorem for $\alpha \in (\frac\gamma2, Q-\frac\gamma4)$. To complete the proof we  consider $\alpha \in [Q-\frac\gamma4, Q)$. Let $L$ be sampled from the power law $\frac2\gamma 2^{-\frac{\alpha^2}2} \ol U_0(\alpha)\ell^{\frac2\gamma(\alpha-Q)-1} \, d\ell$ where $\ol U_0$ is as in Lemma~\ref{lem-len-LF}, and let $A =L^2 X $ where $X$ is sampled from an independent inverse gamma distribution with shape $\frac2\gamma(Q-\alpha)$ and scale $\frac{1}{4 \sin \frac{\pi \gamma^2}4}$. Let $\Pi$ be the joint law of $(A,L)$. Proposition~\ref{prop:inverse-gamma-a} proves that $\Pi[e^{-\mu A - \mu_B L} - 1] = U_{\mathrm{FZZ}}(\alpha)$ for $\alpha \in (\frac2\gamma, Q-\frac\gamma4)$, but the argument works equally well for our range $\alpha \in [Q - \frac\gamma4,Q)$. 
 Thus $\LF_\bbH^{(\alpha, i)}[e^{-\mu \mu_\phi(\bbH) - \mu_B \nu_\phi(\R)}-1]=\Pi[e^{-\mu A- \mu_B L}-1]$ by Theorem~\ref{thm:FZZ-physics}, so 
 $\LF_\bbH^{(\alpha, i)}[\nu_\phi(\R)e^{-\mu \mu_\phi(\bbH) - \mu_B \nu_\phi(\R)}]=\Pi[Le^{-\mu A- \mu_B L}]$. Therefore by standard arguments of characterization of a law by the Laplace transform, the joint law of $(\mu_\phi(\bbH), \nu_\phi(\R))$ under $\LF_\bbH^{(\alpha, i)}$ is the same as that of 
 $(A,L)$ under $\Pi$, which concludes the proof.
 \end{proof}

Since we have now established Theorem~\ref{thm:inverse-gamma} for the whole range $\alpha \in (\frac\gamma2, Q)$, we have the following extension of the probabilistic definition of $U(\alpha)$ and of Theorem~\ref{thm:FZZ-physics}, for which we omit the proof.
\begin{corollary}\label{cor_truncation}
	Let $\alpha \in ( \frac{\gamma}{2}, Q)$. For $k \in \mathbb{N}$ and $\alpha \in (Q - \frac{\gamma(k+1)}{2}, Q - \frac{\gamma k}{2})$, we define $U(\alpha)$ by:
	\begin{equation}\label{definition_by_truncation}
	U(\alpha) = \frac2\gamma 2^{-\frac{\alpha^2}2} \ol U_0(\alpha)  \int_0^\infty \ell^{\frac2\gamma(\alpha-Q)-1} \E \left[e^{-\mu \ell^2 A   - \mu_B \ell} - \sum_{i=0}^k c_i(A) \ell^i \right] \, d\ell.
	\end{equation}
	Here the expectation is with respect to the law of $A$ which is distributed according to the inverse gamma distribution with shape $\frac2\gamma(Q-\alpha)$ and scale $\frac1{4\sin^2 \frac{\pi \gamma^2}4}$. $\ol U_0(\alpha)$ is given by the explicit formula \eqref{eq:U0-explicit}. Lastly the constants $c_i(A)$ are specified by the following expansion in powers of $\ell$:
	\begin{align}
	e^{-\mu \ell^2 A   -\mu_B \ell} = \sum_{i=0}^k c_i(A) \ell^i + O(\ell^{i+1}).
	\end{align}
	Then this definition of $U(\alpha)$ is the meromorphic extension of \eqref{eq:def_U_alpha} on $(Q - \frac{\gamma(k+1)}{2}, Q - \frac{\gamma k}{2})$ and it obeys $U(\alpha) = U_{\mathrm{FZZ}}(\alpha)$.
\end{corollary}
 
Finally, we provide a commonly used form of the FZZ formula; see \cite[(2.44)]{FZZ}.
\begin{proposition}\label{prop-FZZ-laplace}
	For $\alpha \in (\frac\gamma2, Q)$ and $\ell>0$,  writing $A$ for the random area of a sample from $\cM_1^\disk(\alpha; \ell)$, we have 
	\[\cM_1^\disk(\alpha; \ell)[e^{-\mu A}] = \frac2\gamma 2^{-\frac{\alpha^2}2} \ol U_0(\alpha) \ell^{-1} \frac2{\Gamma(\frac2\gamma(Q-\alpha))} \left(\frac12 \sqrt{\frac{\mu}{\sin(\pi\gamma^2/4)}}  \right)^{\frac2\gamma(Q-\alpha)}K_{\frac2\gamma(Q-\alpha)} \left(\ell\sqrt{\frac{\mu}{\sin(\pi\gamma^2/4)}}  \right). \]
	Here, $K_\nu(x)$ is the modified Bessel function of the second kind; see e.g.\ \cite[Section 10.25]{NIST:DLMF}.
\end{proposition}
\begin{proof}
	By Theorem~\ref{thm:inverse-gamma}, writing $\beta = \frac1{4\sin \frac{\pi \gamma^2}4}$ and considering $\ell = 1$, we have
	\begin{align*}
	\cM_1^\disk(\alpha; 1)^\#[e^{-\mu A}] &= \frac{\beta^{\frac2\gamma(Q-\alpha)}}{\Gamma(\frac2\gamma(Q-\alpha))}\int_0^\infty t^{-\frac2\gamma(Q-\alpha) - 1} \exp^{-\frac\beta t - \mu t} \, dt \\
	 &= \frac2{\Gamma(\frac2\gamma(Q-\alpha))} (\mu \beta)^{\frac1\gamma (Q-\alpha)}K_{\frac2\gamma(Q-\alpha)} (\sqrt{4 \mu \beta}),
	\end{align*}
	where the last equality follows from a change of variables $s = \mu t$ and the integral identity $K_\nu(z) = \frac12 (\frac12z)^\nu \int_0^\infty \exp(-s - \frac{z^2}{4s}) \frac{ds}{s^{\nu + 1}}$ \cite[Eq.~10.32.10]{NIST:DLMF} with the choice $z = \sqrt{4 \mu \beta}$ and $\nu = \frac2\gamma(Q-\alpha)$. Since for general $\ell$ we have $\cM_1^\disk(\alpha; \ell)^\#[e^{-\mu A}] = \cM_1^\disk(\alpha; 1)^\#[e^{-\mu  \ell^2 A}]$,
	\[\cM_1^\disk(\alpha; \ell)^\#[e^{-\mu A}] =  \frac2{\Gamma(\frac2\gamma(Q-\alpha))} \ell^{\frac2\gamma(Q-\alpha)} \left(\frac12 \sqrt{\frac{\mu}{\sin (\pi \gamma^2/4)}}\right)^{\frac2\gamma (Q-\alpha)}K_{\frac2\gamma(Q-\alpha)} \left(\ell\sqrt{\frac{\mu}{\sin(\pi\gamma^2/4)}}  \right). \]
	By Lemma~\ref{lem-len-LF} we have $|\cM_1^\disk(\alpha; \ell)| = \frac2\gamma 2^{-\frac{\alpha^2}2} \ol U_0(\alpha) \ell^{\frac2\gamma(\alpha-Q)-1}$, so the result follows. 
\end{proof}
  

%% file: sec5.tex
\section{Proof of the SLE bubble zipper with a  $\gamma$-bulk insertion}\label{sec:bubble}
In this section we prove Theorem~\ref{thm-bubble-zipper}. For technical convenience, we will consider loops on $\bbH$ passing through $\infty$ instead of $0$. More precisely, 
let $\psi(z)=-z^{-1}$ and $\bub_\bbH(i,\infty)=\{\eta: \psi(\eta)\in \bub_\bbH(i,0)  \}$. 
We will find a  measure on $\bub_\bbH(i,\infty)$ with the desired property and then use $\psi$ to pull it back to get $\sm$  in Theorem~\ref{thm-bubble-zipper}.
The proof is a limiting argument  based on Lemma~\ref{lem:hx-embed}  and a three-disk variant of the conformal welding result in Theorem~\ref{thm:disk-welding}.
We set up the framework of the proof in Section~\ref{subsec:bubble-sketch} and carry out the details in Sections~\ref{subsec:decouple} and~\ref{subsec:delta0}.

\subsection{A conformal welding of three quantum disks}\label{subsec:bubble-sketch}
Set $\kappa=\gamma^2$. Let $\eta_1$ be an $\SLE_\kappa(\frac{\gamma^2}2-2, \frac{\gamma^2}2)$ curve on $(\bbH,0,\infty)$. Let $H^+_{\eta_1}$ the be right component of $\bbH\setminus \eta_1$. Conditioning on $\eta_1$, let $\eta_2$ be an $\SLE_\kappa(0;\frac{\gamma^2}2-2)$ on $(\bbH_{\eta_1}^+,0,\infty)$. 
We  denote $\cP(\bbH,0,\infty)$ as the law of $(\eta_1,\eta_2)$. For a general simply connected domain $D$ with two boundary points $(a,b)$, we write 
$\cP(D,a,b)$ as the conformal image of $\cP(\bbH,0,\infty)$.
As a straightforward extension of Theorem~\ref{thm:disk-welding}, we have the following conformal welding result.
\begin{theorem}\label{thm:3disk-welding} 
	Let $\ell, \ell'>0$ and  $(\bbH,\phi,0,\infty)$ be an embedding of a sample from $\cM^{\disk}_{0,2}(2+\gamma^2; \ell, \ell')$.
	Let $(\eta_1, \eta_2)$ be sampled from $\cP(\bbH, 0, \infty)$ independent of $\phi$.
	Let $\bbH^{1}_\eta$, $\bbH^{12}_\eta$ and  $\bbH^2_\eta$ be the left, middle, and right components of $\bbH\setminus (\eta_1\cup \eta_2)$, respectively. 
	The joint law of $(\bbH^{1}_\eta,\phi, 0,\infty)$, $(\bbH^{12}_\eta,\phi, 0,\infty)$, and  $(\bbH^2_\eta,\phi, 0,\infty)$ 
	viewed as marked quantum surfaces equals
	\begin{equation}\label{eq:welding}
		C\iint_0^\infty \cMtwo(\frac{\gamma^2}2; \ell;p) \times \cMtwo(2; p;q) \times \cMtwo (\frac{\gamma^2}2; q;\ell')\,  dp \, dq, \quad \textrm{ for some }C\in (0,\infty).
	\end{equation}
\end{theorem}
\begin{proof}
	This is \cite[Theorem 2.3]{ahs-disk-welding} when   $W_1=W_3=\frac{\gamma^2}{2}$ and $W_2=2$.
\end{proof}

Fix $\delta\in (0,\frac12)$. Sample $(\phi, \mathbf x)$ from $\LF_\bbH^{(\gamma, i)} \times dx$ and sample $(\eta_1, \eta_2)$ from $\cP(\bbH, \mathbf x, \infty)$, and restrict to the event that $\nu_\phi(\mathbf x,\infty)\in (\delta,\frac12)$, $\nu_\phi(\R)\in (1,2)$ and $i$ is in between $\eta_1$ and $\eta_2$. Here and later, we write $\nu_\phi(a,b)$ to denote the $\nu_\phi$-length of the interval $(a,b)$.
Let $M_\delta$ be the law of $(\phi, \mathbf x, \eta_1, \eta_2)$ (restricted to the aforementioned event). See Figure~\ref{fig:Mdelta}.

\begin{figure}[ht!]
	\begin{center}
		\includegraphics[scale=0.65]{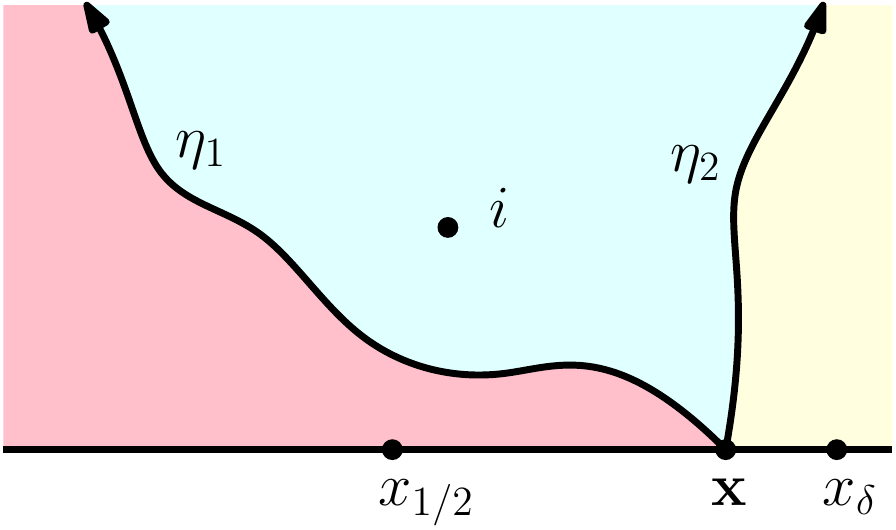}%
	\end{center}
	\caption{\label{fig:Mdelta} Illustration of $(\phi,\mathbf x,\eta_1,\eta_2)$ under $M_\delta$.
	 Sample $(\phi, \mathbf x)$ from $\LF_\bbH^{(\gamma, i)} \times dx$ and sample $(\eta_1, \eta_2)$ from $\cP(\bbH, \mathbf x, \infty)$. 	
	 To represent geometrically the condition $\nu_\phi(\mathbf x,\infty)\in (\delta,\frac12)$, consider the points $x_{1/2}$ and $x_\delta$ such that $\nu_\phi(x_{1/2},\infty)=1/2$ and 	$\nu_\phi(x_{\delta},\infty)=\delta$. Then $M_\delta$ is obtained by restricting to the event
	 $\mathbf x\in (x_{1/2},x_\delta)$, $i$ is in the middle of $\eta_1$ and $\eta_2$ and $\nu_\phi(\R)\in (1,2)$. 
	 Lemma \ref{lem:three-disks} describes the law of the three shaded quantum surfaces.
			}
\end{figure}

\begin{lemma}\label{lem:three-disks}
	There exists $C>0$ such that  for each $\delta \in (0, \frac12)$,  if $(\phi, \mathbf x, \eta_1, \eta_2)$ is sampled from $M_\delta$ then the law of the three  marked  quantum surfaces of $(\bbH, \phi, \eta_1, \eta_2, i, \mathbf x, \infty)$ bounded by $\eta_1$, $\eta_2$ and $\bdy\bbH$ is given by 
	\begin{equation}\label{eq:welding-special}
		C\int_{\delta}^{1/2} \int_{1-b}^{2-b} \int_0^\infty \int_0^\infty  \cMtwo(\frac{\gamma^2}2; a,p) \times  \cM^\disk_{1,2}(2;p,q) \times  \cMtwo(\frac{\gamma^2}{2}; q, b) \, dq\, dp \, da\, db.
	\end{equation}
\end{lemma}
\begin{proof}
	Lemma~\ref{lem:hx-embed} states that if we sample $(\phi, \mathbf x)$ from $\LF_\bbH^{(\gamma, i)} \times dx$ then the law of the quantum surface $(\bbH, \phi, i, \mathbf x, \infty)$ is $C \cM_{1, 2}^\disk(2+\gamma^2)$ for some constant $C$. If we further sample $(\eta_1, \eta_2)$ from $\cP(\bbH, \mathbf x, \infty)$ and restrict to the event that $i$ lies between $\eta_1$ and $\eta_2$, then 
	by Definition~\ref{def:disk-bulk} and Theorem~\ref{thm:3disk-welding}, the quantum surface $(\bbH, \phi, \eta_1,\eta_2, i, \mathbf x, \infty)$ has law 
	$$C\iiiint_0^\infty \cMtwo(\frac{\gamma^2}2; a ,p) \times  \cM^\disk_{1,2}(2; p,q) \times  \cMtwo(\frac{\gamma^2}{2}; q, b) \, dq\, dp \, da\, db$$ for some $C>0$. Further restricting to the event $\nu_\phi(\mathbf x,\infty)\in (\delta,\frac12)$ and $\nu_\phi(\R)\in (1,2)$ we conclude the proof.
\end{proof}

In Section~\ref{subsec:decouple},  we will show that the probability measure $M^{\#}_\delta$  that is proportional to $M_\delta$
concentrates on the event that $\nu_h(\mathbf x,\infty)$ is of order $\delta$ and 
$\log\mathbf x$ is of order $\log\delta^{-1}$. This effectively collapses  the right boundary of $(\bbH,\phi,\mathbf x,\infty)$. 
Building on this, we will prove the following proposition. Recall that for two probability measures $P,Q$ on the same measure space, the total variation distance between $P$ and $Q$ is $\sup_A |P(A) - Q(A)|$ where the supremum is taken over all measurable sets.
\begin{prop}\label{prop-apprx-zipper}
	For each $\delta\in (0,\frac12)$, let $\sm_\delta$ be the marginal law of  $(\mathbf x,\eta_1,\eta_2)$ under  $M^{\#}_\delta$. 
	Let $M$ be   $\LF_\bbH^{(\gamma, i)}$ restricted  to $\{ \nu_\phi(\R) \in (1, 2)\}$. 
	Then the total variational distance of   $M_{\delta}^\#$ and  $M^\# \times \sm_\delta$ tends to 0 as $\delta\to 0$. 
\end{prop}
	
In Section~\ref{subsec:delta0}, we prove that by sending $\delta\to 0$ in the integral~\eqref{eq:welding-special}, we get the conformal welding of $\QD_{1,1} $ and $\cMtwo(\frac{\gamma^2}2)$ as in Theorem~\ref{thm-bubble-zipper}. 
Moreover, $\sm_\delta$ from Proposition~\ref{prop-apprx-zipper} has a weak limit supported on loops rooted at $\infty$, 
whose pushforward under $z\mapsto -z^{-1}$ gives the measure $\sm$ in Theorem~\ref{thm-bubble-zipper}. 

Our approximation procedure (where we conformally weld three quantum disks and have one vanish in the $\delta \to 0$ limit) may seem more complicated than necessary, but it is advantageous for the following reasons. Firstly, in our setup the law of the field of $M_\delta$ is absolutely continuous with respect to $M$, so the statement and proof of Proposition~\ref{prop-apprx-zipper} can avoid modifying the field in some way. Secondly, $M_\delta$ arises from a setup where the interfaces and field are independent (i.e.\ $(\phi, \mathbf x) \sim \LF_\bbH^{(\gamma, i)} \times dx$ then $(\eta_1, \eta_2) \sim \cP(\bbH, \mathbf x, \infty)$). Finally, this approach allows us to avoid SLE estimates entirely.

\subsection{An approximate bubble zipper: proof of Proposition~\ref{prop-apprx-zipper}}\label{subsec:decouple}
Suppose we are in the setting of Lemma~\ref{lem:three-disks} and Proposition~\ref{prop-apprx-zipper}. We first give 
a simple description of the Radon-Nykodim derivative of $\frac{dM_{\delta}}{dM}$.
For $x\in \R$, let $p(x)$ be the conditional  probability that $i$ is in between $\eta_1$ and $\eta_2$ given $\mathbf x=x$. 
Then $p(x)$ is an even function determined by the SLE measure $\cP(\bbH, x,\infty)$.
Define the function $f: \R \to \R$ by 
\begin{equation}\label{def:f}
	f(x) = \left\{
	\begin{array}{ll}
		\int_0^x p(y) \, dy  & \mbox{if } x \geq 0 \\
		- \int_x^0 p(y) \, dy & \mbox{if } x < 0.
	\end{array}
	\right.
\end{equation}
For each $\ell \in (0, \nu_\phi(\R))$, let 
\begin{equation}\label{eq:xell}
	x_\ell=\inf\{x\in \R: \nu_\phi(x, \infty) = \ell\}.
\end{equation}
\begin{lemma}\label{cor-RN}
	For a non-negative measurable function $F$ on $H^{-1}(\bbH)$, we have
	\[
	M_\delta[F(\phi)]	=\int[ f(x_\delta) - f(x_{1/2})] F(\phi) \,  dM.
	\]
\end{lemma}
\begin{proof}
		By definition, if we sample $(\phi, \mathbf x)$ from $M \times dx$ and then sample $(\eta_1, \eta_2)$ from $\cP(\bbH, \mathbf x, \infty)$, and restrict to the event that $\mathbf x \in (x_{1/2}, x_\delta)$ and $i$ is in between $\eta_1$ and $\eta_2$, then the law of $(\phi, \mathbf x, \eta_1, \eta_2)$ is $M_\delta$. 
		Conditioning on $(\phi, \mathbf x)$, the conditional probability that $i$ lies between $\eta_1$ and $\eta_2$ is  $\int_{x_{1/2}}^{x_\delta} p(x)\, dx = f(x_{\delta}) - f(x_{1/2})$.
		This gives  the result.  	
\end{proof}

The hardest step in proving Proposition~\ref{prop-apprx-zipper} is to show that the $M_\delta^\#$-law of $\phi$ converges to that of $M^\#$, i.e., to show that a.s.\  $\frac{dM_\delta^\#}{dM^\#}(\phi) = |M|(f(x_\delta) - f(x_{1/2}))/|M_\delta| \to 1$ as $\delta \to 0$. 
In the argument below Lemma~\ref{lem-total-mass} will give that   $|M_\delta|$ diverges as $\log \delta^{-1}$.  Lemma~\ref{lem-as-tail} will give that $x_\delta = \delta^{-\frac2{\gamma Q} + o(1)}$ a.s. Lemmas~\ref{lem-f-upper} and~\ref{lem-f-lower} will give the asymptotic growth of $f(x_\delta)$. Putting them together we will get the desired limit.

Let  $A=\nu_h(-\infty,\mathbf x)$ and $B=\nu_h(\mathbf x,\infty)$. Let $P$ and $Q$  be the quantum lengths of $\eta_1$ and $\eta_2$ with respect to $\phi$. 
\begin{lemma}\label{lem-total-mass}
	There exists $C\in (0,\infty)$ such that \(\lim_{\delta\to 0}	\frac{|M_{\delta}| }{\log \delta^{-1}}= C\).
	Moreover, the $M_\delta^\#$-law of $(A,P)$  converges in total variational distance  to the probability measure on $[1,2]\times (0,\infty)$ whose density  is proportional to
	\[
	\frac{a^{4/\gamma^2-1}}{(a^{4/\gamma^2} + p^{4/\gamma^2})^2}\,  dadp
	\]
	and  $\lim_{\delta\to 0}M_\delta^\#[B>\eps]=\lim_{\delta\to 0}M_\delta^\#[Q>\eps]=0$ for each $\eps>0$.
\end{lemma}
\begin{proof}
	Our proof relies on~\eqref{eq:welding-special}, which gives a description of the joint distribution of $(A,P,Q,B)$ under $M_\delta$ in terms of 
	$\cM_{0,2}^\disk(\frac{\gamma^2}2)$ and  $\cM_{1,2}^\disk(2)$, whose  quantum length distributions  are given in  Lemma~\ref{lem:lengh}.
	Indeed, write $g = \frac{4}{\gamma^2} > 1$. By Lemma~\ref{lem:lengh} and the definition of $\cM^\disk_{1,2} (2; \ell,r)$,	there exist constants $C_1,C_2$ such that
	\eqb\label{eq-MOT-lengths}
	|\cMtwo(\frac{\gamma^2}2;\ell, r)| = C_1 \frac{(\ell r)^{g-1}}{(\ell^{g} + r^{g})^2} \quad \textrm{and} \quad  |\cM^\disk_{1,2} (2; \ell,r)| = C_2 (\ell + r)^{-g+1}.
	\eqe
	By~\eqref{eq:welding-special}, the  $M_\delta^{\#}$-law of  $(A,P,Q,B)$ is the  probability measure supported on the set 
	$S_\delta=\{ (a,p,q,b)\in  (0,\infty)^4: b\in (\delta,1/2) , a+b\in (1,2)  \}$ whose density function is proportional to 
	\begin{equation}\label{eq:density-apqb}
		m(a,p,q,b) =\frac{(ap)^{g-1}}{(a^g+p^g)^2} \cdot (p+q)^{-g+1} \cdot \frac{(bq)^{g-1}}{(b^g+q^g)^2}
	\end{equation}	
	and 
	\begin{equation}\label{eq:Mdelta}
		|M_\delta|= C_1^2 C_2\int_{S_\delta} m(a,p,q,b) \, da\,dp\,dq\,db.
	\end{equation}	
	Using the change of variable $p = ar$ and $q = bs$, we see that 
	\begin{equation}\label{eq:ratio}
		\iint_0^\infty m(a,p,q,b)
		\,dp\,dq = \frac{1}{a^gb} \iint_0^\infty \frac{r^{g-1}}{(1+r^g)^2} \cdot \frac1{( r + \frac ba s)^{g-1}} \cdot \frac{s^{g-1}}{(1+s^g)^2} \,dr \, ds.
	\end{equation}
	By the monotone convergence theorem, the last integral converges to 
	\[
	C_3=\iint_0^\infty \frac{s^{g-1}}{(1+r^g)^2(1+s^g)^2}   \,drds < \infty \textrm{ as }  b/a\to  0.
	\]
	Therefore:
	\eqb \label{eq-rs}
	\iint_0^\infty m(a,p,q,b)
	\,dp\,dq = \frac{C_3}{a^gb}(1+o_{b/a}(1)).
	\eqe
	
	Integrating~\eqref{eq-rs} over $b\in (\delta,1/2) , a+b\in (1,2)$  and using~\eqref{eq:Mdelta},{\tiny } we get 
	\(\lim_{\delta\to 0}	\frac{|M_{\delta}| }{\log \delta^{-1}}=  C\) with $C := \int_1^2 \frac{C_1^2C_2 C_3}{a^{g}}\, da$. 
	Since $\sup_{\delta\in (0,1/2)}M_\delta[B>\eps]<\infty$ for each $\eps>0$, $\lim_{\delta\to 0}M_\delta^\#[B>\eps]=0$.
	By~\eqref{eq:ratio}, the $M_\delta^\#$-law of $\frac{Q}{B}$ is converging  to the probability measure on $(0,\infty)$ that is proportional to $\frac{s^{g-1}\, ds}{(1+s^g)^2}$. 
	In particular,  $\lim_{\delta\to 0}M_\delta^\#[Q>\eps]=0$. Similarly, the $M_\delta^\#$-law of $(A,\frac{P}{A})$ converges to the probability measure on $[1,2]\times (0,\infty)$ that is proportional to  $\frac{da}{a^g} \times \frac{dr}{(1+r^g)^2}$. This gives the desired limiting joint distribution of $(A,P)$.
\end{proof}

We first gather some basic facts on the quantum boundary length under the GFF measure $P_\bbH$.
	\begin{lemma}\label{lem-GFF-facts}
		Sample $h$ from $P_\bbH$. For each $p>0$ we have $P_\bbH[\nu_h(0,1) < \frac1s] = O(s^{-p})$ as $s \to \infty$. Moreover, there exists $C\in (0,\infty)$
		such that \(P_\bbH[\nu_h(0,y)>t] \le C t^{-\frac4{\gamma^2}} y \)  for all $t > 0 $  and $y \in(0,1)$.
		Finally,  \ $\lim_{y \to 0^+} \frac{\left|\log y\right|}{\left| \log \nu_h (0,y)\right|} = \frac2{\gamma Q}$ a.s.
	\end{lemma}
	\begin{proof}
		We first show that if $G$ is a standard Gaussian independent of $h$ and $y \in (0,1)$, then
		\eqb \label{eq-similarity}
		\nu_h(0,y) \stackrel d= e^{\frac\gamma2(-Q \left| \log y \right|+ \sqrt{2\left|\log y\right|}G)} \nu_h(0,1). 
		\eqe
		Let $c$ be the average of $h$ on $\partial B_y(0) \cap \bbH$, so the marginal law of $c$ is $N(0,2\left| \log y \right|)$. The conditional law of $\wt h = h - c$ given $c$ is the free-boundary GFF on $\bbH$ normalized to have average zero on $\partial B_y(0) \cap \bbH$. Let $\psi(z) =yz$ for $z\in \bbH$. Then $\wt h \circ \psi$ has law $P_\bbH$. Therefore \[\nu_h(0,y) = e^{\frac\gamma2c} \nu_{\tilde h}(0,y) = e^{\frac\gamma2c}\nu_{\tilde h \circ \psi + Q \log |\psi'|}(0,1) \stackrel d=  e^{\frac\gamma2c + \frac\gamma2 Q  \log y} \nu_{h} (0,1).\]
		This proves~\eqref{eq-similarity}. 
		
		Now we address the tail bounds in the lemma. \cite[Theorem 2.12]{rhodes-vargas-review} gives the bound $P_\bbH[\nu_h(0,1) < \frac1s] = O(s^{-p})$ as $s \to \infty$. \cite[Theorem 1.1]{wong-tail}  gives
		$P_\bbH[\nu_h(0,1) > s] \leq C s^{-\frac4{\gamma^2}}$ for all $s>0$, so using~\eqref{eq-similarity} and $s = t e^{\frac\gamma2(Q \left|\log y\right|- \sqrt{2 \left|\log y\right|} G)}$ gives
		\[
		P_\bbH[\nu_h(0,y)>t] \leq C \E[(t e^{\frac\gamma2(Q \left|\log y\right|- \sqrt{2 \left|\log y\right|} G)})^{-\frac4{\gamma^2}}] 
		= Ct^{-\frac4{\gamma^2}} y.
		\]

		For the last assertion, let $y_n = e^{-(\log n)^4}$ and $\eps_n = \frac1{\log n}$ so that $y_n $ and $\eps_n \to 0$ as $n \to \infty$. Let $E_n=\{\nu_h(0,y_n) \in (y_n^{\frac{\gamma Q}2 + 2\eps_n}, y_n^{\frac{\gamma Q}2 - 2\eps_n}) \}$. By the tail bounds,  
		we have 
		\[P_\bbH[E_n^c] \leq \P[e^{\frac\gamma2\sqrt{2\left|\log y_n\right|} G} \not \in (y_n^{\eps_n}, y_n^{-\eps_n})] + P_\bbH[\nu_h(0,1) \not \in (y_n^{\eps_n}, y_n^{-\eps_n})] \leq 2e^{-O(\eps_n^2 \left|\log y_n\right|)} + 
		 O( y_n^{\frac4{\gamma^2}\eps_n}).\]
		Here, the first term is bounded from above by the standard Gaussian tail bound.  
		We conclude that $\sum_n P_\bbH[E_n^c] < \infty$, so the Borel-Cantelli lemma yields that a.s.\ $E_n$ occurs for all sufficiently large $n$. Since $\lim_{n\to\infty} \frac{\left|\log y_{n+1}\right|}{\left|\log y_n \right|} = 1$, this gives that $\lim_{y \to \infty} \frac{\left|\log y\right|}{\left| \log \nu_h (0,y)\right|} = \frac2{\gamma Q}$ a.s.
	\end{proof}

We now give an asymptotic estimate for $\log x_\delta$.
\begin{lemma}\label{lem-as-tail}
	$M$-almost everywhere,  \( \lim_{\delta\to0} \frac{\log x_\delta}{\log \delta^{-1}} = \frac2{\gamma Q}\).
\end{lemma}
\begin{proof}
	Recall from Definition~\ref{def:LFalpha} that $\phi$ sampled from $\LF_\bbH^{(\gamma, i)}$ is given by $\phi = h -2Q \log |z|_+ + \gamma G_\bbH(\cdot,i) + \mathbf c$ where $(h, \mathbf c)$ is sampled from $2^{-\gamma^2/2} P_\bbH\times [e^{(\gamma - Q)c}dc]$.
	Thus $\frac{\nu_\phi(0,y)}{e^{\frac\gamma2 \mathbf c}\nu_h(0,y)} \in (e^{-\frac{\gamma^2}2 G_\bbH(i,1)},1)$ for all $y \in (0,1)$, and so $\LF_\bbH^{(\gamma, i)}$-almost everywhere $\lim_{y \to 0^+} \frac{\log \nu_\phi(0,y)}{\log \nu_h(0,y)} = 1$.
	Therefore by the last assertion of Lemma~\ref{lem-GFF-facts},   we have
	\eqb\label{eq-inverted-as}
\lim_{y \to 0^+} \frac{\left| \log y \right|}{\left| \log \nu_\phi(0,y)\right|} = \frac2{\gamma Q} \quad \LF_\bbH^{(\gamma, i)}\text{-almost everywhere}.
\eqe
 By Lemma~\ref{lem-change-coord} applied to the inversion map $z \mapsto -z^{-1}$, this yields  that $\lim_{y \to \infty} \frac{\log y}{-\log \nu_{\phi}(y, \infty)} = \frac2{\gamma Q}$. Setting $y = x_\delta$ so that $\nu_\phi(y, \infty) = \delta$, we get
$\LF_\bbH^{(\gamma, i)}$-almost everywhere   \( \lim_{\delta\to0} \frac{\log x_\delta}{\log \delta^{-1}} = \frac2{\gamma Q}\). 
Hence the same limit holds for $M$ by restricting to $\nu_\phi(\R)\in (1,2)$.	
\end{proof}

The following tail bound is  quite loose, but suffices for our purposes. 
\begin{lemma}\label{prop-length-tail-bound}
	For $\zeta \in (0,1)$,  there exists $C_\zeta \in (0,\infty)$ such that
	\begin{equation}\label{eq:x-tail}
		M[ \{x_t\ge y  \}] < C_\zeta (t^{4/\gamma^2} y)^{-(1-\zeta)} \quad \text{for all } t > 0 \text{ and } y >1.
	\end{equation}
\end{lemma}
\begin{proof}
	By the definition of $x_t$, \eqref{eq:x-tail} is equivalent to $M[\{ \nu_\phi(y, \infty) > t\}] < C_\zeta (t^{4/\gamma^2} y)^{-(1-\zeta)}$ for all $t>0$ and $y>1$. Using 	Lemma~\ref{lem-change-coord} with the inversion map $z \mapsto -z^{-1}$, this is equivalent to 
	\eqb\label{eq-tail-zero}
	M[ \{ \nu_\phi(0,y) > t\}] < C_\zeta (t^{-4/\gamma^2} y)^{(1-\zeta)} \quad \text{for all } t > 0 \text{ and } y \in(0,1).
	\eqe
	
	Recall that $M$ is $\LF_\bbH^{(\gamma, i)}$ restricted to $\{\nu_\phi(\R) \in (1,2)\}$. By Definition~\ref{def:LFalpha}, $\phi$ sampled from $\LF_\bbH^{(\gamma, i)}$ is $\phi(z) = h(z) - 2Q \log |z|_+ + \gamma G_\bbH(z,i) + \mathbf c$ where $(h, \mathbf c)$ is sampled from $2^{-\gamma^2/2}P_\bbH\times [e^{(\gamma - Q)c}dc]$. Since  $- 2Q \log |z|_+ + \gamma G_\bbH(z,i) \in (-\log2, 0)$ for $z \in [-1,1]$, we have $2^{-\gamma^2/2} e^{\frac\gamma2\mathbf c}\nu_h(I) < \nu_\phi(I) < e^{\frac\gamma2\mathbf c} \nu_h(I)$ for any $I \subset [-1,1]$. Since $M$-a.s.\ we have $\nu_\phi(-1,0) < 2$, for any $r \in \R$ we have
	\alb
	2^{\gamma^2/2}M[ \{ \nu_\phi(0,y) > t\}] &\leq \int_\R P_\bbH \left[e^{\frac\gamma2 c} \nu_h(0,y)>t \text{ and } 2^{-\gamma^2/2}e^{\frac\gamma2c}\nu_h(-1,0) < 2 \right] e^{(\gamma-Q)c} \, dc \\
	&\leq \int_{-\infty}^r P_\bbH[e^{\frac\gamma2c}\nu_h(0,y)>t] e^{(\gamma-Q)c}\,dc + \int_r^\infty P_\bbH[e^{\frac\gamma2c}\nu_h(-1,0)<2^{\gamma^2/2+1}] e^{(\gamma-Q)c}\,dc.
	\ale
	The first term can be bounded using the second assertion in Lemma~\ref{lem-GFF-facts}
	\[ \int_{-\infty}^r P_\bbH[e^{\frac\gamma2c}\nu_h(0,y)>t] e^{(\gamma-Q)c}\,dc \le  C \int_{-\infty}^r t^{-\frac4{\gamma^2}} y e^{\frac2\gamma c} \cdot e^{(\gamma - Q)c} \, dc = \frac2\gamma t^{-\frac4{\gamma^2}} y e^{\frac\gamma2 r}, \] where $C\in (0,\infty)$ is a constant that can change from line to line.
	For the second term, by Lemma~\ref{lem-GFF-facts} for any $p>0$ we have: 
	\[\int_r^\infty P_\bbH[e^{\frac\gamma2c}\nu_h(-1,0)<2^{\gamma^2/2+1}] e^{(\gamma-Q)c}\,dc \le C \int_r^\infty  2^{(\gamma^2/2 + 1)p} e^{-\frac\gamma2pc} e^{(\gamma - Q)c} \, dc \leq C  e^{(\gamma - Q - \frac\gamma2p)r}. \]
	Combining the two  bounds gives:
	\[M[ \{ \nu_\phi(0,y) > t\}] \le C t^{-\frac4{\gamma^2}} y e^{\frac\gamma2 r} + C e^{(\gamma - Q - \frac\gamma2p)r}. \]
	Choosing $r = \frac{\frac{4}{\gamma^2} \log t + \left|\log y\right| }{\frac2\gamma + \frac\gamma2p}$ gives~\eqref{eq-tail-zero} with $\zeta = \frac{\frac\gamma2}{\frac2\gamma + \frac\gamma2p}$. Varying $p\in (0,\infty)$, we get all $\zeta\in (0,1)$.
\end{proof}

The next lemma allows us to  control  $f(x_{1/2})$.
\begin{lemma}\label{lem-O1}
	We have $M[|f(x_{1/2})|] < \infty$. 
\end{lemma}
\begin{proof}
	We first show there is a constant $C>0$ such that 
	\eqb\label{eq-f-weak-log}
	f(y) < C \log y \quad \text{ for all } y > 2. 
	\eqe
	Write $p = M[\{x_{1/2} < 0 \}]$. By Lemma~\ref{lem-as-tail}, for $\delta$ small enough, $M[\{x_\delta <\delta^{-\frac1{\gamma Q}} \}]  < \frac p2$.
	Let $F_\delta = \{x_\frac12 < 0 \text{ and } x_\delta > \delta^{-\frac1{\gamma Q}} \}$ so that $M[F_\delta]\ge \frac{p}2$. By Lemma~\ref{cor-RN} we have \( |M_{\delta}| = M[f(x_\delta) - f(x_{1/2})] \geq M[1_{F_\delta}(f(x_\delta) - f(x_{1/2}))] \).
	Since $f$ is increasing and $f(0) =0$,	we have 
	\[ M[1_{F_\delta}(f(x_\delta) - f(x_{1/2}))] \ge  f(\delta^{-\frac1{\gamma Q}})M[F_\delta] \ge  \frac p2 f(\delta^{-\frac1{\gamma Q}}).  \]

	By Lemma~\ref{lem-total-mass}, $|M_{\delta}|   =(1+o_\delta(1)) \mathfrak C \log \delta^{-1}$. Therefore there exists $C\in (0,\infty)$ such that 
	$f(\delta^{-\frac1{\gamma Q}}) \le C  \log \delta^{-1}$ for all $\delta\in (0,1/2)$.  Setting $y = \delta^{-\frac1{\gamma Q}}$ and possibly choosing $C$ larger
	we obtain~\eqref{eq-f-weak-log}. 
	
	Now we bound $M[|f(x_{1/2})|]$. 
	By~\eqref{def:f} $f$ is an increasing odd function, so 
	\begin{equation}\label{eq:monotone}
		M[1_{x_{1/2}>1} f(x_{1/2})]= \sum_{n=1}^\infty M[1_{2^{n}\geq x_{1/2} >2^{n-1} } f(x_{1/2})] \leq  \sum_{n=1}^\infty f(2^n) M[\{x_{1/2} >2^{n-1} \} ].
	\end{equation}
	By~\eqref{eq-f-weak-log} and the tail bound of $x_{1/2}$ in Lemma~\ref{prop-length-tail-bound}, for a possibly  larger $C$ we have
	\[
	M[1_{x_{1/2}>1} f(x_{1/2})]\le \sum_{n=1}^\infty C \log 2^n M[\{x_{1/2} >2^{n-1} \} ]\le C^2\sum_{n=1}^\infty n 2^{-n/2}
	<\infty.
	\]
	Since $0<f(t)< f(2)$ for $t \in (0,2)$, we have $M[1_{1\ge x_{1/2}\ge 0} f(x_{1/2})] \le f(2)|M|< \infty$
	hence $M[1_{x_{1/2}\ge 0} |f(x_{1/2})|]= M[1_{x_{1/2}\ge 0} f(x_{1/2})] < \infty$. As $M$-almost everywhere $\nu_\phi(\R)\ge 1$, we conclude
	\[
	M[x_{1/2}<-y]=M[ \nu_{\phi} (-y,\infty)\le \frac12 ] \le  M[ \nu_{\phi} (-\infty,-y)\ge \frac12 ]=  	M[x_{1/2}>y] \textrm{ for } y>0.
	\]  
	Since $f$ is an odd function, we have that $M[1_{x_{1/2}<0} |f(x_{1/2})|] \le M[1_{x_{1/2}\ge 0} |f(x_{1/2})|] < \infty$.
\end{proof}

Now, we give the asymptotic upper bound on $f$. 

\begin{lemma}\label{lem-f-upper}
	For any $s>0$ we have $\limsup_{\delta \to 0} |M| f(\delta^{-\frac2{\gamma Q} + s})/|M_\delta| \leq 1$.
\end{lemma}
\begin{proof}
	By Lemma~\ref{cor-RN}, we have 
	\begin{align*}
		|M_\delta| &= M[f(x_\delta) - f(x_{1/2})] \geq M [1_{ \{ x_\delta > \delta^{-\frac2{\gamma Q}+s} \} } (f(x_\delta)-f(x_{1/2}))]\\
		&\geq f(\delta^{-\frac2{\gamma Q}+s})(|M|- M [\{x_\delta \le  \delta^{-\frac2{\gamma Q}+s}\}]) - M[|f(x_{1/2})|].
	\end{align*}
	Since $\lim_{\delta\to 0}M [\{x_\delta \le  \delta^{-\frac2{\gamma Q}+s}\}] =0$ by Lemma~\ref{lem-as-tail}, $|M_\delta| \to \infty$ by Lemma~\ref{lem-total-mass} and $M[|f(x_{1/2})|]<\infty$ by Lemma~\ref{lem-O1} we are done.
\end{proof}
We next obtain the matching asymptotic lower bound on $f$.
\begin{lemma}\label{lem-f-lower}
	For any $s>0$ we have $\liminf_{\delta \to 0} |M| f(\delta^{-\frac2{\gamma Q} - s})/|M_\delta| \geq 1$.
\end{lemma}
\begin{proof}
	Let $E_\delta  = \{x_\delta > \delta^{-\frac2{\gamma Q} - s}\}$. By the monotonicity of $f$ we may assume $\frac2{\gamma Q} +s < \frac8{\gamma^2}$, and by~\eqref{eq-f-weak-log} and monotonicity we have
	\begin{align*}
		|M_\delta| &= M[f(x_\delta) - f(x_{1/2})] \leq  M [1_{ E_\delta } f(x_\delta)]+M [1_{ E_\delta^c } f(x_\delta)]+
		M	 [|f(x_{1/2})|]\\
		&\leq C M[ 1_{ E_\delta }\log (x_\delta)]+   f(\delta^{-\frac2{\gamma Q}-s}) (|M|-M [E_\delta] )+M [|f(x_{1/2})|].
	\end{align*}
	We claim that $M[1_{E_\delta} \log x_\delta] = o(\log \delta^{-1})$. To see this, we write
		\[ M[1_{E_\delta} \log x_\delta] = \int_0^\infty M[E_\delta \cap \{\log x_\delta > w\}]\, dw  \leq \frac8{\gamma^2}\log \delta^{-1} M[E_\delta] + \int_{\frac8{\gamma^2}\log \delta^{-1}}^\infty M[\{\log x_\delta > w\}]\, dw. \]
	By Lemma~\ref{lem-as-tail} the first term is $o(\log \delta^{-1})$, and for the second term, let $\zeta\in (0,1)$ and use 
	Lemma~\ref{prop-length-tail-bound}:
	\[ \int_{\frac8{\gamma^2}\log \delta^{-1}}^\infty M[\{\log x_\delta > w\}]\, dw\le \int_{\frac8{\gamma^2}\log \delta^{-1}}^\infty C_\zeta (\delta^{\frac4{\gamma^2}} e^{w})^{-1+\zeta}\, dw =(1-\zeta)^{-1}C_\zeta\delta^{\frac{(1-\zeta)4}{\gamma^2}}.
	\]
	Thus $M[1_{E_\delta} \log x_\delta] = o(\log \delta^{-1})$. Moreover, $\lim_{\delta \to 0} M[E_\delta] = 0$ by Lemma~\ref{lem-as-tail}, $|M_\delta| \gtrsim \log \delta^{-1}$ by Lemma~\ref{lem-total-mass} and $M[|f(x_{1/2})|]<\infty$ by Lemma~\ref{lem-O1}. This gives the desired result. 
\end{proof}

Now we are ready to conclude the proof of Proposition~\ref{prop-apprx-zipper}. We start by the marginal law of $\phi$.
\begin{lemma}\label{lem:phi-tv}
	The $M^{\#}_\delta$-law of $\phi$  converges to  $M^{\#}$  in total variational distance as $\delta\to 0$.
\end{lemma}
\begin{proof}
	Let $R_\delta(\phi)= \frac{dM_{\delta}^\#}{dM^\#} (\phi)$ be the Radon derivative of the marginal of $\phi$ under  $M^{\#}_\delta$ with respect to $M$. 
	By Lemma~\ref{cor-RN}, we have
	\[
	R_\delta(\phi)=   \frac{|M|}{ |M_\delta|} (f(x_\delta) - f(x_{1/2})) \textrm{ almost surely in }M^\#.
	\]
	By Lemmas~\ref{lem-as-tail},~\ref{lem-f-upper} and~\ref{lem-f-lower},  we see that  $\lim_{\delta\to 0}\frac{|M|}{ |M_\delta|} f(x_\delta) =1$ a.s.\ in $M^{\#}$. Since $ \lim_{\delta\to 0}|M_\delta|=\infty$ and $ -\infty<f(x_{1/2})<\infty$ a.s.\ in $M^\#$, we see that $\lim_{\delta\to 0} R_\delta=1$ a.s.\ in $M^\#$.  Since $M^\#[R_\delta(\phi)]=1 $ for each $\delta$, we have $\lim_{\delta\to 0} M^\#[|R_\delta(\phi)-1|]=0$. Therefore, with the supremum taken over measurable sets $A$, $\sup_A |M_\delta^\#[A] - M^\#[A]| = \sup_A \left| M^\#[1_A (R_\delta(\phi) - 1)] \right| \leq M^\#[|R_\delta(\phi)-1|] \to 0$, which concludes the proof.
\end{proof}
We now deal with the joint law of $(\phi,\mathbf x)$.
\begin{lemma}\label{lem:x-tv}
	Let $\hat m_\delta=1_{\{ 0<x<\delta^{-\frac2{\gamma Q}} \} } p(x) \, dx$.
	The $M_\delta^\#$-law of $(\phi,\mathbf x)$ is $o_\delta(1)$-close in total variational distance to $M^\#\times \hat m^\#_\delta$.
\end{lemma}
\begin{proof}
	The conditional law of $\mathbf x$ given $\phi$ is $ 1_{x_{1/2}<x<{x_\delta}} (f(x_{\delta})-f(x_{1/2}))^{-1} p(x) dx$. Let $d_\delta(\phi)$ be the total variational distance between this distribution and $\hat m^\#_\delta$. Let $E_\delta = \{ |\frac{x_\delta}{\log \delta^{-1}} - \frac2{\gamma Q} | < 1\}$. By Lemmas~\ref{lem-as-tail} and~\ref{lem:phi-tv}, $\lim_{\delta \to 0}M_\delta[E_\delta^c]=0$. Moreover, by Lemmas~\ref{lem-total-mass},~\ref{lem-f-upper} and~\ref{lem-f-lower}, $\lim_{y\to\infty}(\log y)^{-1}f(y)$ is a positive constant, and hence 
	$\lim_{\delta \to 0} M_\delta[1_{E_\delta} d_\delta(\phi)] = 0$. We conclude that $\lim_{\delta \to 0} M_\delta[d_\delta(\phi)] = 0$ as needed. 
\end{proof}
\begin{proof}[Proof of Proposition~\ref{prop-apprx-zipper}]
	Since the conditional law of $(\eta_1,\eta_2)$ under $M_\delta^\#$ given $(\phi,\mathbf x)$ only depends on $\mathbf x$, Lemma~\ref{lem:x-tv} implies  Proposition~\ref{prop-apprx-zipper}.
\end{proof}

	\begin{remark}
		The above argument implies that there exists a constant  $\mathfrak C\in (0,\infty)$ such that 
		\begin{equation}\label{eq-total-mass}
			\lim_{\delta\to 0}	\frac{|M_{\delta}| }{\log \delta^{-1}}= \mathfrak C \quad \textrm{and}\quad
			\lim_{y \to \infty} \frac{f(y)}{ \log y} = \frac{\gamma Q\mathfrak C}{2|M|}.
		\end{equation}
		Moreover, the constant $\mathfrak C$ can be made explicit by keeping track of the various constants in Lemma~\ref{lem-total-mass} (see \cite[Proposition 5.2]{AHS-SLE-integrability} for these constants). We find it interesting that though the function $f$ is defined in terms of SLE, its asymptotics are derived using properties of the Liouville field $\phi$. 
	\end{remark}

\subsection{Passing to the limit: proof of Theorem~\ref{thm-bubble-zipper} }\label{subsec:delta0}

The arguments in this section are technical but standard and can be skipped on a first reading. 
Sample a pair of quantum surfaces $(\cD_1,\cD_2)$ from $\int_1^2 \int_0^\infty \cM_{0,2}^\disk(\frac{\gamma^2}2; a, p) \times  \QD_{1,1}(p) \, dp \, d a$. 
Let $\cD_1\oplus \cD_2$ be the curve-decorated quantum surface obtained by conformally welding the right boundary of $\cD_1$ and the total boundary of $\cD_2$. 
The surface $\cD_1\oplus \cD_2$ naturally carries a marked interior point and a marked boundary point. See Figure~\ref{fig:disk-pinch}. 
Let $(\D, \phi_\D, \eta_\D, 0, i)$ be the unique embedding of  $\cD_1\oplus \cD_2$ in $(\D,0,i)$.
We denote the  law  of $(\phi_\D, \eta_\D)$ by  $M_\D$.  	
Let $f: \bbH\to \D$ be the conformal map with $f(i) = 0$ and $f(\infty) = i$.
We will prove Theorem~\ref{thm-bubble-zipper} by showing the following.
\begin{proposition}\label{prop:MD}
Under the probability measure $M_\D^\#$ proportional to $M_\D$, 
$\phi_\D$ and $\eta_\D$  are independent, and moreover, the law of  $\phi_\D\circ f+ Q\log |f'|$   is  proportional to $\int_1^2 \LF_\bbH^{(\gamma, i)} (\ell) d\ell$. 
\end{proposition}
Recall $M_\delta$ from Lemma~\ref{lem:three-disks} and Proposition~\ref{prop-apprx-zipper}. 
Let  $(\phi, \mathbf x, \eta_1, \eta_2)$ be sampled from  $M_\delta^\#$. 
Let $\cD_{1,\delta}$, $\cD_{2,\delta} $, $\cD_{3,\delta}$ be the left, right, and middle marked quantum surface in the sample space of $M_\delta$, 
so that $(\bbH, \phi, \eta_1, \eta_2, i, \mathbf x )$  is the embedding of the conformally welded surface $\cD_{1,\delta}\oplus \cD_{2,\delta}  \oplus \cD_{3,\delta}$.
Let $\phi^\delta= \phi \circ f^{-1} +Q\log |(f^{-1})'|$ and $\eta^\delta= f \circ \eta_1$, so that $(\D, \phi^\delta, 0, i)$ is the embedding of $\cD_{1,\delta}\oplus \cD_{2,\delta}  \oplus \cD_{3,\delta}$ on $(\D, 0, i)$ by forgetting the interfaces and $\eta^\delta$ is the interface between $\cD_{1,\delta}$ and $\cD_{2,\delta}$.  See Figure~\ref{fig:disk-pinch}.
The following lemma and Proposition~\ref{prop-apprx-zipper} immediately give Proposition~\ref{prop:MD}.

\begin{lemma}\label{prop-couple-bulk-welded}
There exists a coupling between 	$M_\delta^\#$ and $M_\D^\#$ such that the following holds. 
We can find random simply-connected domains $U_\delta, \wt U_\delta \subset \D$ and a conformal map $g_\delta: \wt U_\delta \to U_\delta$ such that  
with probability $1-o_\delta(1)$,
\begin{itemize}
\item $\phi_\D (z)= \phi^\delta \circ g_\delta(z) + Q \log |g'_\delta(z)|  \textrm{ for }z\in \wt U_\delta$,
\item $\mathrm{diam}(\D\backslash U_\delta) = o_\delta(1)$, $\mathrm{diam}(\D \backslash \wt U_\delta) =o_\delta(1)$, 
\item \(	\sup_{z \in K} |g_\delta(z) - z| = o_\delta(1) 	\textrm{ for any compact } K \subset \D.\)
\end{itemize}
\end{lemma}
 
To prove Lemma~\ref{prop-couple-bulk-welded}, we use the following basic coupling result on quantum disks.
Let $\eps >0$. For a simply-connected quantum surface $\cD$ decorated with one bulk and one boundary point, let $\cD^\eps$ be the quantum surface obtained by embedding $\cD$ as $(\bbH, \phi, i, -1)$ and setting $\cD^\eps := (\bbH_\eps, \phi, i, -1, -1-2\eps)$ where $\bbH_\eps= \bbH \backslash B_\eps(-1-\eps)$ with $B_\eps(-1-\eps)=\{ z\in \C: |z+1+\eps|\le \eps \}$.
\begin{lemma}\label{lem-couple-disk-bulk}
	For $\eps>1$ and $\ell > 0$, when $\cD$ and $\wt \cD$ are sampled from $\QD_{1,1}(\ell)^\#$ and $\QD_{1,1}(\wt \ell)^\#$ respectively, the law of $\wt \cD^\eps$ converges to that of $\cD^\eps$ in total variation distance as $\wt \ell \to \ell$. 
\end{lemma}
\begin{proof}
This can be proved   directly via the explicit   Liouville field description $\QD_{1,1}(\ell)$ in Proposition~\ref{prop-embed-QD11}. However, the corresponding statement for  $\QD_{0,2}(\ell)^\#$ and $\QD_{0,2}(\wt \ell)^\#$ is already proved in~\cite[Proposition~2.23]{ahs-disk-welding}. 
(In fact, Proposition~2.23 there proved the general result for $\cM_{0,2}^\disk(W)$.)  
By Definition~\ref{def-QD}, 
we can transfer the result for  $\QD_{0,2}(\ell)^\#$ into the desired result for $\QD_{1,1}(\ell)^\#$.
\end{proof}

\begin{figure}[ht!]
	\begin{center}
		\includegraphics[scale=0.65]{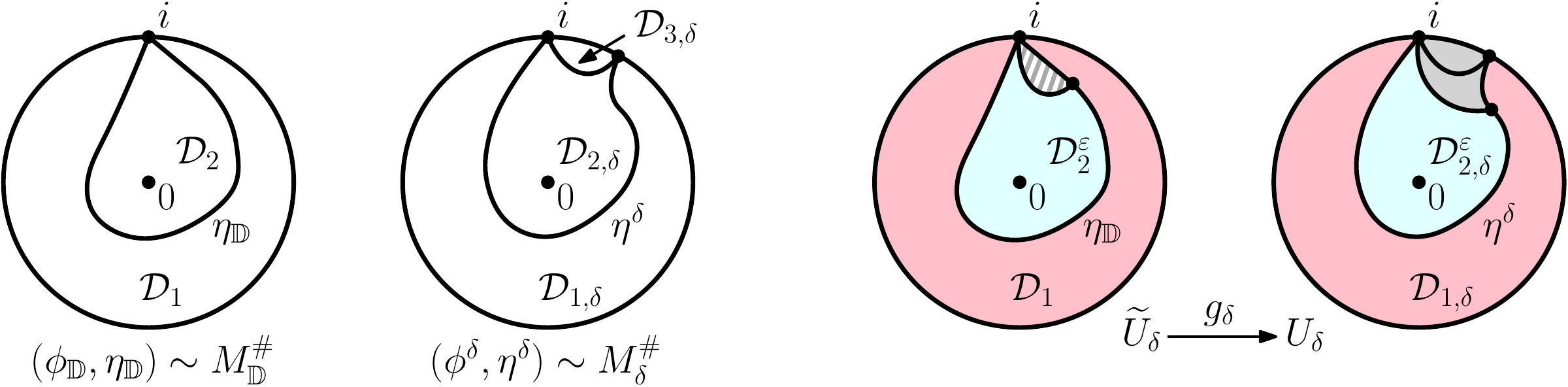}%
	\end{center}
	\caption{\label{fig:disk-pinch} \textbf{Left:} $(\phi_\D, \eta_\D)$ is obtained by embedding $\cD_1 \oplus \cD_2$ into $(\D, 0, i)$, and $(\phi^\delta, \eta^\delta)$ is obtained by embedding $\cD_{1,\delta} \oplus \cD_{2,\delta} \oplus \cD_{3,\delta}$ into $(\D, 0, i)$. \textbf{Right:} We can couple $M_\D^\#$ and $M_\delta^\#$ so the pink and blue quantum surfaces agree with high probability. The domains $\wt U_\delta$ and $U_\delta$ are the unions of the pink and blue regions.
	}
\end{figure}

\begin{proof}[{Proof of Lemma~\ref{prop-couple-bulk-welded}}]
	Recall the marked quantum surfaces $\cD_1,\cD_2$ in the definition of $(\phi_\D,\eta_\D)$ and $M_\D^\#$.
		Let $\wt A$ and $\wt P$ be the left  and right boundary lengths of $\cD_1$, respectively.
		The law of $(\wt A,\wt P)$ is the probability measure on $[1,2]\times (0,\infty)$ proportional to  
		\begin{equation}\label{eq:AP}
	 |\cMtwo(\frac{\gamma^2}2;  a,  p)||\QD_{1,1}( p)| \propto a^{\frac{4}{\gamma^2}-1}/( a^\frac{4}{\gamma^2}+ p^\frac{4}{\gamma^2})^2,
		\end{equation}
			where we have used~\eqref{eq-MOT-lengths} for the expression of $|\cMtwo(\frac{\gamma^2}2;  a,  p)|$. 
		Conditioning on $(\wt A,\wt P)$, the conditional law of $(\cD_1, \cD_2)$ is $\cMtwo(\frac{\gamma^2}2; \wt A, \wt P)^\# \times \QD_{1,1}(\wt P)^\#$.
	
Similarly, recall that the law of $\cD_{1,\delta}$, $\cD_{2,\delta} $, $\cD_{3,\delta}$  is given in \eqref{eq:welding-special}.   
Let $A_\delta$ and $P_\delta$ be the left  and right boundary lengths of $\cD_{1,\delta}$, respectively.
Let $Q_\delta $ and $B_\delta$ be the left  and right boundary lengths of $\cD_{3,\delta}$, respectively. Then, by Lemma~\ref{lem-total-mass}, the $M_\delta^\#$-law of $(A_\delta,P_\delta)$  weakly converge to the the $M_\D^\#$-law of $(\wt A,\wt P)$, and $\lim_{\delta\to 0}M_\delta^\#[Q_\delta>\eps]=\lim_{\delta\to 0}M_\delta^\#[B_\delta>\eps]=0$ for all $\eps>0$.

By Lemma~\ref{lem-total-mass} we can couple  $M_\delta^\#$ and $\wt M^\#$ such that $(\wt A,  \wt P) = (A_\delta, P_\delta)$ with probability $1-o_\delta(1)$. 
By Proposition~\ref{prop-QD2-field}, a sample from $\cM_{1, 2}^\disk(2; p,q)^\#$  can be obtained by sampling a quantum surface from $\QD_{1,1}(p+q)^\#$, then marking a second boundary point at quantum length $p$ counterclockwise from the marked boundary point. Since $(\wt A,\wt P) = ( A_\delta, P_\delta)$ with probability $1-o_\delta(1)$ and $Q_\delta \to 0$ in probability as $\delta \to 0$, by   Lemma~\ref{lem-couple-disk-bulk}, we can extend our coupling such that 
$\lim_{\delta\to 0}\P[ (\cD_1, \cD_2^\eps) = (\cD_{1,\delta},  \cD_{2,\delta}^\eps)]=1$ for any fixed $\eps > 0$; consequently the same holds when $\eps$ is e.g.\ a piecewise constant function of $\delta$ which decays to 0 sufficiently slowly as $\delta \to 0$.

Let $U_\delta$ (resp. $\wt U_\delta$) be the  domain parametrizing the conformal welding of $\cD_1$ and $\cD_2^\eps$ (resp. $\wt\cD_1$ and $\wt\cD_2^\eps$) in the aforementioned embedding of $\cD_1\oplus \cD_2$ (resp.  $\cD_{1,\delta}\oplus \cD_{2,\delta}\oplus \cD_{3,\delta}$) in $(\D, 0, i)$.  
More precisely, $U_\delta$ is the interior of the union of the closures of  the domains corresponding to $\cD_1$ and $\cD_2^\eps$ in our embedding, and the analogous statement holds for $\wt U_\delta$.
By definition, the marked quantum surfaces $(U_\delta, \phi^\delta , 0, i)$ and $(\wt U_\delta, \phi_\D, 0, i^-)$ agree with probability  $1-o_\delta(1)$, where $i^-$ be the boundary point immediately to the left of $i$.  
On this event, by the agreement of quantum surfaces,  there exists a unique conformal map $g_\delta: \wt U_\delta \rta U_\delta$ such that $\phi_\D= \phi^\delta \circ g_\delta+ Q \log |g'_\delta|$, $g_\delta(0)=0$ and $g_\delta(i^-) = i$.

Note that the simply connected  open sets $\wt U_\delta$ are determined by $(\phi_\D, \eta_\D)$ and 
$$ M^{\#}_\D\textrm{-almost surely }  (\D \cup \bdy \D) \setminus \wt U_\delta \textrm{ is decreasing as } \delta \to 0 \textrm{ and their intersection  equals } \{i\}.$$ 
Therefore  $M^{\#}_\D$-almost surely $\lim_{\delta \to 0}\diam (\D \setminus \wt U_\delta)=0$ .  
Hence  in our coupling of $M_\delta^\#$ and $M_\D^\#$,  $\mathrm{diam}(\D \backslash \wt U_\delta) = o_\delta(1)$ with probability $1-o_\delta(1)$.
As a basic deterministic fact in complex analysis,  
$\diam (\D \setminus \wt U_\delta)=0$ if and only if the harmonic measure of $\D \backslash \wt U_\delta$ viewed from $0$ in $\wt U_\delta$  tend to $0$ as $\delta\rta 0$.  
Therefore, in our coupling the harmonic measure of $\D \backslash \wt U_\delta$ viewed from $0$ in $\wt U_\delta$ is $o_\delta(1)$ with probability $1-o_\delta(1)$.
By conformal invariance, the same holds for the harmonic measure of $\D\backslash U_\delta$ viewed from $0$ in $U_\delta$.
So $\diam(\D \backslash U_\delta) = o_\delta(1)$ with probability $1-o_\delta(1)$. This gives the second condition for the coupling.

Finally, since $g_\delta(0) = 0, g_\delta(i_-) = i$, and  $\diam(\D \backslash U_\delta), \diam(\D \backslash \wt U_\delta) \to 0$ with probability as $1-o_\delta(1)$, 
standard conformal distortion estimates yield the third condition.
\end{proof}

\begin{proof}[Proof of Proposition~\ref{prop:MD}]
Sending $\delta\to 0$ in Lemma~\ref{lem-couple-disk-bulk} and Proposition~\ref{prop-apprx-zipper}, 
we see that the law of $\phi_\D\circ f+ Q \log| f'|$ under $M_\D^\#$ is the same as that of $\phi$ under $M^\#$, which is  the probability measure proportional to $\int_1^2 \LF_\bbH^{(\gamma, i)} (\ell) \, d\ell$. 
Moreover, the independence of $\phi_\D $ and $\eta_\D$ follows from the asymptotic independence of $\phi$ and $\eta_1$ under $M_\delta^\#$ established in Proposition~\ref{prop-apprx-zipper}.
\end{proof}

\begin{proof}[{Proof of Theorem~\ref{thm-bubble-zipper}}] 
Let $\psi(z)=-z^{-1}$ and $\bub_\bbH(i,\infty)=\{\eta: \psi(\eta)\in \bub_\bbH(i,0)  \}$.  By Proposition~\ref{prop:MD}, 
the law of $(\phi_\D\circ f+ Q \log| f'|, f^{-1}\circ\eta_\D) $ can be written as the probability measure proportional to $\int_1^2 \LF_\bbH^{(\gamma, i)}(\ell) \, d\ell \times \sm^\infty$
where $\sm^\infty$ is the probability measure on $\bub_\bbH(i,\infty)$ describing the marginal law of $f^{-1}\circ\eta_\D$. Therefore for some $C\in (0,\infty)$ we have 
		\eqb\label{eq-weld-infty}
		\int_1^2 \LF_\bbH^{(\gamma, i)}(\ell) \, d\ell \times \sm^\infty = 
		C\int_1^2 \int_{0}^\infty \QD_{1,1}(r)\times \cMtwo(\frac{\gamma^2}2;\ell,r)\, dr\, d\ell,
		\eqe
	in the sense that when a field and curve are sampled from the left hand side of~\eqref{eq-weld-infty}, the law of the two quantum surfaces cut out by the curve is the right hand side of~\eqref{eq-weld-infty}.

	A scaling argument shows that~\eqref{eq-weld-infty} holds when the integration interval $[1,2]$ is replaced by $[e^{\frac\gamma2 c}, 2e^{\frac\gamma2 c}]$ for any $c \in \R$. Indeed, from Lemma~\ref{lem:Girsanov} and the change of coordinates $\mathbf c' = \mathbf c + c$ it is immediate that when $\phi$ is sampled from $\LF_\bbH^{(\gamma, i)}(\ell)$ then the law of $\phi + c$ is $e^{(Q-\frac\gamma2)c}\LF_\bbH^{(\gamma, i)}(e^{\frac\gamma2 c}\ell)$. Similarly, if $(\bbH, \phi, i, 0)$ has law $\QD_{1,1}(r)$ then $(\bbH, \phi + c, i, 0)$ has law $e^{(Q-\gamma)c}\QD_{1,1}(e^{\frac\gamma2 c}r)$, and if $(\bbH, \phi, 0, 1)$ has law  $\cM_{0,2}^\disk(\frac{\gamma^2}2; \ell, r)$ then $(\bbH, \phi+c, 0, 1)$ has law $e^{\gamma c}\cM_{0,2}^\disk(\frac{\gamma^2}2; e^{\frac\gamma2 c}\ell, e^{\frac\gamma2 c}r)$. Therefore, adding $c$ to the fields in~\eqref{eq-weld-infty} and changing variables $(\ell', r') = (e^{\frac\gamma2 c} \ell, e^{\frac\gamma2 c}r)$ gives~\eqref{eq-weld-infty} with $[1,2]$ replaced with $[e^{\frac\gamma2 c}, 2e^{\frac\gamma2 c}]$. Summing over the intervals $[2^n, 2^{n+1}]$ for integer $n$ yields~\eqref{eq-weld-infty} with $[1,2]$ replaced by $(0,\infty)$. 
		
Finally, since $\psi:z \mapsto -z^{-1}$ satisfies $|\psi'(i)|=1$,  by Lemma~\ref{lem-change-coord} $\LF_\bbH^{(\gamma, i)}$ is mapped to itself by the coordinate  change of $\psi$.
Recall the definition of  $\cMtwo(\frac{\gamma^2}2;\cdot,r)$ from \eqref{eq:gamma^22}.
 Reparametrizing~\eqref{eq-weld-infty}  via $\psi$ yields Theorem~\ref{thm-bubble-zipper}.
\end{proof}

\appendix

\section{Brownian motion in cones}\label{sec:BM-cone}
For $\phi \in (0,2\pi)$ let $\cC_\phi = \{ z \in \C : \arg z \in (0, \phi)\}$ be the cone of angle $\phi$. Following \cite[Section 3]{lawler-werner-soup} we will define various measures corresponding to Brownian motion in $\cC_\phi$ conditioned on certain probability zero events via limiting procedures. These constructions are in the same spirit as \cite[Section 3]{lawler-werner-soup}, but we also consider Brownian path measures in a cone with an endpoint at its vertex. For more details on the validity of these constructions see \cite[Section 3]{lawler-werner-soup}.

Let $\cK$ be the collection of  continuous planar curves $Z$ defined on a finite time interval $[0,T_Z]$, where $T_Z$ is the \emph{duration} of the curve. Endow $\cK$ with the metric $d_\cK(Z_1, Z_2) = \inf_\theta \{ \sup_{0< t < T_{Z_1}} |t -\theta(t)| + |Z_1(t) - Z_2(\theta(t))| \}$ with the infimum taken over increasing homeomorphisms $\theta: [0,T_{Z_1}] \to [0, T_{Z_2}]$. Our Brownian path measures will be non-probability measures on $\cK$ equipped with the Borel $\sigma$-algebra associated to $d_\cK$.

Let $\sm(z,w;t)$ denote the measure on $\cK$ such that $|\sm(z,w;t)| = \frac{1}{2\pi t} e^{-|z-w|^2/2t}$ and $\sm(z,w;t)^\#$ is the Brownian bridge from $z$ to $w$ with duration $t$. Let $\sm_{\cC_\phi}(z,w;t)$ be the restriction of $\sm(z,w;t)$ to paths staying in $\cC_\phi$, and let $\sm_{\cC_\phi}(z,w) = \int_0^\infty \sm_{\cC_\phi}(z,w;t)\, dt$.

\begin{lemma}\label{lem-cone-duration-law}
	For each $w \in \cC_\phi$ and $t > 0$, pick any $\psi \in (0, \phi)$ and define the measure $\sm_{\cC_\phi}(w, 0;t) := \lim_{s \to 0} \frac{s^{-\frac\pi\phi}}{\sin (\frac{\pi\psi}\phi)} \sm_{\cC_\phi}(w, se^{i\psi};t)$. This limit exists and does not depend on the choice of $\psi$. Moreover there is a constant $C = C(\phi)$ such that for all $w, t$,
	\eqb \label{eq-cone-duration-law}
	|\sm_{\cC_\phi}(w,0;t)| = C t^{-1-\frac\pi{\phi}} |w|^{\frac\pi\phi} e^{-\frac{|w|^2}{2t}} \sin (\frac\pi\phi \arg w).
	\eqe
\end{lemma}
\begin{proof}
	This result follows from~\cite[Theorem 2]{shimura-cone} except that  theorem is about the Brownian motion in $\cC_\phi$ which starts  at 0 and ends at $w$. 
		Our statement follows from the time reversal symmetry of Brownian path.
\end{proof}

Define $\sm_{\cC_\phi}(w,0) := \int_0^\infty \sm_{\cC_\phi}(w,0;t) \, dt$ and, for $u > 0$, $\sm_{\cC_\phi}(u, 0) := \lim_{\eps \to 0} \eps^{-1} \sm_{\cC_\phi}(u + i \eps, 0)$. 
\begin{corollary}\label{cor-cone-duration-law}
	Suppose $\phi \in (0, 2\pi), \theta \in (0, \phi)$ and $r > 0$. There is a constant $c_{\phi, \theta} \in (0, \infty)$ such that 
	$|\sm_{\cC_\phi}(re^{i\theta},0)|=c_{\phi, \theta} \cdot r^{-\frac\pi\phi}$. Moreover, the law of the duration of a sample from $\sm_{\cC_\phi}(re^{i\theta},0)^\#$ is the inverse gamma distribution with shape $\frac\pi\phi$ and scale $\frac{r^2}2$ (recall~\eqref{eq:inverse-gamma-parameter}). 
	Similarly, for $u>0$, 	$|\sm_{\cC_\phi}(u,0)|=c_\phi u^{-\frac\pi\phi - 1}$. Moreover, the law of the duration of a sample from $\sm_{\cC_\phi}(u,0)^\#$ is the inverse gamma distribution with shape $\frac\pi\phi$ and scale $\frac{u^2}2$.
\end{corollary}
\begin{proof}
	These claims are immediate from Lemma~\ref{lem-cone-duration-law}. 
\end{proof}

Suppose $\theta \in (0, \phi)$, we describe a path decomposition for $\sm_{\cC_\phi}(u,0)^\#$ where we split the path at the first time it hits the ray $e^{i\theta} \R_+$.
For $z \in \cC_\theta$, let $\sm_{\cC_\theta}(z)$ denote the probability measure corresponding to Brownian motion started at $z$ and killed when it exits $\cC_\theta$. For $y >0$, let $E_{y, \eps}$ be the event that Brownian motion exits $\cC_\theta$ on the boundary interval $(ye^{i\theta}, (y+\eps)e^{i\phi})$, and let $\sm_{\cC_\theta}(z,ye^{i\theta}) = \lim_{\eps \to 0} \eps^{-1} \sm_{\cC_\theta}(z)|_{E_{y, \eps}}$. For $x > 0$ define $\sm_{\cC_\theta}(x,ye^{i\theta}) = \lim_{\eps \to 0} \eps^{-1}\sm_{\cC_\theta}(x + \eps i, ye^{i\theta})$. 

\begin{lemma}\label{lem-markov}
	For $0 < \theta < \phi < 2\pi$ and $u>0$,  we have
	\[ \sm_{\cC_\phi}(u, 0) = \int_0^\infty \sm_{\cC_\theta}(u, r e^{i\theta}) \times  \sm_{\cC_\phi}(re^{i\theta}, 0)\, dr\]
	in the sense that when a sample from the right hand side is concatenated to obtain a path from $u$ to $0$, the law of this concatenated path is the left hand side. 
\end{lemma}
\begin{proof}
	Let $\psi = \frac{\phi + \theta}2$. 
	For $\delta, \eps > 0$, by the strong Markov property of Brownian motion we have
	\[\sm_{\cC_\phi}(u + \eps i, \delta e^{i\psi}) = \int_0^\infty \sm_{\cC_\theta}(u+\eps i , re^{i\theta}) \times \sm_{\cC_\phi}(re^{i\theta}, \delta e^{i\psi}) \, dr.  \]
	Multiplying both sides by $\frac{\delta^{-\frac\pi\phi}}{\sin(\frac{\pi \psi}\phi)} \eps^{-1}$ and sending $\delta, \eps \to 0$ yields the assertion. 
\end{proof}

\begin{lemma}\label{lem-poisson-kernel}
	For $0 < \theta < 2\pi$,  there exists a constant $C$ such that 
	\[|\sm_{\cC_\theta}(u, re^{i\theta})| = C\frac{(ur)^{\frac{\pi}{\theta}-1}}{(u^{\frac\pi\theta} + r^{\frac\pi\theta})^2} \quad \text{for }u,r>0.\]
\end{lemma}
\begin{proof}
This is equivalent to  \cite[Lemma C.2]{ahs-disk-welding} after the shear transform $\Lambda$ in Proposition~\ref{lem-mot}.
\end{proof}

\begin{lemma}\label{lem-moment-small}
	Suppose $0< \theta < \phi < 2\pi$. Let $L$ be sampled as in~\eqref{eq-def-AL}. Then for $\eps\in (0, \frac{\pi}{\phi})$ we have $\E[ (\frac L u)^\eps] < 1$. 
\end{lemma}
\begin{proof}
	By~\eqref{eq-markov}, Corollary~\ref{cor-cone-duration-law} and Lemma~\ref{lem-poisson-kernel} the law of $\frac{L}u$ is proportional to \[1_{x>0} |\sm_{\cC_\theta}(u, ux e^{i\theta})| |\sm_{\cC_\phi}(uxe^{i\theta}, 0)| dx \; \propto\;  1_{x>0} x^{\pi/\theta - 1} (1+x^{\pi/\theta})^{-2} \cdot x^{-\frac\pi\phi}\, dx.\]
	Set $f(x) := \frac{x}{\sin x}$. 	When $p \in (-1, \frac{2\pi}\theta - 1)$, we have
	\[\int_0^\infty \frac{x^p}{(1+x^{\pi/\theta})^2}\, dx = \frac\theta\pi \frac{\pi - \theta(p+1)}{\sin(\pi - \theta(p+1))} = \frac\theta\pi  f(\pi - \theta(p+1)). \]
  Therefore
	\[\E\left[\left(\frac{L}{u}\right)^\eps\right] = \int_0^\infty \frac{x^{\frac\pi\theta -1- \frac\pi\phi + \eps}}{(1+x^{\pi/\theta})^2}\, dx \bigg/ \int_0^\infty \frac{x^{\frac\pi\theta -1 - \frac\pi\phi}}{(1+x^{\pi/\theta})^2}\, dx = \frac{f(\frac{\pi\theta}\phi - \theta \eps)}{f(\frac{\pi\theta}\phi )}. \]
	Since $f(x)$ is increasing on $(0, \pi)$ and $\frac{\pi \theta}\phi \in (0, \pi)$, we obtain the lemma.
\end{proof}